\documentclass[12pt]{amsart}
\pdfoutput=1

\usepackage[paperheight=11in, paperwidth=8.5in, left=1in, top=1in, right=1in, bottom=1in]{geometry}
\usepackage[linktocpage=true, linktoc=all, colorlinks=true, linkcolor=black, citecolor=black, filecolor=black, urlcolor=black, pagebackref=false, pdfstartpage={1}, pdfstartview={FitH}, pdftitle={}, pdfauthor={}, pdfsubject={}, pdfcreator={}, pdfproducer={}, pdfkeywords={}, bookmarksdepth=99]{hyperref}
\usepackage{accents}
\usepackage{afterpage}
\usepackage{amsfonts}
\usepackage{amsmath}
\allowdisplaybreaks[4]
\usepackage{amssymb}
\usepackage{amsthm}
\usepackage{array}
\usepackage{arydshln}
\usepackage{bbm}
\usepackage[justification=centering]{caption}
\usepackage[capitalize,nameinlink,noabbrev,compress]{cleveref}

\usepackage{enumerate}
\usepackage{enumitem}
\usepackage{float}
\restylefloat{figure}
\usepackage{graphicx}
\usepackage{mathrsfs}
\usepackage[framemethod=tikz,ntheorem,xcolor]{mdframed}
\usepackage{multirow}
\usepackage{parcolumns}
\usepackage{scalerel}
\usepackage{tikz}\pdfpageattr{/Group <</S /Transparency /I true /CS /DeviceRGB>>}
\usetikzlibrary{arrows, calc, cd, decorations.markings, intersections, matrix, positioning, through}
\tikzset{commutative diagrams/.cd, every label/.append style = {font = \normalsize}}
\usepackage[all]{xy}
\usepackage[aligntableaux=center]{ytableau}

\tikzset{partition/.style={fill,circle,inner sep=1pt}}
\usetikzlibrary{positioning}
\usetikzlibrary{decorations.pathreplacing}
\usetikzlibrary{matrix,arrows}
\usetikzlibrary{calc}
\tikzset{partition/.style={fill,circle,inner sep=1pt},part/.style={baseline=0,scale=0.5,bend left=45},partlabel/.style={below}}
\usetikzlibrary{shapes,arrows}
\usetikzlibrary{snakes}
\tikzstyle{pnt}=[draw,ellipse,fill,inner sep=1pt]
\tikzstyle{opnt}=[draw,ellipse,inner sep=1pt]
\tikzstyle{opnt}=[ ]
\tikzstyle{pntt}=[draw,ellipse,fill,inner sep=0.5pt]
\tikzstyle{point}=[draw,ellipse,fill,inner sep=2pt]

\usepackage{enumerate}
\usepackage{enumitem}
\usepackage{float}
\restylefloat{figure}
\usepackage{graphicx}
\usepackage{mathrsfs}
\usepackage[framemethod=tikz,ntheorem,xcolor]{mdframed}
\usepackage{multirow}
\usepackage{multicol}
\usepackage{parcolumns}
\usepackage{scalerel}
\usepackage{tikz}\pdfpageattr{/Group <</S /Transparency /I true /CS /DeviceRGB>>}
\usetikzlibrary{arrows, calc, cd, decorations.markings, intersections, matrix, positioning, through}
\tikzset{commutative diagrams/.cd, every label/.append style = {font = \normalsize}}
\usepackage[all]{xy}
\usepackage[aligntableaux=center]{ytableau}

\numberwithin{equation}{section}
\newtheorem{thm}{Theorem}
\AfterEndEnvironment{thm}{\noindent\ignorespaces}
\numberwithin{thm}{section}

\AfterEndEnvironment{conj}{\noindent\ignorespaces}
\newtheorem{cor}[thm]{Corollary}
\AfterEndEnvironment{cor}{\noindent\ignorespaces}
\newtheorem{lem}[thm]{Lemma}
\AfterEndEnvironment{lem}{\noindent\ignorespaces}
\newtheorem{prob}[thm]{Problem}
\AfterEndEnvironment{prob}{\noindent\ignorespaces}
\newtheorem{prop}[thm]{Proposition}
\AfterEndEnvironment{prop}{\noindent\ignorespaces}

\theoremstyle{definition}
\newtheorem{defn}[thm]{Definition}
\AfterEndEnvironment{defn}{\noindent\ignorespaces}
\newtheorem{eg_no_qed}[thm]{Example}
\newenvironment{eg}[1][]{\begin{eg_no_qed}[#1]\pushQED{\qed}}{\popQED\end{eg_no_qed}}
\AfterEndEnvironment{eg}{\noindent\ignorespaces}
\newtheorem{rmk_no_qed}[thm]{Remark}
\newenvironment{rmk}[1][]{\begin{rmk_no_qed}[#1]\pushQED{\qed}}{\popQED\end{rmk_no_qed}}
\AfterEndEnvironment{rmk}{\noindent\ignorespaces}

\theoremstyle{remark}

\AfterEndEnvironment{claim}{\noindent\ignorespaces}
\newtheorem*{claimpf_no_qed}{Proof of Claim}

\AfterEndEnvironment{claimpf}{\noindent\ignorespaces}

\crefname{thm}{Theorem}{Theorems}
\crefname{conj}{Conjecture}{Conjectures}
\crefname{cor}{Corollary}{Corollaries}
\crefname{lem}{Lemma}{Lemmas}
\crefname{prob}{Problem}{Problems}
\crefname{prop}{Proposition}{Propositions}
\crefname{defn}{Definition}{Definitions}
\crefname{eg}{Example}{Examples}
\crefname{eg_no_qed}{Example}{Examples}
\crefname{rmk}{Remark}{Remarks}

\newcommand{\boldit}[1]{\textbf{\textit{#1}}}

\renewcommand{\eqref}[1]{\hyperref[#1]{\textup{(\ref*{#1})}}}

\DeclareMathOperator{\Fl}{Fl}

\DeclareMathOperator{\GL}{GL}
\newcommand{\gl}{\mathfrak{gl}}

\DeclareMathOperator{\inv}{inv}

\DeclareMathOperator{\maj}{maj}

\DeclareMathOperator{\op}{op}

\newcommand{\qbinom}[2]{\begin{bmatrix}#1 \\ #2\end{bmatrix}_q}

\newcommand{\qbinomsmall}[2]{\scalebox{0.7}{$\begin{bmatrix}{#1} \\ {#2}\end{bmatrix}$}_q}

\newcommand{\qfac}[1]{[#1]_q!}
\newcommand{\qn}[1]{[#1]_q}
\newcommand{\qinvbinom}[2]{\begin{bmatrix}#1 \\ #2\end{bmatrix}_{1/q}}
\newcommand{\qinvbinomsmall}[2]{\scalebox{0.7}{$\begin{bmatrix}{#1} \\ {#2}\end{bmatrix}$}_{1/q}}

\newcommand{\spn}[1]{\text{span}(#1)}

\newcommand{\transpose}[1]{#1^t}

\newcommand{\codej}[2]{\mathsf{c}_{#2}(#1)}
\newcommand{\code}[1]{\mathsf{c}(#1)}
\DeclareMathOperator{\crop}{\mathsf{crop}}

\newcommand{\mincoset}[2]{#1^{#2}}
\newcommand{\mincosetsub}[2]{#1_{#2}}
\newcommand{\n}[1]{\mathsf{n}(#1)}

\newcommand{\newrow}[2]{\mathsf{b}_{#1}(#2)}
\newcommand{\newrownoarg}[1]{\mathsf{b}_{#1}}
\newcommand{\ogfbiggest}[1]{\Phi_{\Rbiggest{#1}}}
\newcommand{\ogfprime}[1]{\Phi_{\Rprime{#1}}}

\newcommand{\rowjword}[2]{\mathsf{r}_{#2}(#1)}
\newcommand{\rowword}[1]{\mathsf{r}(#1)}
\newcommand{\Rbiggest}[1]{\mathsf{W}_{#1}}
\newcommand{\Rcell}[2]{\mathring{R}_{#1,#2}}
\newcommand{\Rprime}[1]{\mathsf{P}_{#1}}
\DeclareMathOperator{\RS}{RS}
\newcommand{\Rtabs}[1]{\mathcal{R}(#1)}
\newcommand{\Rtp}[2]{R_{#1,#2}^{>0}}
\newcommand{\Rvar}[2]{R_{#1,#2}}
\newcommand{\Rwords}[1]{\mathsf{W}(#1)}
\newcommand{\Scell}[1]{\mathring{X}_{#1}}
\newcommand{\Scellopp}[1]{\mathring{X}_{#1}^{\op}}
\newcommand{\SF}[1]{\mathcal{B}_{#1}}
\newcommand{\SFcell}[1]{\mathring{\mathcal{B}}_{#1}}
\newcommand{\SFtnn}[1]{\SF{#1}^{\ge 0}}
\newcommand{\Svar}[1]{X_{#1}}
\newcommand{\Svaropp}[1]{X_{#1}^{\op}}
\DeclareMathOperator{\sing}{sing}
\DeclareMathOperator{\sumone}{\vartheta}
\newcommand{\ytableausmall}[1]{\ytableausetup{smalltableaux}\, \begin{ytableau}#1\end{ytableau}\, \ytableausetup{nosmalltableaux}}

\newcommand{\C}{\mathbb{C}}
\newcommand{\SYT}{\mathrm{SYT}}

\newcommand{\fu}{\mathfrak{u}}

\title[Richardson tableaux and Springer fibers]{Richardson tableaux and components of Springer fibers equal to Richardson varieties}
\author{Steven N. Karp}
\address{Department of Mathematics, University of Notre Dame}
\email{\href{mailto:skarp2@nd.edu}{skarp2@nd.edu}}
\author{Martha E. Precup}
\address{Department of Mathematics, Washington University in St.\ Louis}
\email{\href{mailto:martha.precup@wustl.edu}{martha.precup@wustl.edu}}

\begin{document}

\begin{abstract}
Motivated by the study of Springer fibers and their totally nonnegative counterparts, we define a new subset of standard tableaux called \emph{Richardson tableaux}. We characterize Richardson tableaux combinatorially using evacuation as well as in terms of a pair of associated reading words. We also characterize Richardson tableaux geometrically, proving that a tableau is Richardson if and only if the corresponding component of a Springer fiber is a Richardson variety, which in turn holds if and only if its positive part is a top-dimensional cell of the totally nonnegative Springer fiber studied by Lusztig (2021). We prove that each such component is smooth by leveraging a combinatorial description of the corresponding pair of reading words, generalizing a result of Graham--Zierau (2011). Another application is that the cohomology classes of these components can be computed in the Schubert basis using Schubert calculus. Finally, we show that the enumeration of Richardson tableaux is surprisingly elegant: the number of Richardson tableaux of fixed partition shape is a product of binomial coefficients, and the number of Richardson tableaux of size $n$ is the $n$th Motzkin number. As a result, we obtain a new refinement for the Motzkin numbers, as well as a formula for the number of top-dimensional cells in the totally nonnegative Springer fiber.
\end{abstract}

\maketitle
\setcounter{tocdepth}{1}
\tableofcontents


\section{Introduction}\label{sec_intro}

\noindent The purpose of this paper is to explore the connection between Springer fibers and Richardson varieties, leading to interesting new geometry and combinatorics.
Springer fibers are subvarieties of the flag variety $\Fl_n(\C)$ of $\C^n$ which are fibers of Springer's resolution of the nilpotent cone \cite{springer69}, and are indexed by partitions $\lambda = (\lambda_1, \lambda_2, \dots, \lambda_\ell)$ of size $n$. Explicitly, if $N$ is a nilpotent matrix of Jordan type $\lambda$ which is in Jordan canonical form with block sizes weakly decreasing, then we define the Springer fiber to be
\begin{align}\label{intro_SF}
\SF{\lambda} := \{ (0 = F_0 \subset F_1 \subset \cdots \subset F_n = \C^n)\in \Fl_n(\C) \mid N(F_j) \subset F_{j-1} \text{ for all } j>0 \}.
\end{align}
Spaltenstein \cite{spaltenstein76} proved that the irreducible components $\SF{\sigma}$ of $\SF{\lambda}$ are labeled by standard Young tableaux $\sigma$ of shape $\lambda$, and that every $\SF{\sigma}$ has dimension $n(\lambda):=\sum_{i=1}^{\ell} (i-1)\lambda_i$. Springer showed that the cohomology of each Springer fiber carries an action of the symmetric group $S_n$ \cite{springer78}. In this way, Springer fibers naturally embody the interaction between tableau combinatorics, representation theory, and geometry. Although they play a key role in a seminal construction of geometric representation theory, basic questions about their geometry and the geometry of their components remain open. While Springer fibers can be defined for more general algebraic groups $G$, we focus here on the case $G = \GL_n(\C)$.

On the other hand, Richardson varieties $\Rvar{v}{w}$ are irreducible subvarieties of $\Fl_n(\C)$ indexed by pairs of permutations $v,w\in S_n$ such that $v \le w$ in Bruhat order, obtained by intersecting a Schubert variety with an opposite Schubert variety. They are central objects in Kazhdan--Lusztig theory \cite{kazhdan_lusztig79}. For example, the Kazhdan--Lusztig $R$-polynomial is given by the point count of $\Rvar{v}{w}$ over the finite field $\mathbb{F}_q$ \cite{kazhdan_lusztig79,deodhar85a}. Richardson varieties also play an essential role in Lusztig's theory of total positivity \cite{lusztig94}: the totally nonnegative part $\Fl_n^{\ge 0}$ of the flag variety has a regular cell decomposition (see \eqref{decomposition_flag} below) whose cells are the totally positive Richardson cells $\Rtp{v}{w}$. Finally, they are also important in Schubert calculus, since the cohomology class of $\Rvar{v}{w}$ is a product of Schubert classes.

Recently, Lusztig \cite{lusztig21} studied the totally nonnegative Springer fiber $\SF{\lambda}^{\geq 0} := \SF{\lambda}\cap\Fl_n^{\ge 0}$ and proved that, unlike the Springer fiber $\SF{\lambda}$ in $\Fl_n(\C^n)$, its geometric structure is relatively simple. In particular, $\SFtnn{\lambda}$ has a cell decomposition whose cells are of the form $\Rtp{v}{w}$ for pairs $(v,w)$ satisfying a certain Bruhat-theoretic condition. A straightforward consequence of Lusztig's work (see \cref{geometry_positivity} below) is that the top-dimensional cells of $\SF{\lambda}^{\geq 0}$ correspond to the Richardson varieties $\Rvar{v}{w}$ which are irreducible components of $\SF{\lambda}$. As the components of $\SF{\lambda}$ are indexed by standard tableaux of shape $\lambda$, we obtain the following natural question:
\begin{align}\label{intro_question}
\parbox{4.8in}{\centering\emph{For which standard tableaux $\sigma$ of shape $\lambda$ is the associated irreducible component $\SF{\sigma}$ of $\SF{\lambda}$ equal to a Richardson variety $\Rvar{v}{w}$?}}
\end{align}
We give a complete answer to this question, in two steps. First, for every standard tableau $\sigma$ we show that the corresponding irreducible component $\SF{\sigma}$ of the Springer fiber $\SF{\lambda}$ is contained a unique Richardson variety $\Rvar{v_\sigma}{w_\sigma}$ which is minimal with respect to containment. In particular, $\SF{\sigma}$ is a Richardson variety if and only if it is equal to $\Rvar{v_\sigma}{w_\sigma}$. The second step is to characterize all such tableaux $\sigma$, which we do by introducing a new subset of standard tableaux called \emph{Richardson tableaux}. 

To state our first main result, we introduce two permutations $v_\sigma, w_\sigma \in S_n$ associated to a standard tableau $\sigma$ of size $n$. Let $\sigma^\vee$ denote the evacuation tableau of $\sigma$, and let $w_0\in S_n$ denote the longest permutation. Then $v_\sigma^{-1}$ is obtained in one-line notation by reading the entries of $\sigma^\vee$ row by row from top to bottom, and within each row from left to right. We define $w_0w_\sigma^{-1}w_0$ to be obtained similarly by reading the entries of $\sigma$, except that we read the rows from bottom to top (rather than top to bottom). The permutation $w_\sigma$ appeared in work of Pagnon and Ressayre \cite{pagnon_ressayre06}, who showed that the intersection of the Schubert cell for $w_\sigma$ with $\SF{\lambda}$ is a dense subset of the component $\SF{\sigma}$.
\begin{eg}\label{eg_permutations_intro}
Let $n = 8$. Consider the following standard tableau $\sigma$ of shape $\lambda = (4,2,2)$, along with its evacuation $\sigma^{\vee}$:
\[
\sigma = \;\begin{ytableau}
1 & 3 & 4 & 6 \\
2 & 7 \\
5 & 8
\end{ytableau}\;
\quad\text{ and }\quad
\sigma^{\vee} = 
\;\begin{ytableau}
1 & 4 & 6 & 7 \\
2 & 5 \\
3 & 8
\end{ytableau}\;.
\]
By definition, 
\[
v_\sigma^{-1} = 14672538, \quad w_0w_{\sigma}^{-1}w_0 = 58271346, \quad \text{ and } \quad w_0 = 87654321.
\]
Thus $v_{\sigma} = 15726348$ and $w_{\sigma} = 75182364$.
\end{eg}

The pair $(v_\sigma, w_\sigma)$ indexes the `Richardson envelope' of the component $\SF{\sigma}$: 
\begin{thm}[\cref{main_geometric} below]\label{intro.goemetric}
Let $\sigma$ be a standard tableau of shape $\lambda$. Then $\Rvar{v_\sigma}{w_\sigma}$ is the unique minimal Richardson variety containing the irreducible component $\SF{\sigma}$ of $\SF{\lambda}$.
\end{thm}

It is not clear, a priori, that there exists a unique minimal Richardson variety containing $\SF{\sigma}$, but this is part of the content of \cref{intro.goemetric}.

In order to answer question \eqref{intro_question}, we wish to know when $\SF{\sigma}$ is equal to $\Rvar{v_\sigma}{w_\sigma}$. This is encapsulated in our titular definition:
\begin{defn}\label{defn_richardson}
Let $\sigma$ be a standard tableau of size $n$. For all $1 \le j \le n$, let $\rowjword{\sigma}{j}$ denote the row number in which the entry $j$ appears in $\sigma$, and let $\sigma[j-1]$ be the standard tableau obtained from $\sigma$ by deleting the boxes containing $j, j+1, \ldots, n$. We call $\sigma$ a \boldit{Richardson tableau} if for all $1 \le j \le n$ with $\rowjword{\sigma}{j} > 1$, the last (and therefore largest) entry of $\sigma[j-1]$ in row $\rowjword{\sigma}{j}-1$ is greater than every entry in rows $i \ge \rowjword{\sigma}{j}$.
\end{defn}

\begin{eg}\label{eg_richardson_tableau_intro}
Consider the following two standard tableaux of shape $\lambda= (4,2,2)$:
\[
\sigma = \;\begin{ytableau}
1 & 3 & 4 & 6 \\
2 & 7 \\
5 & 8
\end{ytableau}\;
\quad \text{ and } \quad
\tau = \;\begin{ytableau}
1 & 2 & 4 & 7 \\
3 & 6 \\
5 & 8
\end{ytableau}\;.
\]
We can verify that $\sigma$ is Richardson. On the other hand, $\tau$ is not Richardson. To see this, we can take $j=6$: then $\tau[5] = \ytableausmall{1 & 2 & 4\\ 3\\ 5}$, and the largest number in row $\rowjword{\tau}{6} - 1=1$, which is $4$, is less than $5$. Of the $56$ standard tableaux of shape $\lambda$, exactly $15$ are Richardson.
\end{eg}

We prove various combinatorial and geometric characterizations of Richardson tableaux:
\begin{thm}\label{intro.thm}
Let $\sigma$ be a standard tableau of shape $\lambda$ and size $n$. Then the following statements are equivalent.
\begin{enumerate}[label=(\roman*), leftmargin=*, itemsep=2pt]
\item\label{main1} The standard tableau $\sigma$ is a Richardson tableau.
\item\label{main2} The standard tableau $\sigma$ is a concatenation of prime Richardson tableau.
\item\label{main2.5} The standard tableau obtained from $\sigma$ by deleting its first row (and renumbering the entries accordingly) is Richardson, and for every entry $j$ in the second row of $\sigma$ the entry $j-1$ appears in the first row.
\item\label{main3} The evacuation $\sigma^{\vee}$ of $\sigma$ is a Richardson tableau.
\item\label{main4} Every evacuation slide of $\sigma$ is an $L$-slide, that is, the path traced through the tableau by each evacuation slide is `L'-shaped.
\item\label{main5} We have $\ell(w_\sigma)-\ell(v_\sigma)=n(\lambda)$.
\item\label{main6} The irreducible component $\SF{\sigma}$ of $\SF{\lambda}$ is a Richardson variety.
\item\label{main7} We have $\SF{\sigma} = \Rvar{v_\sigma}{w_\sigma}$.
\item\label{main8} We have $s_j v_\sigma \not\leq w_{\sigma}$ in Bruhat order for all $j\in [n-1]\setminus \{\lambda_1, \lambda_1+\lambda_2, \ldots\}$.
\end{enumerate}
\end{thm}

In particular, the equivalence of statements \ref{main1} and \ref{main6} answers question \eqref{intro_question}. The equivalences in \cref{intro.thm} are proved below as follows:
\ref{main1} $\Leftrightarrow$ \ref{main2} is \cref{prime_concatenation},
\ref{main1} $\Leftrightarrow$ \ref{main2.5} is \cref{crop_tableau},
\ref{main1} $\Leftrightarrow$ \ref{main3} is \cref{evacuation_closed},
\ref{main1} $\Leftrightarrow$ \ref{main4} is \cref{thm-Lslides},
\ref{main1} $\Leftrightarrow$ \ref{main5} is \cref{thm-lengths},
\ref{main1} $\Leftrightarrow$ \ref{main6} $\Leftrightarrow$ \ref{main7} is \cref{main}, and
\ref{main1} $\Leftrightarrow$ \ref{main8} is \cref{richardson_Z}.

\cref{intro.thm} has several interesting combinatorial and geometric consequences, which we now discuss. Combinatorially, considering the concatenation of Richardson tableaux as in \cref{intro.thm}\ref{main2} naturally leads us to study \emph{Richardson words}. We leverage this approach to find (see \cref{prime_recursion}) a recursion describing the set of all prime Richardson words, namely, Richardson words which cannot be broken down further into smaller Richardson words.  This recursion yields the following enumerative result:
\begin{thm}[\cref{motzkin_richardson} below]\label{intro.Motzkin}
For all $n\ge 0$, the number of Richardson tableaux of size $n$ equals the Motzkin number $M_n$. 
\end{thm}

See \cref{figure_motzkin} for an illustration of \cref{intro.Motzkin}. 
\begin{figure}[h]
\begin{center}\setlength{\tabcolsep}{2pt}
\begin{tabular}{ccccccccc}
$\begin{tikzpicture}[baseline=(current bounding box.center),scale=0.28]
\pgfmathsetmacro{\n}{4};
\useasboundingbox(-0.8,0)rectangle(\n+0.8,\n/2);
\coordinate(u)at(1,1);
\coordinate(d)at(1,-1);
\coordinate(h)at(1,0);
\tikzstyle{out1}=[inner sep=0,minimum size=1.2mm,circle,draw=black,fill=black]
\coordinate(p0)at(0,0){};
\coordinate(p1)at($(p0)+(u)$){};
\coordinate(p2)at($(p1)+(u)$){};
\coordinate(p3)at($(p2)+(d)$){};
\coordinate(p4)at($(p3)+(d)$){};
\draw[thick](p0.center)--(p1.center)--(p2.center)--(p3.center)--(p4.center);
\foreach \x in {1,...,\n}{
\pgfmathsetmacro{\y}{int(\x-1)};
\draw[thick](p\y.center)--(p\x.center);
\node[out1]at(p\y){};}
\node[out1]at(p\n){};
\end{tikzpicture}$
&
$\begin{tikzpicture}[baseline=(current bounding box.center),scale=0.28]
\pgfmathsetmacro{\n}{4};
\useasboundingbox(-0.8,0)rectangle(\n+0.8,\n/2);
\coordinate(u)at(1,1);
\coordinate(d)at(1,-1);
\coordinate(h)at(1,0);
\tikzstyle{out1}=[inner sep=0,minimum size=1.2mm,circle,draw=black,fill=black]
\coordinate(p0)at(0,0){};
\coordinate(p1)at($(p0)+(u)$){};
\coordinate(p2)at($(p1)+(h)$){};
\coordinate(p3)at($(p2)+(h)$){};
\coordinate(p4)at($(p3)+(d)$){};
\draw[thick](p0.center)--(p1.center)--(p2.center)--(p3.center)--(p4.center);
\foreach \x in {1,...,\n}{
\pgfmathsetmacro{\y}{int(\x-1)};
\draw[thick](p\y.center)--(p\x.center);
\node[out1]at(p\y){};}
\node[out1]at(p\n){};
\end{tikzpicture}$
&
$\begin{tikzpicture}[baseline=(current bounding box.center),scale=0.28]
\pgfmathsetmacro{\n}{4};
\useasboundingbox(-0.8,0)rectangle(\n+0.8,\n/2);
\coordinate(u)at(1,1);
\coordinate(d)at(1,-1);
\coordinate(h)at(1,0);
\tikzstyle{out1}=[inner sep=0,minimum size=1.2mm,circle,draw=black,fill=black]
\coordinate(p0)at(0,0){};
\coordinate(p1)at($(p0)+(u)$){};
\coordinate(p2)at($(p1)+(h)$){};
\coordinate(p3)at($(p2)+(d)$){};
\coordinate(p4)at($(p3)+(h)$){};
\draw[thick](p0.center)--(p1.center)--(p2.center)--(p3.center)--(p4.center);
\foreach \x in {1,...,\n}{
\pgfmathsetmacro{\y}{int(\x-1)};
\draw[thick](p\y.center)--(p\x.center);
\node[out1]at(p\y){};}
\node[out1]at(p\n){};
\end{tikzpicture}$
&
$\begin{tikzpicture}[baseline=(current bounding box.center),scale=0.28]
\pgfmathsetmacro{\n}{4};
\useasboundingbox(-0.8,0)rectangle(\n+0.8,\n/2);
\coordinate(u)at(1,1);
\coordinate(d)at(1,-1);
\coordinate(h)at(1,0);
\tikzstyle{out1}=[inner sep=0,minimum size=1.2mm,circle,draw=black,fill=black]
\coordinate(p0)at(0,0){};
\coordinate(p1)at($(p0)+(u)$){};
\coordinate(p2)at($(p1)+(d)$){};
\coordinate(p3)at($(p2)+(u)$){};
\coordinate(p4)at($(p3)+(d)$){};
\draw[thick](p0.center)--(p1.center)--(p2.center)--(p3.center)--(p4.center);
\foreach \x in {1,...,\n}{
\pgfmathsetmacro{\y}{int(\x-1)};
\draw[thick](p\y.center)--(p\x.center);
\node[out1]at(p\y){};}
\node[out1]at(p\n){};
\end{tikzpicture}$
&
$\begin{tikzpicture}[baseline=(current bounding box.center),scale=0.28]
\pgfmathsetmacro{\n}{4};
\useasboundingbox(-0.8,0)rectangle(\n+0.8,\n/2);
\coordinate(u)at(1,1);
\coordinate(d)at(1,-1);
\coordinate(h)at(1,0);
\tikzstyle{out1}=[inner sep=0,minimum size=1.2mm,circle,draw=black,fill=black]
\coordinate(p0)at(0,0){};
\coordinate(p1)at($(p0)+(u)$){};
\coordinate(p2)at($(p1)+(d)$){};
\coordinate(p3)at($(p2)+(h)$){};
\coordinate(p4)at($(p3)+(h)$){};
\draw[thick](p0.center)--(p1.center)--(p2.center)--(p3.center)--(p4.center);
\foreach \x in {1,...,\n}{
\pgfmathsetmacro{\y}{int(\x-1)};
\draw[thick](p\y.center)--(p\x.center);
\node[out1]at(p\y){};}
\node[out1]at(p\n){};
\end{tikzpicture}$
&
$\begin{tikzpicture}[baseline=(current bounding box.center),scale=0.28]
\pgfmathsetmacro{\n}{4};
\useasboundingbox(-0.8,0)rectangle(\n+0.8,\n/2);
\coordinate(u)at(1,1);
\coordinate(d)at(1,-1);
\coordinate(h)at(1,0);
\tikzstyle{out1}=[inner sep=0,minimum size=1.2mm,circle,draw=black,fill=black]
\coordinate(p0)at(0,0){};
\coordinate(p1)at($(p0)+(h)$){};
\coordinate(p2)at($(p1)+(u)$){};
\coordinate(p3)at($(p2)+(h)$){};
\coordinate(p4)at($(p3)+(d)$){};
\draw[thick](p0.center)--(p1.center)--(p2.center)--(p3.center)--(p4.center);
\foreach \x in {1,...,\n}{
\pgfmathsetmacro{\y}{int(\x-1)};
\draw[thick](p\y.center)--(p\x.center);
\node[out1]at(p\y){};}
\node[out1]at(p\n){};
\end{tikzpicture}$
&
$\begin{tikzpicture}[baseline=(current bounding box.center),scale=0.28]
\pgfmathsetmacro{\n}{4};
\useasboundingbox(-0.8,0)rectangle(\n+0.8,\n/2);
\coordinate(u)at(1,1);
\coordinate(d)at(1,-1);
\coordinate(h)at(1,0);
\tikzstyle{out1}=[inner sep=0,minimum size=1.2mm,circle,draw=black,fill=black]
\coordinate(p0)at(0,0){};
\coordinate(p1)at($(p0)+(h)$){};
\coordinate(p2)at($(p1)+(u)$){};
\coordinate(p3)at($(p2)+(d)$){};
\coordinate(p4)at($(p3)+(h)$){};
\draw[thick](p0.center)--(p1.center)--(p2.center)--(p3.center)--(p4.center);
\foreach \x in {1,...,\n}{
\pgfmathsetmacro{\y}{int(\x-1)};
\draw[thick](p\y.center)--(p\x.center);
\node[out1]at(p\y){};}
\node[out1]at(p\n){};
\end{tikzpicture}$
&
$\begin{tikzpicture}[baseline=(current bounding box.center),scale=0.28]
\pgfmathsetmacro{\n}{4};
\useasboundingbox(-0.8,0)rectangle(\n+0.8,\n/2);
\coordinate(u)at(1,1);
\coordinate(d)at(1,-1);
\coordinate(h)at(1,0);
\tikzstyle{out1}=[inner sep=0,minimum size=1.2mm,circle,draw=black,fill=black]
\coordinate(p0)at(0,0){};
\coordinate(p1)at($(p0)+(h)$){};
\coordinate(p2)at($(p1)+(h)$){};
\coordinate(p3)at($(p2)+(u)$){};
\coordinate(p4)at($(p3)+(d)$){};
\draw[thick](p0.center)--(p1.center)--(p2.center)--(p3.center)--(p4.center);
\foreach \x in {1,...,\n}{
\pgfmathsetmacro{\y}{int(\x-1)};
\draw[thick](p\y.center)--(p\x.center);
\node[out1]at(p\y){};}
\node[out1]at(p\n){};
\end{tikzpicture}$
&
$\begin{tikzpicture}[baseline=(current bounding box.center),scale=0.28]
\pgfmathsetmacro{\n}{4};
\useasboundingbox(-0.8,0)rectangle(\n+0.8,\n/2);
\coordinate(u)at(1,1);
\coordinate(d)at(1,-1);
\coordinate(h)at(1,0);
\tikzstyle{out1}=[inner sep=0,minimum size=1.2mm,circle,draw=black,fill=black]
\coordinate(p0)at(0,0){};
\coordinate(p1)at($(p0)+(h)$){};
\coordinate(p2)at($(p1)+(h)$){};
\coordinate(p3)at($(p2)+(h)$){};
\coordinate(p4)at($(p3)+(h)$){};
\draw[thick](p0.center)--(p1.center)--(p2.center)--(p3.center)--(p4.center);
\foreach \x in {1,...,\n}{
\pgfmathsetmacro{\y}{int(\x-1)};
\draw[thick](p\y.center)--(p\x.center);
\node[out1]at(p\y){};}
\node[out1]at(p\n){};
\end{tikzpicture}$ \\[20pt]
\scalebox{0.6}{$\begin{ytableau}
1 & 2 & 3 & 4 \\
\none \\
\none \\
\none
\end{ytableau}$}
&
\scalebox{0.6}{$\begin{ytableau}
1 & 2 & 3 \\
4 \\
\none \\
\none
\end{ytableau}$}
&
\scalebox{0.6}{$\begin{ytableau}
1 & 2 & 4 \\
3 \\
\none \\
\none
\end{ytableau}$}
&
\scalebox{0.6}{$\begin{ytableau}
1 & 3 & 4 \\
2 \\
\none \\
\none
\end{ytableau}$}
&
\scalebox{0.6}{$\begin{ytableau}
1 & 3 \\
2 & 4 \\
\none \\
\none
\end{ytableau}$}
&
\scalebox{0.6}{$\begin{ytableau}
1 & 2 \\
3 \\
4 \\
\none
\end{ytableau}$}
&
\scalebox{0.6}{$\begin{ytableau}
1 & 3 \\
2 \\
4 \\
\none
\end{ytableau}$}
&
\scalebox{0.6}{$\begin{ytableau}
1 & 4 \\
2 \\
3 \\
\none
\end{ytableau}$}
&
\scalebox{0.6}{$\begin{ytableau}
1 \\
2 \\
3 \\
4
\end{ytableau}$}
\end{tabular}\vspace*{-6pt}
\caption{An illustration of \cref{intro.Motzkin} for $n=4$ (where $M_4 = 9$). Top row: Motzkin paths with $4$ steps. Bottom row: Richardson tableaux of size $4$.}
\label{figure_motzkin}
\end{center}
\end{figure}
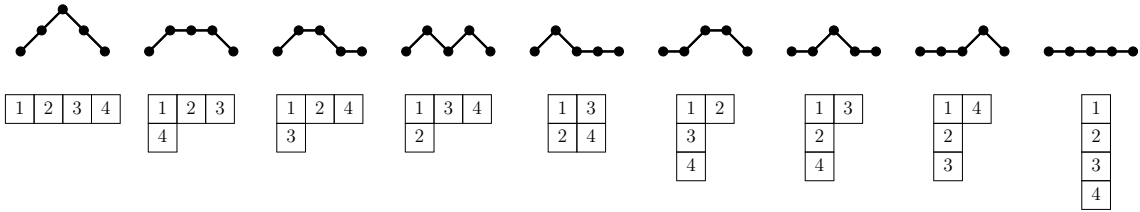

We also find a recursive description for the set of all Richardson words with a fixed shape in \cref{richardson_chop}, which leads to the following formula:
\begin{thm}[\cref{q_count} below]\label{intro.count}
The number of Richardson tableaux of fixed partition shape $\lambda = (\lambda_1, \lambda_2,\ldots, \lambda_\ell)$ is
\begin{eqnarray}\label{count.intro}
\binom{\lambda_{\ell-1}}{\lambda_\ell}\binom{\lambda_{\ell-2} + \lambda_\ell}{\lambda_{\ell-1} + \lambda_\ell}\binom{\lambda_{\ell-3} + \lambda_{\ell-1} + \lambda_\ell}{\lambda_{\ell-2} + \lambda_{\ell-1} + \lambda_\ell}\cdots\binom{\lambda_1 + \lambda_3 + \lambda_4 + \cdots + \lambda_\ell}{\lambda_2 + \lambda_3 + \lambda_4 + \cdots + \lambda_\ell}.
\end{eqnarray}
\end{thm}

We also prove a $q$-analogue of the formula~\eqref{count.intro} using the major index statistic on standard tableaux; see \cref{q_count}. Note that summing \eqref{count.intro} over all partitions of a fixed size $n$ gives a formula for the Motzkin number $M_n$. As far as we know, this is the only known refinement of the Motzkin numbers indexed by partitions. It is surprising that answering the geometric question \eqref{intro_question} leads to such elegant enumerative formulas as in \cref{intro.Motzkin,count.intro}. We see this as evidence that Richardson tableaux $\sigma$ and the associated components $\SF{\sigma}$ are worthy of further study.

Combining \cref{intro.count} with \cref{intro.thm}\ref{main6} yields the following enumerative result on the cellular structure of $\SFtnn{\lambda}$:
\begin{cor}[\cref{tp_main} below]\label{tp_main_intro}
The totally nonnegative Springer fiber $\SF{\lambda}^{\geq 0}$ has dimension $n(\lambda)$, and its top-dimensional cells are precisely $\Rtp{v_\sigma}{w_\sigma}$ for all Richardson tableaux $\sigma$ of shape $\lambda$. In particular, the number of top-dimensional cells of $\SF{\lambda}^{\geq 0}$ is given by~\eqref{count.intro}.
\end{cor}

We point out that while the geometry of a complex Springer fiber only depends on the Jordan type of the nilpotent matrix $N$ in \eqref{intro_SF}, the structure of its totally nonnegative does depend on $N$ (even within the same conjugacy class); see \cref{eg_tnn} below. \cref{tp_main_intro} only applies to totally nonnegative Springer fibers associated to nilpotent matrices $N$ which are in Jordan form with block sizes weakly decreasing; it remains an open problem to enumerate the top-dimensional cells for other choices of $N$.

Geometrically, \cref{intro.thm} implies that the cohomology classes of the components of $\SF{\lambda}$ indexed by Richardson tableaux are given by a product of Schubert polynomials:
\begin{cor}[\cref{cor.Schubert} below]\label{intro.Schubert}
Let $\sigma$ be a Richardson tableaux. The cohomology class $[\SF{\sigma}]$ is represented by the product of Schubert polynomials $\mathfrak{S}_{v_\sigma} \cdot \mathfrak{S}_{w_0w_\sigma}$ in the cohomology ring $H^*(\Fl_n(\C); \mathbb{Z}) \cong \mathbb{Z}[x_1, \ldots, x_n]/ I^+$.
\end{cor}

\cref{intro.Schubert} partially answers (for these components) a problem posed by Springer \cite{springer83} of finding the expansion of the cohomology class $[\SF{\sigma}]$ in the Schubert basis. Namely, when $\sigma$ is a Richardson tableau, \cref{intro.Schubert} tells us that the answer to Springer's question is given by the Schubert structure constants for the product $\mathfrak{S}_{v_\sigma} \cdot \mathfrak{S}_{w_0w_\sigma}$. We discuss this further in \cref{sec_cohomologyclass}.

The question of which components $\SF{\sigma}$ of Springer fibers are smooth has been extensively studied \cite{vargas79,spaltenstein82,fung03,fresse_melnikov10,fresse_melnikov11,fresse11,graham_zierau11}. For example, it is known that $\SF{\sigma}$ is smooth if $\sigma$ is hook-shaped or has at most two rows, while \eqref{eg_singular_tableaux} in \cref{sec_smoothness} below provides examples of tableaux $\sigma$ such that $\SF{\sigma}$ is singular. Similarly, smoothness of the Richardson varieties $\Rvar{v}{w}$ has also been studied \cite{billey_coskun12,knutson_woo_yong13}, in particular for the special sub-families of Schubert varieties \cite{lakshmibai_sandhya90,billey_lakshmibai00} and positroid varieties \cite{billey_weaver25}. While Richardson varieties and components of Springer fibers can in general be singular, we show that the varieties in the intersection of these two families are smooth:
\begin{thm}[\cref{thm.smooth} below]\label{intro.smooth}
Let $\sigma$ be a Richardson tableau. Then the component $\SF{\sigma} = \Rvar{v_\sigma}{w_\sigma}$ of the Springer fiber $\SF{\lambda}$ is smooth. 
\end{thm}

Our proof of \cref{intro.smooth} uses the combinatorics of the permutations $v_\sigma$ and $w_\sigma$ along with the equivalence \ref{main1} $\Leftrightarrow$ \ref{main8} of \cref{intro.thm}, which follows from Lusztig's work on totally nonnegative Springer fibers ~\cite{lusztig21}.

We also compare the collection of components of $\SF{\lambda}$ studied here with other special components in the literature. We prove that the family of components of $\SF{\lambda}$ equal to Richardson varieties:
\begin{itemize}[itemsep=2pt]
\item strictly contains the family of Richardson components of Pagnon and Ressayre \cite{pagnon_ressayre06};
\item is distinct from the family of generalized Richardson components studied by Fresse \cite{fresse11} (with neither family contained in the other); and
\item strictly contains the family of components associated to $K$-orbits defined by Barchini and Zierau \cite{barchini_zierau08} and studied further by Graham and Zierau \cite{graham_zierau11}.
\end{itemize}
In particular, \cref{intro.smooth} implies that $K$-orbit components of the Springer fiber $\SF{\lambda}$ are smooth, recovering \cite[Lemma 2.6]{graham_zierau11}. An explicit example of a component $\SF{\sigma}$ which is equal to a Richardson variety but which is {\itshape not} part of any of the three families listed above is provided by the Richardson tableau
\begin{align}\label{nonexample}
\sigma = \;\begin{ytableau}
1 & 3 \\
2 \\
4
\end{ytableau}\;.
\end{align}

We point out that one consequence of \cref{intro.thm} is that for all Richardson tableaux $\sigma$, the component $\SF{\sigma}$ is $T$-invariant, where $T\subseteq \GL_n(\C)$ denotes the maximal torus of diagonal matrices.  We discuss some implications of this in \cref{GZ_components}.

\subsection*{Outline}
We introduce relevant notation and background in \cref{sec.notation}. We prove our main combinatorial and enumerative results in \cref{sec_richardson_words,sec_enumeration,sec_slides,sec_reading} and our main geometric results in \cref{sec_geometric,sec_positivity,sec_smoothness}. In particular, in \cref{sec_richardson_words} we study Richardson tableaux, Richardson words, concatenation, and prime Richardson words, and we study their interaction with evacuation. This leads to the characterizations \ref{main2} and \ref{main3} of \cref{intro.thm}.  The enumerative results of \cref{intro.Motzkin,intro.count} appear in \cref{sec_enumeration}, along with a discussion of generating functions associated to Richardson tableaux. \cref{sec_slides} proves the $L$-slide characterization of Richardson tableaux from \cref{intro.thm}\ref{main4}. \cref{sec_reading} introduces the reading words $v_\sigma$ and $w_\sigma$ and proves the length-difference characterization from \cref{intro.thm}\ref{main5}. In \cref{sec_geometric} we identify the Richardson envelope of each component of the Springer fiber (as in \cref{intro.goemetric}) and prove the geometric characterizations \ref{main6} and \ref{main7} of \cref{intro.thm}. \cref{sec_positivity} discusses applications of these geometric results to Lusztig's work on totally nonnegative Springer fibers and derives \cref{intro.thm}\ref{main8}. We prove the smoothness result of \cref{intro.smooth} in \cref{sec_smoothness}. Finally, in \cref{sec_cohomologyclass} we discuss the cohomology classes of the components equal to Richardson varieties, and in \cref{sec_comparison} we compare our components with other collections of components of Springer fibers.

\subsection*{Acknowledgments}
We thank Sara Billey and Vasu Tewari for suggesting we use the major index to prove a $q$-analogue of \eqref{count.intro}, which led to \eqref{q_count_q} below. We are particularly grateful to Vasu Tewari for suggesting the statement of \cref{intro.Motzkin} and for verifying \cref{intro.Motzkin,count.intro} computationally. We also thank Matt Baker and William Graham for helpful comments. We used \texttt{Sage} \cite{sagemath} throughout the course of this work. S.N.K.\ was partially supported by the National Science Foundation under Award No.\ 2452061, by a travel support gift from the Simons Foundation, and by a grant from the Institute for Advanced Study School of Mathematics. M.E.P.\ was partially supported by NSF CAREER grant 2237057.


\section{Notation and background}\label{sec.notation}

\noindent In this section we recall standard background and introduce notation which will be used in the rest of the paper. We let $\mathbb{N}$ denote the set of nonnegative integers and $\mathbb{Z}_{>0}$ the set of positive integers. We set $[n] := \{1, 2, \dots, n\}$ for $n\in\mathbb{N}$.


\subsection{Partitions and tableaux}\label{sec_background_tableaux}
We recall background on tableaux combinatorics; see \cite[Chapter 7]{stanley24} for details. Given $n\in\mathbb{N}$, a \boldit{partition of $n$} is weakly decreasing sequence of positive integers $\lambda = (\lambda_1, \dots, \lambda_\ell)$ such that $\lambda_1 + \cdots + \lambda_\ell = n$. We call $n$ the \boldit{size} of $\lambda$, denoted $|\lambda|$, and write $\lambda\vdash n$. We call $\ell$ the \boldit{length} of $\lambda$, denoted $\ell(\lambda)$. By convention, $\lambda_i := 0$ for all $i > \ell(\lambda)$. We set
\begin{align}\label{lambda_dimension}
n(\lambda):= \sum_{i=1}^{\ell(\lambda)} (i-1)\lambda_i \in \mathbb{N}.
\end{align}

The \boldit{(Young) diagram} of $\lambda$ is the array of $n$ left-justified boxes with exactly $\lambda_i$ boxes in row $i$ for all $i\ge 1$ (where row $1$ is the top row). A \boldit{standard (Young) tableau of shape $\lambda$} is a filling $\sigma$ of the boxes of the diagram of $\lambda$ using the entries $1, \dots, n$, each exactly once, such that entries strictly increase across rows from left to right and down columns from top to bottom. We call $n$ the \boldit{size} of $\sigma$, denoted $|\sigma|$. For $0 \le j \le n$, we let $\sigma[j]$ denote the standard tableaux obtained from $\sigma$ by deleting the boxes containing the entries $j+1, \dots, n$. For convenience, if $n \ge 1$ we set $\bar{\sigma} := \sigma[n-1]$. We let $\SYT(\lambda)$ denote the set of standard tableaux of shape $\lambda$.

Given a standard tableau $\sigma$ of size $n$, for $1 \le j \le n$, we let $\rowjword{\sigma}{j}$ denote the row of $\sigma$ in which $j$ appears. We define the \boldit{lattice word} of $\sigma$ to be $\rowword{\sigma} := \rowjword{\sigma}{1}\rowjword{\sigma}{2} \cdots \rowjword{\sigma}{n}$. Note that $\sigma$ is uniquely determined by its lattice word. We can also recover the shape $\lambda$ of $\sigma$ from $\rowword{\sigma}$ by observing that $\lambda_i$ is the number of occurrences of $i$ in $\rowword{\sigma}$. Conversely, a word $r$ on the alphabet $\mathbb{Z}_{>0}$ comes from a standard tableau $\sigma$ if and only if it satisfies the \boldit{lattice property}: every prefix (i.e.\ initial subword) of $r$ has at least as many $i$'s as $i+1$'s for all $i\ge 1$. All words in this paper are assumed to be finite.

For examples of tableaux and lattice words, see \cref{eg_richardson_tableau_intro,eg_richardson_word}.


\subsection{Evacuation on standard tableaux}\label{sec_background_evacuation}
Let $\lambda\vdash n$. We recall the definition of the \boldit{evacuation} involution (also known as the \boldit{Sch\"{u}tzenberger involution}) on $\SYT(\lambda)$, which we denote by $\sigma\mapsto\sigma^\vee$. We refer to \cite[Section A.1]{fulton97} or \cite[Section A1.4]{stanley24} for details.

Given $\sigma\in\SYT(\lambda)$, we construct $\sigma^\vee$ by performing the following steps for $j = 1, \dots, n$, starting with the tableau $\sigma$.
\begin{enumerate}[label=(\arabic*), leftmargin=36pt, itemsep=2pt]
\item Delete the entry $j$ from the top-left corner, forming an empty box.
\item\label{jeu_de_taquin_step} Slide the empty box down and to the right until it reaches an outer corner of the tableau, where at each step, the empty box swaps places with the lesser of the entries immediately to its right and immediately below it. (This swapping procedure is a special case of \emph{jeu de taquin}.) We call this process the \boldit{$j^\text{th}$ evacuation slide} of $\sigma$, and we call the path the empty box takes through the tableau its \boldit{slide path}.
\item Write the entry $n+1-j$ in $\sigma^\vee$ in the location of the empty box.
\item Remove the empty box from the tableau $\sigma$.
\end{enumerate}
Then $\sigma^\vee$ is a standard tableau of shape $\lambda$, and $(\sigma^\vee)^\vee = \sigma$.

\begin{eg}\label{eg.slide}
Consider the following standard tableau of shape $\lambda = (4,4,3,3,2)$:
\[
\sigma = \;\begin{ytableau} *(yellow)1&*(yellow)2&5&13\\3&*(yellow)4&9&15\\6&*(yellow)8&*(yellow)11\\7&12&*(yellow)14\\10&16\end{ytableau}\;.
\]
We have highlighted the slide path of the first evacuation slide. Explicitly, the path of the empty box in step \ref{jeu_de_taquin_step} above is as shown:
\[
\;\scalebox{0.721}{$\begin{ytableau} *(red!80)\empty&*(yellow)2&5&13\\3&*(yellow)4&9&15\\6&*(yellow)8&*(yellow)11\\7&12&*(yellow)14\\10&16\end{ytableau}$}\;
\rightsquigarrow
\;\scalebox{0.721}{$\begin{ytableau} *(yellow)2&*(red!80)\empty&5&13\\3&*(yellow)4&9&15\\6&*(yellow)8&*(yellow)11\\7&12&*(yellow)14\\10&16\end{ytableau}$}\;
\rightsquigarrow
\;\scalebox{0.721}{$\begin{ytableau} *(yellow)2&*(yellow)4&5&13\\3&*(red!80)\empty&9&15\\6&*(yellow)8&*(yellow)11\\7&12&*(yellow)14\\10&16\end{ytableau}$}\;
\rightsquigarrow
\;\scalebox{0.721}{$\begin{ytableau} *(yellow)2&*(yellow)4&5&13\\3&*(yellow)8&9&15\\6&*(red!80)\empty&*(yellow)11\\7&12&*(yellow)14\\10&16\end{ytableau}$}\;
\rightsquigarrow
\;\scalebox{0.721}{$\begin{ytableau} *(yellow)2&*(yellow)4&5&13\\3&*(yellow)8&9&15\\6&*(yellow)11&*(red!80)\empty\\7&12&*(yellow)14\\10&16\end{ytableau}$}\;
\rightsquigarrow
\;\scalebox{0.721}{$\begin{ytableau} *(yellow)2&*(yellow)4&5&13\\3&*(yellow)8&9&15\\6&*(yellow)11& *(yellow)14\\7&12&*(red!80)\empty\\10&16\end{ytableau}$}\;.
\]
Therefore in $\sigma^\vee$, the largest number (which is $16$) appears at the end of the fourth row. Continuing in this way, we find that
\[
\sigma^\vee = \;\begin{ytableau}
1 & 3 & 5 & 7 \\
2 & 6 & 9 & 10 \\
4 & 11 & 14 \\
8 & 13 & 16 \\
12 & 15
\end{ytableau}\;.
\]
For additional examples of evacuation, see \cref{eg_evacuation_concatenation,eg_Lslide}.
\end{eg}

Recall that $\bar{\sigma}$ is the standard tableau of size $n-1$ obtained from $\sigma$ by deleting the box containing $n$. We will need the following property of evacuation:
\begin{lem}\label{lem.induction}
Let $\sigma$ be a standard tableau. Then $(\bar{\sigma})^{\vee}$ is obtained from $\sigma^{\vee}$ by performing the first evacuation slide on $\sigma^{\vee}$, and then decreasing every entry by $1$.
\end{lem}
\begin{proof}
Since $(\sigma^\vee)^\vee = \sigma$, the slide path of the first evacuation slide of $\sigma^\vee$ ends in the box containing $n$ in $\sigma$. The subsequent evacuation slides of $\sigma^\vee$ will then determine $\bar{\sigma}$.  The result now follows from the fact that evacuation is an involution.
\end{proof}

\begin{eg}
Let $\sigma$ and $\sigma^\vee$ be as in \cref{eg.slide}. We apply~\cref{lem.induction} to compute $(\bar{\sigma})^{\vee}$ from $\sigma^\vee$:
\[
\sigma^\vee = \;\begin{ytableau}
*(yellow)1 & 3 & 5 & 7 \\
*(yellow)2 & 6 & 9 & 10 \\
*(yellow)4 & 11 & 14 \\
*(yellow)8 & 13 & 16 \\
*(yellow)12 & *(yellow)15
\end{ytableau}\;
\quad\rightsquigarrow\quad
\;\begin{ytableau}
*(yellow)2 & 3 & 5 & 7 \\
*(yellow)4 & 6 & 9 & 10 \\
*(yellow)8 & 11 & 14 \\
*(yellow)12 & 13 & 16 \\
*(yellow)15
\end{ytableau}\;
\quad\rightsquigarrow\quad
\;\begin{ytableau}
1 & 2 & 4 & 6 \\
3 & 5 & 8 & 9 \\
7 & 10 & 13 \\
11 & 12 & 15 \\
14 
\end{ytableau}\; = (\bar{\sigma})^{\vee}.
\]
Above, we highlighted the first slide path of $\sigma^{\vee}$ on the left, then performed the first evacuation slide, and then decreased every entry by $1$ to obtain $(\bar{\sigma})^{\vee}$.
\end{eg}


\subsection{Symmetric group}\label{sec_background_Sn}
We recall standard facts about the \boldit{symmetric group} $S_n$ of all permutations of $[n]$; see \cite[Chapter 2]{bjorner_brenti05} for details. We write permutations $w \in S_n$ in \boldit{one-line notation} as $w = w(1)w(2)\cdots w(n)$. The inverse of the permutation $w$ is denoted $w^{-1}$. An \boldit{inversion} of a permutation $w\in S_n$ is a pair $(i,j)\in [n]^2$ such that $i < j$ and $w(i) > w(j)$. The number of inversions of $w$ is called its \boldit{length}, denoted $\ell(w)$. We let $e := 12\cdots n$ denote the identity permutation, and let $w_0 := n \cdots 21$ denote the longest element of $S_n$. Note that $w_0^{-1} = w_0$.

For $(i,j)\in [n]^2$ with $i < j$, let $t_{i,j} \in S_n$ denote the reflection which swaps $i$ and $j$ and fixes all other elements of $[n]$. For $j\in [n-1]$, we let $s_j$ denote $t_{j,j+1}$, called a \boldit{simple reflection}.

The \boldit{Bruhat order} is a partial order $\le$ on $S_n$ which is the transitive closure of all relations $wt_{i,j} \le w$, where $w\in S_n$ and $(i,j)$ is an inversion of $w$. The Bruhat order has minimum $e$ and maximum $w_0$, and is graded by the length function $\ell(\cdot)$. Equivalently, a \boldit{reduced word} for $w\in S_n$ is an expression $w = s_{j_1} \cdots s_{j_m}$ with $m$ minimum (i.e.\ $m = \ell(w)$). Then for all $v,w\in S_n$, we have $v \le w$ if and only if $v$ has a reduced word which is obtained from a reduced word for $w$ by deleting $\ell(w) - \ell(v)$ simple reflections. The involutions $w \mapsto w^{-1}$ and $w\mapsto w_0ww_0$ of $S_n$ are automorphisms of the Bruhat order (i.e.\ they are order-preserving), while the involutions $w\mapsto ww_0$ and $w\mapsto w_0w$ are order-reversing.

\begin{eg}\label{eg_bruhat}
Let $w = 25341 \in S_5$. Then $w$ has $\ell(w) = 6$ inversions: $(1,5)$, $(2,3)$, $(2,4)$, $(2,5)$, $(3,5)$, and $(4,5)$. In particular, we have $wt_{1,5} \le w$, i.e., $15342 \le 25341$. We can also see that $15342 \le 25341$ in terms of reduced words: $15342=s_4s_3s_2s_3s_4$ is a subword of $25341=s_4s_3s_1s_2s_3s_4$.
\end{eg}

Given $\lambda\vdash n$, we define the \boldit{Young subgroup} $S_\lambda$ to be the subgroup of $S_n$ of all permutations which fix set-wise each set
\begin{eqnarray}\label{eqn.Young}
\{\lambda_1+\cdots + \lambda_{i-1}+1, \lambda_1+\cdots + \lambda_{i-1}+2,\ldots, \lambda_1+\cdots + \lambda_{i-1}+ \lambda_i\} 
\end{eqnarray}
for all $1 \le i \le \ell(\lambda)$.
Equivalently, $S_\lambda$ is generated by $s_j$ for all $j\in [n-1]\setminus \{\lambda_1, \lambda_1+\lambda_2, \ldots\}$. The subgroup $S_\lambda$ is a special case of a parabolic subgroup of $S_n$. We consider the cosets of the left action of $S_\lambda$ on $S_n$. Every coset of $S_\lambda\backslash S_n$ contains a unique element of minimal length. Let $\mincoset{S_n}{\lambda}$ denote the set of such minimal-length coset representatives. Explicitly, given $w\in S_n$, we have $w\in\mincoset{S_n}{\lambda}$ if and only if $w^{-1}(j) < w^{-1}(j+1)$ for all $j\in [n-1]\setminus \{\lambda_1, \lambda_1+\lambda_2, \ldots\}$. That is, $w\in\mincoset{S_n}{\lambda}$ if and only if the numbers of each set~\eqref{eqn.Young} appear in increasing order when the one-line notation of $w$ is read from left to right.

Every $w\in S_n$ has a unique \boldit{parabolic factorization},
\[
w = \mincosetsub{w}{\lambda}\cdot\mincoset{w}{\lambda}, \quad \text{ where } \ell(w) = \ell(\mincosetsub{w}{\lambda}) + \ell(\mincoset{w}{\lambda}),\, \mincosetsub{w}{\lambda}\in S_\lambda, \text{ and } \mincoset{w}{\lambda}\in\mincoset{S_n}{\lambda}.
\]
This defines a projection map $S_n\to\mincoset{S_n}{\lambda}$, $w\mapsto\mincoset{w}{\lambda}$. We let $w_{0,\lambda} := \mincosetsub{(w_0)}{\lambda}$ denote the longest element of $S_\lambda$.

\begin{eg}\label{eg_mincoset}
Let $\lambda = (3,2)$ and $w = 25341$. Then the parabolic factorization of $w$ is $w = 23154 \cdot 14253$. We also have $w_{0,\lambda} = \mincosetsub{(54321)}{\lambda} = 21543$. Note that we obtain the shortest coset representative of $w=25341$, which is $ \mincoset{w}{\lambda}=14253$, by sorting the substrings $123$ and $45$ into increasing order.
\end{eg}

We will need the following property of the projection map $S_n\to\mincoset{S_n}{\lambda}$:
\begin{lem}[{\cite[Proposition 2.5.1]{bjorner_brenti05}}]\label{bruhat_projection}
Let $\lambda\vdash n$ and $v\le w$ in $S_n$. Then $\mincoset{v}{\lambda} \le \mincoset{w}{\lambda}$.
\end{lem}


\subsection{Schubert and Richardson varieties}\label{sec_background_schubert}
We recall the definition of the type $A$ flag variety and its Schubert and Richardson subvarieties; see \cite[Section 1]{speyer} for details. Given $n\in\mathbb{N}$, we let $\GL_n(\C)$ denote the general linear group of all $n\times n$ invertible matrices, and let $B$ and $B_-$ denote the Borel subgroups of upper- and lower-triangular matrices, respectively. We let $\transpose{A}$ denote the transpose of an $n\times n$ matrix $A$.

A \boldit{(complete) flag} in $\C^n$ is a sequence of nested subspaces of $\C^n$,
\[
F_\bullet = (0\subset F_1 \subset F_2 \subset \cdots \subset F_{n-1}\subset \C^n),
\]
such that $\dim(F_j) = j$ for all $1\le j \le n-1$.
For convenience, we set $F_0 := \{0\}$ and $F_n := \C^n$. The set of all complete flags is the \boldit{(complete) flag variety}, denoted $\Fl_n(\C)$, which is a smooth projective variety of dimension $\binom{n}{2}$. We identify $\Fl_n(\C)$ with the space $\GL_n(\C)/B$, where the coset $gB$ for $g\in\GL_n(\C)$ defines the flag $F_\bullet$ such that each $F_j$ is spanned by the first $j$ columns of $g$.

The Weyl group of $\GL_n(\C)$ is the symmetric group $S_n$. Given a permutation $w\in S_n$, we let $\dot w$ denote the \boldit{permutation matrix} defined by $\dot{w}_{i,j} := \delta_{i,w(j)}$ for all $i,j\in [n]$. The map $S_n\to\GL_n(\C)$, $w \mapsto \dot{w}$ is a group representation, and $\dot{w}^{-1} = \transpose{\dot{w}}$ for all $w\in S_n$. For $w\in S_n$ we have the permutation flag $\dot{w}B\in\Fl_n(\C)$.

The \boldit{Schubert cells} and \boldit{opposite Schubert cells} are the $B$-orbits and $B_-$-orbits of $\Fl_n(\C)$, respectively, and are indexed by $S_n$:
\[
\Scell{w} := B\dot w B/B \subseteq \Fl_n(\C) \quad \text{ and } \quad \Scellopp{w} := B_- \dot w B/B \subseteq \Fl_n(\C)
\]
for $w \in S_n$. The Schubert cell $\Scell{w}$ is an affine cell of dimension $\ell(w)$, and the opposite Schubert cell $\Scellopp{w}$ is an affine cell of codimension $\ell(w)$. We denote their Zariski closures in $\Fl_n(\C)$ as
\[
\Svar{w} := \overline{\Scell{w}} \quad \text{ and } \quad \Svaropp{w} := \overline{\Scellopp{w}},
\]
called a \boldit{Schubert variety} and an \boldit{opposite Schubert variety}, respectively. For example, $\Svar{w_0} = \Svaropp{e} = \Fl_n(\C)$ and $\Svar{e} = \Svaropp{w_0}=\{\dot e B\}$. For all $v,w\in S_n$, 
\[
\Svar{v} \subseteq \Svar{w}
\quad\Leftrightarrow\quad
\Svaropp{w} \subseteq \Svaropp{v}
\quad\Leftrightarrow\quad
v \le w \text{ in Bruhat order}.
\]

For $v,w\in S_n$, we define the \boldit{open Richardson variety} and the \boldit{Richardson variety} in $\Fl_n(\C)$ by
\begin{align}\label{richardson_defn}
\Rcell{v}{w} := \Scellopp{v} \cap \Scell{w} \quad \text{ and } \quad \Rvar{v}{w} := \overline{\Rcell{v}{w}} = \Svaropp{v} \cap \Svar{w},
\end{align}
respectively. We have
\begin{align}\label{richardson_nonempty}
\Rcell{v}{w} \neq \emptyset
\quad\Leftrightarrow\quad
\Rvar{v}{w}\neq \emptyset
\quad\Leftrightarrow\quad
v \le w \text{ in Bruhat order}.
\end{align}
If $v\le w$, then $\Rvar{v}{w}$ is an irreducible projective variety of dimension $\ell(w) - \ell(v)$. Note that $\Rvar{e}{w} = \Svar{w}$ and $\Rvar{v}{w_0} = \Svaropp{v}$. For all $v \le w$ and $v'\le w'$ in $S_n$, we have
\begin{align}\label{richardson_closure}
\Rvar{v'}{w'} \subseteq \Rvar{v}{w}
\quad\Leftrightarrow\quad
v \le v' \le w' \le w.
\end{align}

\begin{eg}\label{eg_flag}
Consider the flag $F_\bullet = gB \in \Fl_4(\C)$ represented by the matrix
\[
g = \begin{bmatrix}
2 & 1 & 0 & 0 \\
0 & 0 & 0 & 1 \\
1 & 0 & 1 & 0 \\
1 & 0 & 0 & 0
\end{bmatrix} \in \GL_4(\C),
\]
that is, $F_j$ (for $j = 1,2,3,4$) is spanned by the first $j$ columns of $g$. Let us verify that $F_\bullet$ is contained in the open Richardson variety $\Rcell{v}{w}$ with $v = 1342$ and $w = 4132$, i.e., $F_\bullet \in \Scellopp{v} \cap \Scell{w}$. We have
\[
\dot{v} = \begin{bmatrix}
1 & 0 & 0 & 0 \\
0 & 0 & 0 & 1 \\
0 & 1 & 0 & 0 \\
0 & 0 & 1 & 0
\end{bmatrix}
\quad\text{ and }\quad
\dot{w} = \begin{bmatrix}
0 & 1 & 0 & 0 \\
0 & 0 & 0 & 1 \\
0 & 0 & 1 & 0 \\
1 & 0 & 0 & 0
\end{bmatrix},
\]
and the reader can confirm that 
\[
\Scellopp{v} = \left\{\begin{bmatrix}
1 & 0 & 0 & 0 \\
* & 0 & 0 & 1 \\
* & 1 & 0 & 0 \\
* & * & 1 & 0
\end{bmatrix}\cdot B\right\}
\quad \text{ and } \quad
\Scell{w} = \left\{\begin{bmatrix}
* & 1 & 0 & 0 \\
* & 0 & * & 1 \\
* & 0 & 1 & 0 \\
1 & 0 & 0 & 0
\end{bmatrix}\cdot B\right\},
\]
where each $*$ denotes an arbitrary element of $\C$.  It is clear from the description of $\Scell{w}$ above that $F_\bullet = gB \in \Scell{w}$. Note that $F_\bullet = g'B$ with
\[
g' = \begin{bmatrix}
1 & 0 & 0 & 0 \\[2pt]
0 & 0 & 0 & 1 \\[2pt]
\frac{1}{2} & 1 & 0 & 0 \\[2pt]
\frac{1}{2} & 1 & 1 & 0
\end{bmatrix} \in \GL_4(\C).
\]
So $F_\bullet = g'B \in \Scellopp{v}$ also. Thus $F_\bullet \in \Rcell{v}{w}$.
\end{eg}


\subsection{Springer fibers}\label{sec_background_springer}
We recall some background about Springer fibers; see \cite{jantzen04} or \cite{tymoczko17} for an introduction. For $n\in\mathbb{N}$, let $\gl_n(\C)$ denote the Lie algebra of $\GL_n(\C)$ consisting of all $n\times n$ matrices, and let $\fu$ denote the Lie algebra of the unipotent radical of $B$, which consists of all $n\times n$ upper-triangular matrices with zeros on the diagonal. Given $\lambda\vdash n$, we let $N_\lambda\in\gl_n(\C)$ denote the nilpotent matrix in Jordan canonical form with Jordan blocks of sizes $\lambda_1, \dots, \lambda_{\ell(\lambda)}$ (in that order along the diagonal). That is,
\begin{align}\label{nilpotent_formula}
N_\lambda := \sum_{\substack{j \in [n-1],\\ j\notin \{ \lambda_1, \lambda_1+\lambda_2 , \ldots \}}} E_{j,j+1},
\end{align}
where $E_{j,j+1}$ denotes the elementary matrix with unique nonzero entry equal to $1$ in position $(j,j+1)$. 

Each nilpotent matrix $N\in\gl_n(\C)$ is conjugate to $N_\lambda$ for a unique $\lambda\vdash n$; we call this $\lambda$ the \boldit{Jordan type} of $N$. We also define the \boldit{Springer fiber of $N$} to be
\begin{gather}\label{def_springer}
\begin{aligned}
\SF{N} &:= \{F_\bullet \in \Fl_n(\C) \mid N(F_j) \subseteq F_{j-1} \text{ for all $1 \le j \le n$}\} \\
&\hphantom{:}= \{gB \in \GL_n(\C)/B \mid g^{-1}Ng \in \fu\},
\end{aligned}
\end{gather}
which is a closed subvariety of $\Fl_n(\C)$ of dimension $n(\lambda)$. We let $\SF{\lambda}$ denote $\SF{N_\lambda}$.

\begin{rmk}\label{change_of_basis}
If $N\in\gl_n(\C)$ is a nilpotent matrix of Jordan type $\lambda$, then $\SF{N}$ is isomorphic to $\SF{\lambda}$ as a complex algebraic variety, via a change of basis. However, we point out that the definition of Richardson varieties $\Rvar{v}{w}$ (as well as the totally nonnegative part of $\SF{N}$ discussed in \cref{sec_positivity}) depends on the choice of basis in an essential way. Hence for our purposes it is important that we fix a choice of basis and work with $\SF{\lambda}$, rather than an arbitrary Springer fiber $\SF{N}$.
\end{rmk}

We will need the following description of the permutation flags contained in $\SF{\lambda}$:
\begin{lem}\label{permutations_in_SF}
Let $\lambda\vdash n$ and $w\in S_n$. Then $\dot{w}B\in\SF{\lambda}$ if and only if $w\in\mincoset{S_n}{\lambda}$.
\end{lem}

\begin{proof}
By \eqref{nilpotent_formula} we see that $\dot{w}^{-1}N_\lambda \dot w$ is upper-triangular if and only if the numbers of each set~\eqref{eqn.Young} appear in increasing order in the one-line notation of $w$, i.e., $w\in\mincoset{S_n}{\lambda}$.
\end{proof}

Following Spaltenstein \cite{spaltenstein76} (cf.\ \cite{steinberg76,vargas79}), we define a map
\[
\psi : \SF{\lambda} \to \SYT(\lambda)
\]
as follows. Given $F_\bullet \in \SF{\lambda}$, note that $N_\lambda$ induces a nilpotent endomorphism of the quotient space $\C^n/F_j$ for all $0 \le j \le n$. We let $\psi(F_\bullet)$ be the unique standard tableau $\sigma$ such that for all $j\in [n]$, the shape of $\sigma[j]$ is the Jordan type of $N$ acting on $\C^n/F_{n-j}$. Then for $\sigma\in\SYT(\lambda)$, we define
\[
\SFcell{\sigma} := \psi^{-1}(\sigma) \quad \text{ and } \quad \SF{\sigma} := \overline{\SFcell{\sigma}}.
\]
Spaltenstein showed that $\SFcell{\sigma}$ is locally closed in $\SF{\lambda}$, and that the irreducible components of $\SF{\lambda}$ are precisely $\SF{\sigma}$ for $\sigma\in\SYT(\lambda)$, which all have the same dimension $n(\lambda)$.

\begin{eg}\label{eg_springer}
Let $\lambda = (2,1,1)$, and let $F_\bullet = gB \in \Fl_4(\C)$ denote the flag from \cref{eg_flag}. That is,
\[
N_\lambda = \begin{bmatrix}
0 & 1 & 0 & 0 \\
0 & 0 & 0 & 0 \\
0 & 0 & 0 & 0 \\
0 & 0 & 0 & 0
\end{bmatrix} \in \gl_4(\C)
\quad\text{ and }\quad
g = \begin{bmatrix}
2 & 1 & 0 & 0 \\
0 & 0 & 0 & 1 \\
1 & 0 & 1 & 0 \\
1 & 0 & 0 & 0
\end{bmatrix} \in \GL_4(\C).
\]
Then $F_\bullet \in \SF{\lambda}$: indeed, $N_\lambda(F_j) = 0$ for $j = 1,2,3$, and $N_\lambda(\C^4) = \spn{\transpose{(1,0,0,0)}} \subseteq F_3$.

Now let us find $\sigma = \psi(F_\bullet)$. Since $N_\lambda(\C^4)$ is spanned by the vector $\transpose{(1,0,0,0)}$, which is contained in $F_2$ but not in $F_1$, we see that $N_\lambda$ acts as zero on $\C^4/F_2$ but is nonzero on $\C^4/F_1$. This is enough information to determine that $\sigma = \ytableausmall{1 & 3\\ 2\\ 4}$, i.e., $F_\bullet \in \SFcell{\sigma}$.

We point out that $F_\bullet$ is also contained in the irreducible component $\SF{\tau}$ (in addition to $\SF{\sigma}$), where $\tau = \ytableausmall{1 & 2\\ 3\\ 4}$. To see this, define
\[
h(t) = \begin{bmatrix}
2 & t & 1 & 0 \\
0 & 0 & 0 & 1 \\
1 & 1 & 0 & 0 \\
1 & 0 & 0 & 0
\end{bmatrix} \in \GL_4(\C) \quad \text{ for } t\in \C.
\]
The reader can check that $h(t)B \in \SFcell{\tau}$ for all $t\in\C$. Since $F_\bullet = \lim_{t\to\infty} h(t)B$, we get $F_\bullet \in \overline{\SFcell{\tau}} = \SF{\tau}$.
\end{eg}

\begin{rmk}\label{remark_steinberg}
We could alternatively have defined $\psi$ by considering the action of $N$ on $F_j$, rather than $\C^n/F_{n-j}$. This defines an equally valid labeling of the irreducible components of $\SF{\lambda}$ by $\SYT(\lambda)$, which is the convention used by Steinberg \cite{steinberg88}.  In \cite{van_leeuwen00}, van Leeuwen showed these conventions are related by evacuation; see \cref{dual_evacuation} below.
\end{rmk}


\section{Richardson tableaux and Richardson words}\label{sec_richardson_words}

\noindent In this section we introduce Richardson tableaux and Richardson words. The main result is \cref{evacuation_closed}, which states that the set of Richardson tableaux of shape $\lambda$ is closed under evacuation.


\subsection{Richardson tableaux}
Recall the definition of a Richardson tableau from \cref{defn_richardson}. We have the following equivalent definition:
\begin{prop}\label{condition_2}
Let $\sigma\in\SYT(\lambda)$, where $\lambda\vdash n$. Then $\sigma$ is Richardson if and only if for all $j\in [n]$ with $\rowjword{\sigma}{j} > 1$, the largest entry of the tableau $\sigma[j-1]$ in rows $i < \rowjword{\sigma}{j}$ is greater than every entry in rows $i \ge \rowjword{\sigma}{j}$.
\end{prop}

\begin{proof}
The backward direction follows from the definitions. For the forward direction (via the contrapositive), we suppose that $\sigma$ violates the conclusion of the proposition for some $j\in [n]$ with $\rowjword{\sigma}{j} > 1$, and then show that $\sigma$ is not Richardson. To this end, let $b$ denote the largest entry of $\sigma[j-1]$ in rows $i \ge \rowjword{\sigma}{j}$. Then there exists a row $s < \rowjword{\sigma}{j}$ such that the largest entry of $\sigma[j-1]$ in row $s$, which we denote by $a$, is less than $b$; let us take the maximum $s$ with this property. If $s = \rowjword{\sigma}{j} - 1$, then $\sigma$ violates \cref{defn_richardson} for $j$, and we are done. Otherwise, let $j'$ denote the largest entry of $\sigma[j-1]$ in row $s+1$. Then by the maximality of $s$, we have $j' > b$. Since $a<b$ also, $\sigma$ violates \cref{defn_richardson} for $j'$.
\end{proof}

Given a partition $\lambda$, we let $\Rtabs{\lambda}$ denote the set of Richardson tableaux of shape $\lambda$.
\begin{rmk}\label{explicit}
We can explicitly describe $\Rtabs{\lambda}$ in some special cases:
\begin{enumerate}[label=(\roman*), leftmargin=*, itemsep=2pt]
\item\label{explicit_two_row} If $\ell(\lambda) \le 2$, that is, if $\lambda$ has at most two rows, then $\Rtabs{\lambda}$ consists of all standard tableaux with no two consecutive entries in the second row. For example, there are two standard tableaux of shape $(2,2)$:
\begin{gather*}
\sigma = \; \begin{ytableau}
1 & 3 \\
2 & 4 
\end{ytableau}
\quad 
\text{ and } 
\quad 
\tau = \;\begin{ytableau}
1 &2\\
3 & 4
\end{ytableau}\;.
\end{gather*}
It is straightforward to confirm that $\sigma$ is Richardson while $\tau$ is not ($j=4$ gives a violation of \cref{defn_richardson}). 

\item\label{explicit_rectangle} If $\lambda = (k^\ell) := (\overbrace{k,\ldots, k}^{\text{$\ell$ times}})$ is a rectangle, then $\Rtabs{\lambda}$ contains a single element, consisting of the tableau obtained by writing entries $1, 2, \dots, k\ell$ column-by-column, from top to bottom and left to right. For example,
\[
\Rtabs{3,3,3,3} = \left\{\;
\begin{ytableau}1&5&9\\2&6&10\\3&7&11\\4&8&12\end{ytableau}
\;\right\}.
\]
\item\label{explicit_hook} If $\lambda = (k, 1^{\ell-1}) := (k, \overbrace{1, \ldots, 1}^{\text{$\ell-1$ times}})$ is a hook, then every standard tableau of shape $\lambda$ is Richardson.\qedhere
\end{enumerate}
\end{rmk}


\subsection{Richardson words}
Recall the lattice word $\rowword{\sigma}$ of a standard tableau $\sigma$ defined in \cref{sec_background_tableaux}. We call a word on the alphabet $\mathbb{Z}_{>0}$ a \boldit{Richardson word} if it is the lattice word of a Richardson tableau. Working with Richardson tableaux is equivalent to working with Richardson words. We will use the combinatorics of words to recursively characterize Richardson tableaux (see \cref{prime_star,prime_recursion}). We begin with the following observation:
\begin{prop}\label{richardson_tableau_to_word}
Let $r$ be a word on the alphabet $\mathbb{Z}_{>0}$. Then $r$ is a Richardson word if and only if for all $j\in [n]$ with $r_j \ge 2$, when we read the prefix $r_1r_2\cdots r_{j-1}$ from right to left, the number $r_j-1$ appears, and it appears before any occurrence of $r_j, r_j+1, \dots$.
\end{prop}

\begin{proof}
This is the translation of \cref{defn_richardson} into the setting of words.
\end{proof}

\begin{eg}\label{eg_richardson_word}
Recall the standard tableaux $\sigma$ and $\tau$ from \cref{eg_richardson_tableau_intro}. They have lattice words
\[
\rowword{\sigma} = 12113123 \quad \text{ and } \quad \rowword{\tau} = 11213213.
\]
The first word is Richardson. The second word is not Richardson, since it fails the condition in \cref{richardson_tableau_to_word} when $j=6$: when we read the prefix $11213$ from right to left, $r_6-1=2-1=1$ appears after $r_6+1=2+1=3$.
\end{eg}

We point out that the condition on words in \cref{richardson_tableau_to_word} implies the lattice property.


\subsection{Concatenation and prime decompositions}
Let $\circ$ denote concatenation of words. Also let $\circ$ denote the induced \boldit{concatenation of standard tableaux}, i.e., given standard tableaux $\sigma$ and $\tau$, the standard tableau $\sigma \circ \tau$ is defined by
\[
\rowword{\sigma \circ \tau} = \rowword{\sigma} \circ \rowword{\tau}.
\]
Equivalently, $\sigma \circ \tau$ is obtained by concatenating the rows of $\sigma$ and $\tau$, after adding $|\sigma|$ to every entry of $\tau$. Note that concatenation is associative.

\begin{eg}\label{eg_concatenation}
The concatenation of words $12113 \circ 123 = 12113123$ corresponds to the concatenation
\begin{gather*}
\;\begin{ytableau}
1 & 3 & 4 \\
2 \\
5
\end{ytableau}\;
\circ
\;\begin{ytableau}
1 \\
2 \\
3
\end{ytableau}\;
=
\;\begin{ytableau}
1 & 3 & 4 & 6 \\
2 & 7 \\
5 & 8
\end{ytableau}\;
\end{gather*}
of standard tableaux.
\end{eg}

\begin{rmk}\label{concatenation_geometric}
Fresse \cite[Proposition 7.1]{fresse11} showed that for all standard tableaux $\sigma$ and $\tau$, we have the isomorphism
\[
\SF{\sigma\circ\tau} \cong \SF{\sigma} \times \SF{\tau}
\]
of irreducible components of Springer fibers. This reduces the study of $\SF{\sigma\circ\tau}$ to the study of $\SF{\sigma}$ and $\SF{\tau}$, and geometrically motivates the notion of prime decomposition for standard tableaux below.
\end{rmk}

We recall the definition of the prime decomposition of a standard tableau, which appeared in work of Tirrell \cite[Proposition 1.3.3]{tirrell16}. Given a lattice word $r$, we call $r = r_1 \circ \cdots \circ r_m$ a \boldit{lattice decomposition of $r$} if $r_1, \dots, r_m$ are lattice words. We call $r$ \boldit{prime} if it is nonempty and its only lattice decomposition is $r = r$. We can verify from the definition of the lattice property that the common refinement of two lattice decompositions of $r$ is again a lattice decomposition. We call the common refinement of all lattice decompositions of $r$ its \boldit{prime decomposition}. Equivalently, the prime decomposition of $r$ is the unique lattice decomposition $r = r_1 \circ \cdots \circ r_m$ such that $r_1, \dots, r_m$ are prime. We define a \boldit{prime Richardson word} to be a Richardson word which is also prime.

We analogously define a \boldit{prime standard tableau} to be a standard tableau whose lattice word is prime. Then every standard tableau $\sigma$ has a unique \boldit{prime decomposition} $\sigma = \sigma_1 \circ \cdots \circ \sigma_m$. We define a \boldit{prime Richardson tableau} to be a Richardson tableau which is also prime.

\begin{eg}\label{eg.4}There are ten standard tableaux of size $4$, shown below.
\begin{center}
\vspace*{4pt}
\scalebox{0.7}{$\begin{ytableau}
1 & 2 & 3 & 4 \\
\none \\
\none \\
\none
\end{ytableau}$}
\;\;
\scalebox{0.7}{$\begin{ytableau}
1 & 2 & 3 \\
4 \\
\none \\
\none
\end{ytableau}$}
\;\;
\scalebox{0.7}{$\begin{ytableau}
1 & 2 & 4 \\
3 \\
\none \\
\none
\end{ytableau}$}
\;\;
\scalebox{0.7}{$\begin{ytableau}
1 & 3 & 4 \\
2 \\
\none \\
\none
\end{ytableau}$}
\;\;
\scalebox{0.7}{$\begin{ytableau}
1 & 2 \\
3 & 4 \\
\none \\
\none
\end{ytableau}$}
\;\;
\scalebox{0.7}{$\begin{ytableau}
1 & 3 \\
2 & 4 \\
\none \\
\none
\end{ytableau}$}
\;\;
\scalebox{0.7}{$\begin{ytableau}
1 & 2 \\
3 \\
4 \\
\none
\end{ytableau}$}
\;\;
\scalebox{0.7}{$\begin{ytableau}
1 & 3 \\
2 \\
4 \\
\none
\end{ytableau}$}
\;\;
\scalebox{0.7}{$\begin{ytableau}
1 & 4 \\
2 \\
3 \\
\none
\end{ytableau}$}
\;\;
\scalebox{0.7}{$\begin{ytableau}
1 \\
2 \\
3 \\
4
\end{ytableau}$}
\vspace*{4pt}
\end{center}
Of these ten, all are Richardson tableaux except for $\;\scalebox{0.7}{$\begin{ytableau}1 & 2 \\ 3 & 4\end{ytableau}$}\;$. Also, of these ten, exactly three are prime tableaux, shown below.
\begin{center}
\vspace*{4pt}
\scalebox{0.7}{$\begin{ytableau}
1 & 2 \\
3 & 4 \\
\none \\
\none
\end{ytableau}$}
\hspace*{36pt}
\scalebox{0.7}{$\begin{ytableau}
1 & 3 \\
2 \\
4 \\
\none
\end{ytableau}$}
\hspace*{36pt}
\scalebox{0.7}{$\begin{ytableau}
1 \\
2 \\
3 \\
4
\end{ytableau}$}
\vspace*{4pt}
\end{center}
The two prime Richardson tableaux of size $4$ are the last two of these tableaux. Note that the last two rows of these prime Richardson tableaux have length $1$. We will show in \cref{prime_last_two_rows} below that this is true of all prime Richardson tableaux with at least two rows.
\end{eg}

\begin{rmk}\label{rmk.zerocase} There is a unique partition of size $0$, whose diagram is empty. Thus there exactly one standard tableau of size $0$ (called the \boldit{empty tableau}), which is vacuously Richardson. Its lattice word is the empty word, which is the identity element for concatenation of words. Note that by definition (and to ensure the uniqueness of prime factorization), the empty tableau/word is not prime.
\end{rmk}

The decomposition in \cref{eg_concatenation} is a prime decomposition. Note that the tableau in \cref{eg_concatenation} of size $8$ is a Richardson tableau, and so are its prime factors. This is a general phenomenon:
\begin{lem}\label{prime_concatenation}~
\begin{enumerate}[label=(\roman*), leftmargin=*, itemsep=2pt]
\item\label{prime_concatenation_forward} The concatenation of Richardson tableaux is a Richardson tableau.
\item\label{prime_concatenation_backward} Every prime factor in the prime decomposition of a Richardson tableau is a Richardson tableau.
\end{enumerate}
\end{lem}

\begin{proof}
For part \ref{prime_concatenation_forward}, we must show that if $r_1$ and $r_2$ are Richardson words, then so is $r_1 \circ r_2$. This follows from \cref{richardson_tableau_to_word}. For part \ref{prime_concatenation_backward}, we must show that if $r_1 \circ r_2$ is a Richardson word and $r_1$ and $r_2$ are lattice words, then $r_1$ and $r_2$ are also Richardson words. The fact that $r_1$ is Richardson follows from \cref{richardson_tableau_to_word}. To obtain that $r_2$ is Richardson, we need to know that if the number $i$ appears in $r_2$, then $i-1$ appears to the left of $i$ in $r_2$ (whence the fact that $r$ is Richardson implies that $r_2$ is Richardson). This follows from the fact that $r_2$ is a lattice word.
\end{proof}

We can rephrase \cref{prime_concatenation} as a statement about sets of words under concatenation (see \cref{prime_star}). To do so, we introduce some notation. For $\ell\ge 0$, let $\Rbiggest{\ell}$ denote the set of Richardson words on the alphabet $[\ell]$, i.e., $\Rbiggest{\ell}$ is the set of lattice words of Richardson tableaux of length at most $\ell$. Similarly, for $\ell\ge 1$, let $\Rprime{\ell}$ denote the set of prime Richardson words whose largest letter is $\ell$, i.e., $\Rprime{\ell}$ is the set of lattice words of prime Richardson tableaux of length $\ell$. Note that the set of prime Richardson words is the disjoint union $\bigsqcup_{\ell\ge 1}\Rprime{\ell}$.

If $S$ and $T$ are sets of finite words, we define
\[
S\circ T := \{s\circ t \mid s\in S, t\in T\}.
\]
We also let $S^*$ denote the set of all finite words formed by arbitrary concatenations of elements in $S$ (including the empty word), i.e.,
\[
S^* := \{s_1 \circ \cdots \circ s_m \mid m\in\mathbb{N} \text{ and } s_1, \dots, s_m \in S\}.
\]

\begin{prop}\label{prime_star}
We have
\begin{align}\label{equation_prime_star}
\Rbiggest{\ell} = (\Rprime{1} \sqcup \cdots \sqcup \Rprime{\ell})^* \quad \text{ for all } \ell\ge 0,
\end{align}
where every element of $\Rbiggest{\ell}$ is obtained uniquely as a concatenation of words in $\Rprime{1} \sqcup \cdots \sqcup \Rprime{\ell}$.
\end{prop}

\begin{proof}
This follows from \cref{prime_concatenation}.
\end{proof}

\begin{rmk}\label{prime_algorithm}
We can calculate the prime decomposition of every Richardson word $r$ by reading it from right to left, as follows. Let $j$ denote the last letter of $r$. By definition, $j-1$ appears in $r$; we find the first occurrence, reading right to left. Then continuing to read from right to left from the first occurrence of $j-1$, we find the next occurrence of $j-2$, and then the next occurrence of $j-3$, and so on, until we reach $1$. Let $s$ denote the suffix of $r$ beginning with this $1$. By construction $s$ is prime, so it is the last prime factor in the prime decomposition of $s$. To find the remaining prime factors of $r$, we recursively apply the procedure above to $r$ with $s$ deleted.
\end{rmk}

\begin{eg}\label{eg_prime_decomposition}
Let us illustrate the procedure in \cref{prime_algorithm} for the Richardson tableau
\[
\sigma = \;\begin{ytableau}
1 & 4 & 8 & 9 & 11 \\
2 & 5 & 10 \\
3 & 6 & 12 \\
7
\end{ytableau}\; \quad \text{ with Richardson word } r = \rowword{\sigma} = 123123411213.
\]
To find the prime decomposition of $r$, we read it from right to left. Since $r$ ends in a $3$, we read from right to left until we find a $2$, and then continue reading until we find a $1$. This gives the suffix $1213$, which is the last word in the prime decomposition. Deleting this suffix from $r$ leaves $12312341$, which ends in $1$, so the second-last word in the prime decomposition is $1$. Continuing on, we find that the prime decomposition of $r$ is
\[
r = 123 \circ 1234 \circ 1 \circ 1213.
\]
In terms of tableaux, the prime decomposition of $\sigma$ is
\begin{gather*}
\sigma = \;\begin{ytableau}
1 \\
2 \\
3
\end{ytableau}\;
\circ
\;\begin{ytableau}
1 \\
2 \\
3 \\
4
\end{ytableau}\;
\circ
\;\begin{ytableau}
1
\end{ytableau}\;
\circ
\;\begin{ytableau}
1 & 3 \\
2 \\
4
\end{ytableau}\;.
\end{gather*}
Note that each prime factor is a Richardson tableau, in agreement with \cref{prime_concatenation}\ref{prime_concatenation_backward}.
\end{eg}

We can reinterpet \cref{prime_algorithm} as follows:
\begin{prop}\label{prime_recursion}
We have $\Rprime{1} = \{1\}$ and
\begin{align}\label{equation_prime_recursion}
\Rprime{\ell} = \Rprime{\ell-1} \circ \Rbiggest{\ell-2} \circ \{\ell\} \quad \text{ for all } \ell\ge 2,
\end{align}
where every element of $\Rprime{\ell}$ is obtained uniquely as a concatenation from the right-hand side.
\end{prop}

\begin{proof}
This follows from \cref{prime_algorithm}.
\end{proof}

Now using \eqref{equation_prime_recursion} and \eqref{equation_prime_star}, we can recursively calculate $(\Rprime{i})_{i\ge 1}$:
\begin{eg}\label{eg_prime_recursion}
We have $\Rprime{1} = \{1\}$, so the unique prime Richardson tableau with exactly one row is $\ytableausmall{1}$. We have
\[
\Rprime{2} = \Rprime{1} \circ \Rbiggest{0} \circ \{2\} = \{1\} \circ \{\emptyset\} \circ \{2\} = \{12\},
\]
so the unique prime Richardson tableau with exactly two rows is
\[
\;\begin{ytableau}
1 \\
2
\end{ytableau}\;.
\]
We have
\[
\Rprime{3} = \Rprime{2} \circ \Rbiggest{1} \circ \{3\} = \{12\} \circ \{1\}^* \circ \{3\} = \{123, 1213, 12113, 121113, \dots\},
\]
so the prime Richardson tableaux with exactly three rows are precisely as follows.
\[
\;\begin{ytableau}
1 \\
2 \\
3
\end{ytableau}\;,
\quad
\;\begin{ytableau}
1 & 3 \\
2 \\
4
\end{ytableau}\;,
\quad
\;\begin{ytableau}
1 & 3 & 4 \\
2 \\
5
\end{ytableau}\;,
\quad
\;\begin{ytableau}
1 & 3 & 4 & 5 \\
2 \\
6
\end{ytableau}\;,
\quad\dots
\]
As a final example, we calculate that
\begin{gather*}
\Rprime{4} = \Rprime{3} \circ \Rbiggest{2} \circ \{4\} = \{12\} \circ \{1\}^* \circ \{3\} \circ \{1, 12\}^* \circ \{4\},
\end{gather*}
which gives an explicit description of prime Richardson tableaux with exactly $4$ rows.
\end{eg}

We will need the following consequence of \cref{prime_recursion} in \cref{sec_motzkin}:
\begin{cor}\label{prime_last_two_rows}
The last two rows of every prime Richardson tableau \textup{(}other than $\ytableausmall{1}$\textup{)} have length one. 
\end{cor}
\begin{proof}
Applying \eqref{equation_prime_recursion} twice, we obtain
\begin{align}\label{delete_two_recursion_old}
\Rprime{\ell} = \Rprime{\ell-2} \circ \Rbiggest{\ell-3} \circ \{\ell-1\} \circ \Rbiggest{\ell-2} \circ \{\ell\} \quad \text{ for all } \ell\ge 3.
\end{align}
This (along with the fact that $\Rprime{2} = \{12\}$) implies the result. 
\end{proof}

We point out that \cref{prime_recursion} implies the following characterization of prime Richardson tableaux:
\begin{cor}\label{prime_characterization}
Let $r$ be a nonempty Richardson word with largest letter $\ell$. Then $r$ is prime if and only if:
\begin{enumerate}[label=(\roman*), leftmargin=*, itemsep=2pt]
\item\label{prime_characterization_biggest} $\ell$ appears exactly once in $r$, in the final position; and
\item\label{prime_characterization_smaller} for all $1 \le j \le \ell-1$, the letters of $r$ between the first occurrence of $j$ and the first occurrence of $j+1$ are all at most $j-1$ \textup{(}when we read $r$ from left to right\textup{)}.
\end{enumerate}
\end{cor}

\begin{proof}
The forward direction follows from \eqref{equation_prime_recursion} by induction. For the backward direction, suppose that $r$ satisfies condition \ref{prime_characterization_biggest} but is not prime; we will show that condition \ref{prime_characterization_smaller} fails. Let $r = r_1 \circ \cdots \circ r_m$ denote the prime decomposition of $r$, so that $m \ge 2$. Let $j$ denote the largest letter of $r_1$, whence $j \le \ell-1$ by condition \ref{prime_characterization_biggest}. 
By~\eqref{equation_prime_recursion} applied to $r_1$, the first occurrence of $j$ in $r$ is at the end of $r_1$.  Consider the first occurrence of $j+1$ in $r_2 \circ \cdots \circ r_m$. The lattice property implies that at least one $j$ appears in $r_2\circ \cdots \circ r_m$ before this $j+1$. Therefore in the word $r$, there is a $j$ between the first occurrence of $j$ and the first occurrence of $j+1$, so condition \ref{prime_characterization_smaller} fails.
\end{proof}

\begin{eg}\label{prime_hook}
Recall from \cref{explicit}\ref{explicit_hook} that every standard tableau $\sigma$ of hook shape is Richardson. We can apply \cref{prime_characterization} to $r = \rowword{\sigma}$ to see that such a $\sigma$ is prime if and only if both $2$ and $n$ appear in the first column of $\sigma$ (where $n = |\sigma|$ and we assume $n \ge 2$).
\end{eg}


\subsection{Closure under evacuation}
The goal of this subsection is to prove \cref{evacuation_closed}, which states that the set of Richardson tableaux is closed under evacuation. We let evacuation act on words, i.e., if $\sigma$ is a standard tableau, then we set $\rowword{\sigma}^\vee := \rowword{\sigma^\vee}$. We will need the following key property:
\begin{prop}\label{evacuation_concatenation}
Evacuation anticommutes with concatenation. That is,
\[
(\sigma \circ \tau)^\vee = \tau^\vee \circ \sigma^\vee \; \text{ for all standard tableaux $\sigma$ and $\tau$}.
\]
\end{prop}

\begin{proof}
For $j = 1, \dots, |\sigma|$, the $j$th evacuation slide of $\sigma\circ\tau$ first follows the path of the $j$th evacuation slide of $\sigma$, until it reaches $\tau$, say in row $k$; then it proceeds straight across row $k$ of $\tau$, shifting the row one unit to the left. In particular, after performing $|\sigma|$ evacuation slides to $\sigma\circ\tau$, we obtain the tableau $\tau$. The result follows.
\end{proof}

\begin{eg}\label{eg_evacuation_concatenation}
Let $\sigma$ be as in \cref{eg_prime_decomposition}. Then by \cref{evacuation_concatenation}, we have
\begin{align*}
\sigma^\vee &= \;\begin{ytableau}
1 & 3 \\
2 \\
4
\end{ytableau}^{\;\vee}
\circ
\;\begin{ytableau}
1
\end{ytableau}^{\;\vee}
\circ
\;\begin{ytableau}
1 \\
2 \\
3 \\
4
\end{ytableau}^{\;\vee}
\circ
\;\begin{ytableau}
1 \\
2 \\
3
\end{ytableau}^{\;\vee}\\[4pt]
&=
\;\begin{ytableau}
1 & 3 \\
2 \\
4
\end{ytableau}\;
\circ
\;\begin{ytableau}
1
\end{ytableau}\;
\circ
\;\begin{ytableau}
1 \\
2 \\
3 \\
4
\end{ytableau}\;
\circ
\;\begin{ytableau}
1 \\
2 \\
3
\end{ytableau}\;
=
\;\begin{ytableau}
1 & 3 & 5 & 6 & 10 \\
2 & 7 & 11 \\
4 & 8 & 12 \\
9
\end{ytableau}\;,
\end{align*}
which we can also verify using the definition of evacuation (see \cref{eg_Lslide} below).
\end{eg}

Let $\ell\in\mathbb{N}$, and let $r$ be a word on the alphabet $[\ell]$ such that $1, \dots, \ell$ all appear. Then we define $\newrow{\ell}{r}$ to be the word obtained from $r$ by increasing the first occurrences of $1, \dots, \ell$ by $1$, and then prepending $1$. When $\sigma$ is a standard tableau of length $\ell$, then $\newrow{\ell}{\rowword{\sigma}}$ is the lattice word of the standard tableau obtained from $\sigma$ by shifting the first column down by $1$ unit, increasing all entries by $1$, and placing $1$ in the top-left corner. (Informally, this is an inverse evacuation slide involving only the first column.)

\begin{eg}\label{eg_word_slide}
We have $\newrow{3}{12113123} = 1\underline{2}\underline{3}11\underline{4}123$, where we have underlined the entries which were increased by $1$. Note that $12113123$ is the lattice word of the tableau $\sigma$, and the corresponding operation on tableaux is depicted below:
\begin{gather*}
\begin{gathered}
\sigma = \, \\
~
\end{gathered}\;\begin{ytableau}
1 & 3 & 4 & 6 \\
2 & 7 \\
5 & 8 \\
\none
\end{ytableau}\;
\quad\longmapsto\quad
\;\begin{ytableau}
1 & 4 & 5 & 7 \\
2 & 8 \\
3 & 9 \\
6
\end{ytableau}\;.
\end{gather*}
Note that $\sigma$ is obtained from the tableau on the right by performing the first evacuation slide (and renumbering), and the slide path is straight down the first column.
\end{eg}

\begin{lem}\label{word_slide_lemma}
Let $r$ be the lattice word of a standard tableau of length $\ell\ge 0$. Then \[
(r\circ\{\ell+1\})^\vee = \newrow{\ell}{r^\vee}.
\]
\end{lem}

\begin{proof}
Let $\sigma$ denote the standard tableau whose lattice word is $r\circ\{\ell+1\}$. We must show that $\newrow{\ell}{r^\vee}$ is the lattice word of $\sigma^\vee$. Note that $r$ is the lattice word of $\bar{\sigma}$, and $\sigma$ is obtained from $\bar{\sigma}$ by adding a box in a new row with the entry $|\sigma|$. By \cref{lem.induction}, $(\bar{\sigma})^\vee$ (whose lattice word is $r^\vee$) is obtained from $\sigma^\vee$ by performing the first evacuation slide on $\sigma^{\vee}$ and decreasing every entry by $1$. We must show that this slide is along the first column. This follows from the fact that $|\sigma|$ occurs at the bottom of the first column of $\sigma$.
\end{proof}

\begin{thm}\label{evacuation_closed_prime}
For all $\ell\ge 1$, the set $\Rprime{\ell}$ of prime Richardson words with largest letter $\ell$ is closed under the evacuation map $(\cdot)^\vee$.
\end{thm}

\begin{proof}
We proceed by induction on $\ell$. For the base cases, we can verify that $\Rprime{1} = \{1\}$ and $\Rprime{2} = \{12\}$ are closed under evacuation. Now suppose that $\ell\ge 3$ and $\Rprime{1}, \dots, \Rprime{\ell-1}$ are closed under evacuation (whence $\Rbiggest{0}, \dots, \Rbiggest{\ell-1}$ are closed under evacuation by \eqref{equation_prime_star} and \cref{evacuation_concatenation}). We must show that $\Rprime{\ell}$ is closed under evacuation. We have
\begin{align*}
\Rprime{\ell}^\vee &= (\Rprime{\ell-1} \circ \Rbiggest{\ell-2} \circ \{\ell\})^\vee \qquad \text{(by \eqref{equation_prime_recursion})}\\[2pt]
&= \newrow{\ell-1}{(\Rprime{\ell-1} \circ \Rbiggest{\ell-2})^\vee} \qquad \text{(by \cref{word_slide_lemma})}\\[2pt]
&= \newrow{\ell-1}{\Rbiggest{\ell-2}^\vee \circ \Rprime{\ell-1}^\vee} \qquad \text{(by \cref{evacuation_concatenation})}\\[2pt]
&= \newrow{\ell-1}{\Rbiggest{\ell-2} \circ \Rprime{\ell-1}} \qquad \text{(by the induction hypothesis)}.
\end{align*}
Therefore $\Rprime{\ell}$ is closed under evacuation if and only if 
\begin{align}\label{induction_to_prove}
\newrow{\ell-1}{\Rbiggest{\ell-2} \circ \Rprime{\ell-1}} = \Rprime{\ell}.
\end{align}

By the induction hypothesis, \eqref{induction_to_prove} holds with $\ell$ replaced by $\ell-1$, i.e., we have
\begin{align}\label{induction_hypothesis_equivalent}
\newrow{\ell-2}{\Rbiggest{\ell-3} \circ \Rprime{\ell-2}} = \Rprime{\ell-1}.
\end{align}
Now using \eqref{equation_prime_recursion} again, we obtain
\[
\Rbiggest{\ell-2} \circ \Rprime{\ell-1} = \Rbiggest{\ell-2} \circ \Rprime{\ell-2} \circ \Rbiggest{\ell-3} \circ \{\ell-1\}.
\]
We also have
\[
\Rbiggest{\ell-2} \circ \Rprime{\ell-2} \circ \Rbiggest{\ell-3} = \Rbiggest{\ell-3} \circ \Rprime{\ell-2} \circ \Rbiggest{\ell-2},
\]
since both sides are the set of Richardson words on the alphabet $[\ell-2]$ with at least one prime factor in $\Rprime{\ell-2}$. Hence
\begin{align*}
\newrow{\ell-1}{\Rbiggest{\ell-2} \circ \Rprime{\ell-1}} &= \newrow{\ell-1}{\Rbiggest{\ell-3} \circ \Rprime{\ell-2} \circ \Rbiggest{\ell-2} \circ \{\ell-1\}} \\[2pt]
&= \newrow{\ell-2}{\Rbiggest{\ell-3} \circ \Rprime{\ell-2}} \circ \Rbiggest{\ell-2} \circ \{\ell\} \qquad \text{(by definition of $\newrownoarg{i}$)}\\[2pt]
&= \Rprime{\ell-1} \circ \Rbiggest{\ell-2} \circ \{\ell\} \qquad \text{(by \eqref{induction_hypothesis_equivalent})}\\[2pt]
&= \Rprime{\ell} \qquad \text{(by \eqref{equation_prime_recursion})}.
\end{align*}
This proves \eqref{induction_to_prove} and completes the induction.
\end{proof}

\begin{cor}\label{evacuation_closed}
The set of Richardson tableaux is closed under the evacuation map $(\cdot)^\vee$.
\end{cor}

\begin{proof}
This follows from \eqref{equation_prime_star}, \cref{evacuation_closed_prime}, and \cref{evacuation_concatenation}.
\end{proof}


\section{Enumeration of Richardson tableaux}\label{sec_enumeration}

\noindent In this section we count Richardson tableaux in various ways. Our main results are formulas for the number of Richardson tableaux of fixed size (\cref{motzkin_richardson}), the number of Richardson tableaux of fixed shape (\cref{q_count}), and generating functions of Richardson tableaux (\cref{richardson_ogf}).


\subsection{Enumeration of Richardson tableaux of fixed size}\label{sec_motzkin}
For $n\in\mathbb{N}$, define the \boldit{Motzkin number} $M_n$ to be the number of paths with $n$ steps from $(0,0)$ to $(n,0)$ which never pass below the $x$-axis, using three possible kinds of steps: up (i.e., the step $(1,1)$), down (i.e., the step $(1,-1)$), and horizontal (i.e., the step $(1,0)$). We have $(M_n)_{n=0}^\infty = (1, 1, 2, 4, 9, 21, \dots)$; for example, see \cref{figure_motzkin}. For further background on Motzkin numbers, we refer to \cite{donaghey_shapiro77} and \cite[Exercises 6.37 and 6.38]{stanley24}.

We will show that the number of Richardson tableaux of size $n$ equals $M_n$ in  \cref{motzkin_richardson},\footnote{We thank Vasu Tewari for suggesting this result to us.} and then use it to asymptotically calculate the proportion of standard tableaux which are Richardson in \cref{proportion_richardson}. The key is a bijection between prime Richardson words and Richardson words:
\begin{thm}\label{delete_two}
For all $\ell\geq 2$, there is a bijection 
\begin{eqnarray}\label{eqn.bijection}
\Psi: \Rprime{\ell} \to \Rbiggest{\ell-2}\setminus\Rbiggest{\ell-3}
\end{eqnarray}
between the set $\Rprime{\ell}$ of prime Richardson words whose largest letter is $\ell$ and the set $\Rbiggest{\ell-2} \setminus \Rbiggest{\ell-3}$ of Richardson words whose largest letter is $\ell-2$, such that for all $r\in\Rprime{\ell}$, the multiset of letters appearing in $r$ is the union of $\{\ell-1, \ell\}$ with the multiset of letters appearing in $\Psi(r)$. \textup{(}By convention, we set $\Rbiggest{-1} := \emptyset$.\textup{)}
\end{thm}

\begin{proof} If $\ell=2$ then $\Rprime{\ell} = \{12\}$, and $\Rbiggest{0}\setminus\Rbiggest{-1} = \{\varnothing\}$ is the singleton consisting of the empty word $\varnothing$ (cf.\ \cref{rmk.zerocase}). We set $\Psi(12) := \varnothing$.

Now assume $\ell\geq 3$. Recall from \eqref{delete_two_recursion_old} that
\begin{align}\label{delete_two_recursion}
\Rprime{\ell} = \Rprime{\ell-2} \circ \Rbiggest{\ell-3} \circ \{\ell-1\} \circ \Rbiggest{\ell-2} \circ \{\ell\} \quad \text{ for all } \ell\ge 3.
\end{align}
Given $r\in\Rprime{\ell}$, write $r = s \circ t \circ (\ell-1) \circ u \circ \ell$ as in \eqref{delete_two_recursion}. We define the map~\eqref{eqn.bijection} by
\[
\Psi( r) = s \circ t \circ (\ell-1) \circ u \circ \ell \mapsto t \circ s \circ u,
\]
which is well-defined by \cref{prime_concatenation}. We claim $\Psi$ is a bijection. Indeed, note that we can recover $s$ (and hence $t$ and $u$) from $t \circ s \circ u$ as the first prime factor of $t \circ s \circ u$ in $\Rprime{\ell-2}$.
\end{proof}

\begin{eg}
We illustrate the map $\Psi$ from \eqref{eqn.bijection}. Consider the prime Richardson tableau
\[
\sigma = \;\begin{ytableau}
1 & 3 & 5 \\
2 & 6 \\
4 \\
7
\end{ytableau}\;,
\]
which has length $\ell = 4$ and lattice word $r = 1213124\in \Rprime{4}$. We factor $r$ as in~\eqref{delete_two_recursion}:
\[
r = 1213124 = 12\circ 1 \circ 3 \circ 12 \circ 4 = s \circ t \circ 3 \circ u \circ 4,
\]
where $s=12\in \Rprime{2}$, $t=1\in \Rbiggest{1}$, and $u=12\in\Rbiggest{2}$. Therefore
\[
\Psi(r) = t\circ s \circ u = 1\circ 12 \circ 12 = 11212\in \Rbiggest{2},
\]
which is the lattice word of the tableau
\[
\;\begin{ytableau}
1 & 2 & 4\\
3 & 5
\end{ytableau}\;
\]
of length $\ell-2$. Note that this tableau is Richardson but not prime.
\end{eg}

Recall from \cref{prime_last_two_rows} that the last two rows of every prime Richardson tableau of length $\ell\geq 2$ have length one. The bijection $\Psi$ from \eqref{eqn.bijection} restricts to a bijection between prime Richardson tableaux with a specific shape and the Richardson tableaux of the shape obtained by deleting the last two rows (see \cref{figure_delete_two} for an example):
\begin{cor}\label{cor_delete_two}
Let $\lambda = (\lambda_1, \dots, \lambda_{\ell-2}, 1, 1)$ be a partition of length $\ell \ge 2$ whose last two parts are $1$. Then the number of prime Richardson tableaux of shape $\lambda$ equals the number of Richardson tableaux of shape $(\lambda_1, \dots, \lambda_{\ell-2})$.
\end{cor}

\begin{proof}
The map $\Psi$ from \cref{delete_two} restricts to a bijection on the corresponding sets of lattice words.
\end{proof}

\begin{figure}[ht]
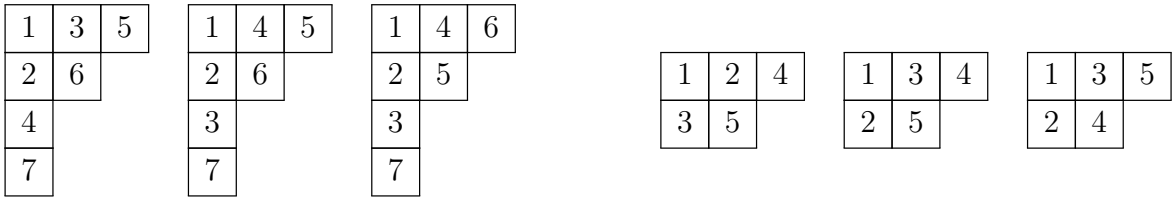

\begin{center}
\[
\;\begin{ytableau}
1 & 3 & 5 \\
2 & 6 \\
4 \\
7
\end{ytableau}\;
\hspace*{8pt}
\;\begin{ytableau}
1 & 4 & 5 \\
2 & 6 \\
3 \\
7
\end{ytableau}\;
\hspace*{8pt}
\;\begin{ytableau}
1 & 4 & 6 \\
2 & 5 \\
3 \\
7
\end{ytableau}\;
\hspace*{48pt}
\;\begin{ytableau}
1 & 2 & 4 \\
3 & 5
\end{ytableau}\;
\hspace*{8pt}
\;\begin{ytableau}
1 & 3 & 4 \\
2 & 5
\end{ytableau}\;
\hspace*{8pt}
\;\begin{ytableau}
1 & 3 & 5 \\
2 & 4
\end{ytableau}\;
\]
\caption{An illustration of \cref{cor_delete_two}. Left: prime Richardson tableaux of shape $\lambda = (3,2,1,1)$. Right: Richardson tableaux of shape $(3,2)$.}
\label{figure_delete_two}
\end{center}
\end{figure}

\begin{thm}\label{motzkin_richardson}
For all $n\in\mathbb{N}$, the number of Richardson tableaux of size $n$ equals the Motzkin number $M_n$.
\end{thm}

\begin{proof}
Let $M(x) = \sum_{n\ge 0}M_nx^n$ denote the ordinary generating function for the Motzkin numbers. It satisfies the identity (see, e.g., \cite[(1)]{donaghey_shapiro77})
\begin{align}\label{motzkin_ogf}
M(x) = 1 + xM(x) + x^2M(x)^2.
\end{align}
Note that $M(x)$ the unique solution in $\C[[x]]$ to the quadratic equation \eqref{motzkin_ogf} (the other solution given by the quadratic formula is a formal Laurent series which is not in $\C[[x]]$).

On the other hand, consider the ordinary generating functions for Richardson tableaux and prime Richardson tableaux, respectively:
\[
R(x) = \sum_{\substack{\text{Richardson}\\ \text{tableaux $\sigma$}}}x^{|\sigma|}, \quad P(x) = \sum_{\substack{\text{prime Richardson}\\ \text{tableaux $\sigma$}}}x^{|\sigma|}.
\]
The theorem follows immediately once we show $R(x) = M(x)$. It suffices to prove that $R(x)$ also satisfies the quadratic equation \eqref{motzkin_ogf}.

By the prime factorization for Richardson tableaux (cf.~\cref{prime_star}), we have
\[
R(x) = \frac{1}{1 - P(x)}.
\]
By \cref{cor_delete_two}, we have $[x^\ell]P(x) = [x^{\ell-2}]R(x)$ for all $\ell \ge 2$, and $[x]P(x) = 1$ (where $[x^i]f(x)$ denotes the coefficient of $x^i$ in $f(x)$). Hence $P(x) = x + x^2R(x)$, and plugging this into the line above gives
\[
R(x) = \frac{1}{1 - x - x^2R(x)}.
\]
Clearing denominators shows that $R(x)$ satisfies \eqref{motzkin_ogf}, as desired.
\end{proof}

It would be interesting to describe an explicit bijection proving \cref{motzkin_richardson}:
\begin{prob}\label{motzkin_problem}
Construct an explicit bijection between the set of Richardson tableaux of size $n$ and the set of Motzkin paths with $n$ steps, for all $n\in\mathbb{N}$ \textup{(}by composing the bijections implicit in the proof of \cref{motzkin_richardson}, or otherwise\textup{)}.
\end{prob}

\begin{cor}\label{proportion_richardson}
For $n\in\mathbb{N}$, let $p_n$ denote the proportion of standard tableaux of size $n$ which are Richardson. Then
\[
p_n \sim \frac{3^{n + \frac{3}{2}}e^{\frac{n}{2} - \sqrt{n} + \frac{1}{4}}}{\sqrt{2\pi}\hspace*{1pt}n^{\frac{n+3}{2}}} \quad \text{ as } n\to\infty.
\]
\end{cor}

\begin{proof}
By \cref{motzkin_richardson}, we have $p_n = \frac{M_n}{T_n}$ for all $n\in\mathbb{N}$, where $T_n$ denotes the number of standard tableaux of size $n$. It follows from \eqref{motzkin_ogf} that
\begin{align}\label{motzkin_asymptotics}
M_n \sim \frac{3^{n + \frac{3}{2}}}{2\sqrt{\pi}\hspace*{1pt}n^{\frac{3}{2}}} \quad \text{ as } n\to\infty;
\end{align}
see, e.g., \cite[Example VI.3]{flajolet_sedgewick09}, noting therein that $U_n = M_{n-1}$. On the other hand, it follows from properties of the Robinson--Schensted correspondence \cite[Corollary 7.13.9]{stanley24} that $T_n$ is the number of involutions in $S_n$. Chowla, Herstein, and Moore \cite[Theorem 8]{chowla_herstein_moore51} (cf.\ \cite[Proposition VIII.2]{flajolet_sedgewick09}) found the asymptotics of $T_n$:
\begin{align}\label{involution_asymptotics}
T_n \sim \frac{n^{\frac{n}{2}}}{\sqrt{2}\hspace*{1pt}e^{\frac{n}{2} - \sqrt{n} + \frac{1}{4}}} \quad \text{ as } n\to\infty.
\end{align}
The result follows by combining \eqref{motzkin_asymptotics} and \eqref{involution_asymptotics}.
\end{proof}


\subsection{Enumeration of Richardson tableaux of fixed shape}
The main result of this subsection is \cref{q_count}, which gives a $q$-count of the number of Richardson tableaux $\Rtabs{\lambda}$ of a fixed shape $\lambda$. We first recall some background on \boldit{$q$-analogues}, where $q$ is an indeterminate. Given $0 \le k \le n$, set
\[
\qn{n} := 1 + q + \cdots + q^{n-1}, \quad \qfac{n} := \qn{1}\qn{2}\cdots\qn{n}, \quad \qbinom{n}{k} := \frac{\qfac{n}}{\qfac{k}\qfac{n-k}}.
\]
When $q=1$, we recover $n$, $n!$, and $\binom{n}{k}$, respectively. For example, 
\begin{align*}
\qbinom{4}{2} = \frac{\qfac{4}}{\qfac{2}\qfac{2}} = \frac{\qn{4}\qn{3}}{\qn{2}} &= \frac{(1 + q + q^2 + q^3)(1 + q + q^2)}{1+q} \\
&= 1 + q + 2q^2 + q^3 + q^4.
\end{align*}

Given a standard tableau $\sigma$ of size $n$, we call $j\in [n-1]$ a \boldit{descent} of $\sigma$ if $j+1$ is in a strictly lower row than $j$, i.e., $\rowjword{\sigma}{j+1} > \rowjword{\sigma}{j}$. By definition, $n$ is never a descent of $\sigma$. We define the \boldit{major index} of $\sigma$ by
\[
\maj(\sigma) := \sum_{\text{descents $j$ of $\sigma$}}j.
\]

\begin{eg}\label{eg_maj}
As in \cref{eg_permutations_intro}, let
\[
\sigma = \;\begin{ytableau}
1 & 3 & 4 & 6 \\
2 & 7 \\
5 & 8
\end{ytableau}\;.
\]
Then the descents of $\sigma$ are $1$, $4$, $6$, and $7$, so $\maj(\sigma) = 1 + 4 + 6 + 7 = 18$.
\end{eg}

Our goal is to prove the following enumerative result:\footnote{We thank Sara Billey and Vasu Tewari for suggesting trying to find a $q$-analogue of \eqref{q_count_1}, and thank Vasu Tewari for verifying \eqref{q_count_q} experimentally.}
\begin{thm}\label{q_count}
Let $\lambda$ be a partition of length $\ell$. Then
\begin{align}\label{q_count_q}
\sum_{\sigma\in\Rtabs{\lambda}}q^{\maj(\sigma)} = q^{\sum_{2 \le i \le j \leq \ell}\lambda_i\lambda_j} \cdot \prod_{i = 1}^{\ell-1}\qbinom{\lambda_i + \lambda_{i+2} + \cdots + \lambda_{\ell}}{\lambda_{i+1} + \lambda_{i+2} + \cdots+ \lambda_{\ell}}
\end{align}
\textup{(}where the $i = \ell-1$ term of the product is $\qbinomsmall{\lambda_{\ell-1}}{\lambda_{\ell}}$\textup{)}. In particular, setting $q=1$ gives
\begin{align}\label{q_count_1}
|\Rtabs{\lambda}| = \binom{\lambda_{\ell-1}}{\lambda_\ell}\binom{\lambda_{\ell-2} + \lambda_\ell}{\lambda_{\ell-1} + \lambda_\ell}\binom{\lambda_{\ell-3} + \lambda_{\ell-1} + \lambda_\ell}{\lambda_{\ell-2} + \lambda_{\ell-1} + \lambda_\ell}\cdots\binom{\lambda_1 + \lambda_3 + \lambda_4 + \cdots + \lambda_\ell}{\lambda_2 + \lambda_3 + \lambda_4 + \cdots + \lambda_\ell}.
\end{align}
\end{thm}

\begin{eg}\label{eg_q_count}
The $8$ Richardson tableaux of shape $\lambda = (3,2,1)$ are shown below.
\[
\scalebox{0.81}{$\begin{ytableau}
1 & 2 & 4 \\
3 & 5 \\
6\end{ytableau}$}
\hspace*{9.9pt}
\scalebox{0.81}{$\begin{ytableau}
1 & 2 & 5 \\
3 & 6 \\
4\end{ytableau}$}
\hspace*{9.9pt}
\scalebox{0.81}{$\begin{ytableau}
1 & 3 & 4 \\
2 & 5 \\
6\end{ytableau}$}
\hspace*{9.9pt}
\scalebox{0.81}{$\begin{ytableau}
1 & 3 & 5 \\
2 & 4 \\
6\end{ytableau}$}
\hspace*{9.9pt}
\scalebox{0.81}{$\begin{ytableau}
1 & 3 & 5 \\
2 & 6 \\
4
\end{ytableau}$}
\hspace*{9.9pt}
\scalebox{0.81}{$\begin{ytableau}
1 & 3 & 6 \\
2 & 4 \\
5
\end{ytableau}$}
\hspace*{9.9pt}
\scalebox{0.81}{$\begin{ytableau}
1 & 4 & 5 \\
2 & 6 \\
3\end{ytableau}$}
\hspace*{9.9pt}
\scalebox{0.81}{$\begin{ytableau}
1 & 4 & 6 \\
2 & 5 \\
3\end{ytableau}$}
\hspace*{9.9pt}
\]
We calculate that
\[
\sum_{\sigma\in\Rtabs{\lambda}}q^{\maj(\sigma)} = q^{11} + q^{10} + q^{10} + q^9 + q^9 + q^8 + q^8 + q^7 = q^7\cdot\qbinom{4}{3}\cdot\qbinom{2}{1},
\]
in agreement with \cref{q_count}.
\end{eg}

\begin{eg}\label{two_row_enumeration}
According to \eqref{q_count_1}, the number of Richardson tableaux of a partition $\lambda = (\lambda_1, \lambda_2)$ of length at most two is $\binom{\lambda_1}{\lambda_2}$. We can see this explicitly as follows. A standard tableau of shape $\lambda$ is uniquely determined by the subset $I$ of entries in its second row, and by \cref{explicit}\ref{explicit_two_row}, the tableau is Richardson if and only if $I$ contains no two consecutive entries. Hence we must enumerate subsets $I$ of $\{2, 3, \dots, \lambda_1 + \lambda_2\}$ of size $\lambda_2$ with no two consecutive entries. The map
\[
I = \{i_1 < \cdots < i_{\lambda_2}\} \mapsto \{i_1 - 1 < i_2 - 2 < \cdots < i_{\lambda_2} - \lambda_2\}
\]
provides a bijection from such subsets onto the subsets of $[\lambda_1]$ of size $\lambda_2$, so there are $\binom{\lambda_1}{\lambda_2}$ such subsets.
\end{eg}

As a consequence of \cref{q_count}, we obtain a new formula for the Motzkin numbers:
\begin{cor}\label{motzkin_formula}
For all $n\in\mathbb{N}$, we have
\[
M_n = \sum_{\lambda\vdash n} \prod_{i = 1}^{\ell(\lambda)-1} \binom{\lambda_i + \lambda_{i+2} + \cdots + \lambda_{\ell(\lambda)}}{\lambda_{i+1} + \lambda_{i+2} + \cdots + \lambda_{\ell(\lambda)}}
\]
\textup{(}where the $i = \ell(\lambda)-1$ term of the product is $\binom{\lambda_{\ell(\lambda)-1}}{\lambda_{\ell(\lambda)}}$\textup{)}.
\end{cor}

\begin{proof}
This follows by combining \cref{motzkin_richardson} and \eqref{q_count_1}.
\end{proof}

Now we turn to the proof of \cref{q_count}. As in the previous section, we utilize lattice words. Given a word $r$ of length $n$, we define an \boldit{ascent} of $r$ to be an index $j \in [n-1]$ such that $r_j < r_{j+1}$. Note that the descents of every standard tableau $\sigma$ are precisely the ascents of $\rowword{\sigma}$. (For example, if $\sigma$ is as in \cref{eg_maj}, then $\rowword{\sigma} = 12113123$ has ascents $1$, $4$, $6$, and $7$.) Hence
\begin{align}\label{maj_rowword}
\maj(\sigma) = \sum_{\substack{j\in [n-1], \\ \rowjword{\sigma}{j} < \rowjword{\sigma}{j+1}}}j 
\end{align}
for all standard tableaux $\sigma$ of size $n$.
Also, let $\sumone$ denote the function on all words $r$ on the alphabet $\mathbb{Z}_{>0}$ defined by
\[
\sumone(r) := \sum_{\substack{j \ge 1,\\ r_j = 1}}(j-1). 
\]
If $r$ is the lattice word of a tableau $\sigma$, then $\sumone(r)$ is the sum of the entries in the first row minus the length of the first row. 

\begin{eg}\label{eg_maj_rowword}
Let $\sigma$ be the standard tableau from \cref{eg_maj}. Then
\[
\sumone(\rowword{\sigma}) = \sumone(12113123) = (1-1) + (3-1) + (4-1) + (6-1) = 10.
\]
We have $10=(1+3+4+6)-4$, the sum of the entries in the first row of $\sigma$ minus $\lambda_1=4$. 
Note also that
\[
\maj(\sigma) = 18 = 28 - 10 = \binom{8}{2} - \sumone(\rowword{\sigma)},
\]
which is a particular case of \cref{richardson_maj}\ref{richardson_maj_dual} below.
\end{eg}

The next lemma says that the descents of a Richardson tableau are determined by the entries in its first row. In particular, we can write its major index in terms of $\sumone$:
\begin{lem}\label{richardson_maj}
Let $\sigma$ be a Richardson tableau of size $n$.
\begin{enumerate}[label=(\roman*), leftmargin=*, itemsep=2pt]
\item\label{richardson_maj_max} For all $j\in [n-1]$, we have that $j$ is a descent of $\sigma$ if and only if $j+1$ is not in the first row of $\sigma$.
\item\label{richardson_maj_dual} We have
\begin{align}\label{maj_dual_formula}
\maj(\sigma) = \binom{n}{2} - \sumone(\rowword{\sigma}).
\end{align}
\end{enumerate}
\end{lem}

\begin{proof}
Part \ref{richardson_maj_max} follows from the definition of the Richardson property. Indeed, if $j+1$ does not appear in the first row of $\sigma$ and $j$ is not a descent (i.e., $j$ appears weakly below $j+1$), then $\sigma$ violates \cref{defn_richardson} for $j+1$. 

For part \ref{richardson_maj_dual}, by \eqref{maj_rowword} and part \ref{richardson_maj_max} we have
\begin{gather*}
\maj(\sigma) = \sum_{\substack{j\in [n-1], \\ \rowjword{\sigma}{j} < \rowjword{\sigma}{j+1}}}j = \sum_{\substack{j\in [n-1], \\ \rowjword{\sigma}{j+1} > 1}}j = \binom{n}{2} - \sum_{\substack{j\in [n-1], \\ \rowjword{\sigma}{j+1} = 1}}j = \binom{n}{2} - \sumone(\rowword{\sigma}),
\end{gather*}
as desired.
\end{proof}

\begin{rmk} 
For $\lambda\vdash n$, the maximum number of descents among all standard tableaux of shape $\lambda \vdash n$ is $n-\lambda_1$. Thus, \cref{richardson_maj}\ref{richardson_maj_max} implies that every Richardson tableau of shape $\lambda$ has the maximum number of descents. However, if a standard tableau $\sigma$ of shape $\lambda$ has the maximum number of descents, it is not necessarily Richardson. One such example is 
\[
\sigma = \;\begin{ytableau}
1 & 3 & 5 & 7 \\
2 & 4 \\
6 & 8
\end{ytableau}\;,
\]
which violates \cref{defn_richardson} for $j=8$.
\end{rmk}

We will also need the following property of $\sumone$:
\begin{lem}\label{sumone_q}
Let $a,b\ge 0$. Then
\[
\sum_rq^{\sumone(r) - \binom{a}{2}} = \qbinom{a+b}{b},
\]
where the sum is over all words $r$ of length $a+b$ with exactly $a$ $1$'s and $b$ $2$'s.
\end{lem}

\begin{proof}
An \boldit{inversion} of a word $r$ of length $[n]$ is a pair $(i,j) \in [n]^2$ such that $i < j$ and $r_i > r_j$. Let $\inv(r)$ denote the number of inversions of $r$. By \cite[Proposition 1.7.1]{stanley12}, we have
\[
\sum_rq^{\inv(r)} = \qbinom{a+b}{b},
\]
where the sum is over all words $r$ of length $a+b$ with exactly $a$ $1$'s and $b$ $2$'s. Therefore it suffices to show that for every such word $r$, we have $\inv(r) = \sumone(r) - \binom{a}{2}$. To prove this, let $j_1 < \cdots < j_a$ denote the positions of the $1$'s in $r$, so that $\sumone(r) = j_1+\cdots + j_a -a $. For all $k\in [a]$, we see that $r$ has exactly $j_k - k$ inversions of the form $(\cdot,j_k)$. We get
\[
\inv(r) = \sum_{k=1}^a \left( j_k - k\right) = \sumone(r) - \sum_{k=1}^{a} (k-1) = \sumone(r) - \binom{a}{2},
\]
as desired.
\end{proof}

Given a partition $\lambda\vdash n$, define the partition $\crop(\lambda)\vdash n-\lambda_1$ to be $(\lambda_2, \lambda_3, \dots)$. Given a standard tableau $\sigma\in\SYT(\lambda)$, let $\crop(\sigma)\in\SYT(\crop(\lambda))$ denote the standard tableau obtained from $\sigma$ by deleting the first row and renumbering the remaining entries as $1, \dots, n-\lambda_1$ (so that they remain in the same relative order). Similarly, given a word $r$ on the alphabet $\mathbb{Z}_{>0}$, we let $\crop(r)$ denote the word obtained from $r$ by deleting all the $1$'s and then decreasing all of the remaining letters by $1$. Note that $\rowword{\crop(\sigma)} = \crop(\rowword{\sigma})$. The final ingredient we will need for the proof of~\cref{q_count} is a recursive characterization of Richardson words using $\crop$:
\begin{prop}\label{richardson_chop}~
\begin{enumerate}[label=(\roman*), leftmargin=*, itemsep=2pt]
\item\label{richardson_chop_forward} If $r$ is a Richardson word, then $\crop(r)$ is a Richardson word.
\item\label{richardson_chop_backward} If $s$ is a Richardson word, then the word $r$ obtained from $s$ as follows is a Richardson word with $\crop(r) = s$:
\begin{itemize}[itemsep=2pt]
\item increase every letter of $s$ by $1$, and then replace every occurrence of $2$ by the string $12$, to form the word $t$;
\item insert an arbitrary number of $1$'s into $t$ to form the word $r$ in any way such that we do not insert a $1$ immediately before a $2$ in $t$.
\end{itemize}
Moreover, every Richardson word $r$ such that $\crop(r)=s$ is uniquely obtained in this way.
\end{enumerate}
\end{prop}

\begin{proof}
This follows from \cref{richardson_tableau_to_word}. (The requirement that we do not insert a $1$ immediately before a $2$ in $t$ ensures that $r$ is obtained uniquely.)
\end{proof}

\begin{eg}\label{eg_richardson_chop}
Let $\lambda = (4,2,2)$, so that $\crop(\lambda) = (2,2)$, and let $\sigma\in\SYT(\lambda)$ be as in \cref{eg_maj,eg_maj_rowword}. That is,
\[
\sigma = \;\begin{ytableau}
1 & 3 & 4 & 6 \\
2 & 7 \\
5 & 8
\end{ytableau}\;
\quad\text{ and }\quad
\crop(\sigma) = \;\begin{ytableau}
1 & 3 \\
2 & 4
\end{ytableau}\;.
\]
In the setting of \cref{richardson_chop}\ref{richardson_chop_backward}, we have $r = \rowword{\sigma} = 12113123$, $s = \crop(r) = 1212 = \rowword{\crop(\sigma)}$, and $t = 123123$. Notice that $r$ is obtained from $t$ by inserting two $1$'s into $t$ before the $3$.  
\end{eg}

Although we will continue working with words (rather than tableaux) to prove \cref{q_count}, we point out that \cref{richardson_chop} can be rephrased in terms of tableaux:
\begin{cor}\label{crop_tableau}
Let $\sigma$ be a standard tableau. Then $\sigma$ is Richardson if and only if:
\begin{itemize}[itemsep=2pt]
\item $\crop(\sigma)$ is Richardson; and
\item for every entry $j$ in the second row of $\sigma$, the entry $j-1$ appears in the first row.
\end{itemize}
\end{cor}

\begin{proof}
This is just \cref{richardson_chop} stated in terms of tableaux.
\end{proof}

\begin{eg}\label{eg_richardson_induction} This example demonstrates the inductive argument of our proof of~\cref{q_count}, which appears below. Recall from \cref{eg_q_count} that there are $8$ Richardson tableaux of shape $\lambda = (3,2,1)$. We explain how each such tableau (or equivalently, its lattice word $r$) is obtained from a Richardson tableau of shape $\crop(\lambda) = (2,1)$ as in \cref{richardson_chop}\ref{richardson_chop_backward}. There are two Richardson words corresponding to the shape $(2,1)$: $s=112$ and $s=121$.

First we consider the case $s=112$. Then $t=12123$, and to obtain $r$ we must insert one more $1$ into $t$ in any way such that we do not insert a $1$ immediately before a $2$ in $t$. There are four such Richardson words, each shown below along with its corresponding Richardson tableau:
\begin{align}\label{eqn.112}
\begin{array}{cccc}
\hspace*{6pt}\;\begin{ytableau}
1 & 2 & 4 \\
3 & 5 \\
6
\end{ytableau}\;\hspace*{6pt}
&
\hspace*{6pt}\;\begin{ytableau}
1 & 3 & 4 \\
2 & 5 \\
6
\end{ytableau}\;\hspace*{6pt}
&
\hspace*{6pt}\;\begin{ytableau}
1 & 3 & 5 \\
2 & 4 \\
6
\end{ytableau}\;\hspace*{6pt}
&
\hspace*{6pt}\;\begin{ytableau}
1 & 3 & 6 \\
2 & 4 \\
5
\end{ytableau}\;\hspace*{6pt}\\[24pt]
r=112123
&
r=121123
&
r=121213
&
r=121231.
\end{array}
\end{align}
Similarly, for the case $s = 121$  we have that $t = 12312$ and we obtain
\begin{align}\label{eqn.121}
\begin{array}{cccc}
\hspace*{6pt}\;\begin{ytableau}
1 & 2 & 5 \\
3 & 6 \\
4
\end{ytableau}\;\hspace*{6pt}
&
\hspace*{6pt}\;\begin{ytableau}
1 & 3 & 5 \\
2 & 6 \\
4
\end{ytableau}\;\hspace*{6pt}
&
\hspace*{6pt}\;\begin{ytableau}
1 & 4 & 5 \\
2 & 6 \\
3
\end{ytableau}\;\hspace*{6pt}
&
\hspace*{6pt}\;\begin{ytableau}
1 & 4 & 6 \\
2 & 5 \\
3
\end{ytableau}\;\hspace*{6pt}\\[24pt]
r=112312
&
r=121312
&
r=123112
&
r=123121
\end{array}
\end{align}

This process partitions the set of Richardson tableaux of shape $\lambda$ into two subsets of equal size on which $\sumone$ is particularly well-behaved. Indeed, we have
\[
\sum_{\text{$r$ in~\eqref{eqn.112}}} q^{\sumone(r)} = q^4(1+q+q^2+q^3) \quad \text{ and } \quad 
\sum_{\text{$r$ in~\eqref{eqn.121}}} q^{\sumone(r)} = q^5(1+q+q^2+q^3).
\]
The reader can confirm that this is a special case of the formula~\eqref{sumone_to_multiply} in the proof of~\cref{q_count} below.
\end{eg}

\begin{proof}[Proof of \cref{q_count}]
Let $n = |\lambda|$ and $\ell = \ell(\lambda)$, and for a partition $\mu$ let $\Rwords{\mu} := \{\rowword{\sigma} \mid \sigma\in\Rtabs{\mu}\}$ denote the set of all lattice words of Richardson tableaux of shape $\mu$. By \eqref{maj_dual_formula}, it suffices to show that
\[
\sum_{r\in\Rwords{\lambda}}q^{\sumone(r)} = q^{\binom{n}{2} - \sum_{2 \le i \le j \le \ell}\lambda_i\lambda_j}\cdot\prod_{i=1}^{\ell-1}\qinvbinom{\lambda_i + \lambda_{i+2} + \cdots + \lambda_\ell}{\lambda_{i+1} + \lambda_{i+2} + \cdots+\lambda_\ell}.
\]
Using the identities $\qinvbinomsmall{b}{a} = q^{-a(b-a)}\qbinomsmall{b}{a}$ and $\binom{n}{2} = \sum_{i=1}^\ell \binom{\lambda_i}{2} + \sum_{1\leq i<j\leq \ell} \lambda_i\lambda_j$ (from, e.g., \cite[Lemma 2.9 and Proposition 4.7(ii)]{karp_thomas}), a calculation shows that this is equivalent to
\begin{align}\label{sumone_count}
\sum_{r\in\Rwords{\lambda}}q^{\sumone(r)} = q^{\sum_{i=1}^\ell\binom{\lambda_i}{2}} \cdot \prod_{i=1}^{\ell-1}\qbinom{\lambda_i + \lambda_{i+2} + \cdots +\lambda_\ell}{\lambda_{i+1} + \lambda_{i+2} + \cdots+\lambda_\ell}.
\end{align}
We prove \eqref{sumone_count} by induction on $n$. The base case is $n=0$, whence $\lambda = \varnothing$; then \eqref{sumone_count} states that $1 = 1$.

Now suppose that $n\ge 1$ and that \eqref{sumone_count} holds for partitions of size at most $n-1$. In particular, we have
\begin{align}\label{sumone_induction}
\sum_{s\in\Rwords{\crop(\lambda)}}q^{\sumone(s)} = q^{\sum_{i=2}^\ell\binom{\lambda_i}{2}} \cdot \prod_{i=2}^{\ell-1}\qbinom{\lambda_i + \lambda_{i+2} + \cdots +\lambda_\ell}{\lambda_{i+1} + \lambda_{i+2} + \cdots+\lambda_\ell}.
\end{align}
Now fix $s\in\Rwords{\crop(\lambda)}$, and let $\Rwords{\lambda;s} := \{r\in\Rwords{\lambda} \mid \crop(r) = s\}$. We obtain the partition $\Rwords{\lambda} = \bigsqcup_{s\in \Rwords{\crop(\lambda)}} \Rwords{\lambda;s}$. In particular, it suffices to show that
\begin{align}\label{sumone_to_multiply}
\sum_{r\in\Rwords{\lambda;s}}q^{\sumone(r)} = q^{\sumone(s) + \binom{\lambda_1}{2}}\qbinom{n-\lambda_2}{n-\lambda_1},
\end{align}
since summing \eqref{sumone_to_multiply} over all $s\in\Rwords{\crop(\lambda)}$ and plugging in \eqref{sumone_induction} yields \eqref{sumone_count}.

To prove \eqref{sumone_to_multiply}, we use \cref{richardson_chop}\ref{richardson_chop_backward} to construct the elements of $\Rwords{\lambda;s}$ starting from $s$. First note that if $s$ has $1$'s in positions $j_1 < \cdots < j_{\lambda_2}$, then $t$ (as defined in \cref{richardson_chop}\ref{richardson_chop_backward}) has $1$'s in positions $j_1 < j_2 + 1 < \cdots < j_{\lambda_2} + \lambda_2 - 1$, so
\begin{align}\label{sumone_t_s}
\sumone(t) = \sumone(s) + \binom{\lambda_2}{2}.
\end{align}

For each $r\in\Rwords{\lambda;s}$, let $\tilde{r}$ denote the word obtained from $r$ by replacing every substring $12$ and every letter $j\ge 3$ by $2$.  For example, if $r=12113123$ then $\tilde{r} = 211222$. The map $r\mapsto \tilde{r}$ defines a bijection between $\Rwords{\lambda;s}$ and the set of all words of length $n - \lambda_2$ with exactly $\lambda_1 - \lambda_2$ many $1$'s and $n - \lambda_1$ many $2$'s. Summing over all such $\tilde{r}$'s, we have
\begin{align}\label{sumone_q_induction}
\sum_{\tilde{r}}q^{\sumone(\tilde{r}) - \binom{\lambda_1 - \lambda_2}{2}} = \qbinom{n-\lambda_2}{n-\lambda_1}
\end{align}
by \cref{sumone_q}.

Now we claim that
\begin{align}\label{sumone_r}
\sumone(r) = \sumone(\tilde{r}) + \sumone(t) + \lambda_2(\lambda_1 - \lambda_2).
\end{align}
Note that plugging \eqref{sumone_t_s} into \eqref{sumone_r} and then into \eqref{sumone_q_induction} yields
\[
\sum_{r\in\Rwords{\lambda;s}}q^{\sumone(r) - \sumone(s) - \binom{\lambda_2}{2} - \lambda_2(\lambda_1 - \lambda_2) - \binom{\lambda_1 - \lambda_2}{2}} = \qbinom{n-\lambda_2}{n-\lambda_1},
\]
which we see is equivalent to \eqref{sumone_to_multiply} using the identity $\binom{\lambda_2}{2} + \lambda_2(\lambda_1-\lambda_2) + \binom{\lambda_1-\lambda_2}{2} = \binom{\lambda_1}{2}$. Therefore to complete the proof, it suffices to prove \eqref{sumone_r}.

To this end, we recall from \cref{richardson_chop} that we obtain each word $r\in\Rwords{\lambda;s}$ by using $s$ to produce the word $t$ (by replacing every $2$ with the string $12$), and then inserting $\lambda_1-\lambda_2$ many $1$'s into $t$, but not before a $2$.  Thus, there are two types of occurrences of the digit $1$ in the word $r$: the $1$'s appearing in $t$ (which we call \emph{old $1$'s}) and the $1$'s inserted into $t$ (which we call \emph{new $1$'s}). That is, a $1$ in $r$ is old if it appears immediately before a $2$, and is new otherwise. By construction, $r$ has $\lambda_2$ old $1$'s and $\lambda_1 - \lambda_2$ new $1$'s. For example, in the setup of \cref{eg_richardson_chop}, we have $s = 1212$, $t = 123123$, and $r = \underline{1}2\overline{1}\overline{1}3\underline{1}23$, where $\underline{1}$ denotes an old $1$ and $\overline{1}$ denotes a new $1$.

Note that the old $1$'s in $r$ correspond to the $1$'s in $t$, while the new $1$'s in $r$ correspond to the $1$'s in $\tilde{r}$. To see how this leads to \eqref{sumone_r} (with the additional term $\lambda_2(\lambda_1 - \lambda_2)$), consider a pair of $1$'s in $r$, one old and one new (of which there are $\lambda_2(\lambda_1 - \lambda_2)$ in total). If the new $1$ appears before the old $1$, this shifts the position of the old $1$ one space to the right relative to its position in $t$, thereby adding one to the count of $\sumone$. On the other hand, if the old $1$ appears before the new $1$, this shifts the position of the new $1$ one space to the right relative to its position in $\tilde{r}$, thereby also adding one to the count of $\sumone$. In total we add $\lambda_2(\lambda_1 - \lambda_2)$ to the count of $\sumone$, which proves \eqref{sumone_r}, as claimed.
\end{proof}


\subsection{Enumeration via generating functions}
In this subsection, we study the multivariate generating functions for Richardson words and prime Richardson words. We show that the generating functions satisfy simple recursive formulas and use these to give a second proof of the formula~\eqref{q_count_1} for the number of Richardson tableaux of shape $\lambda$. It would be interesting to generalize these arguments to prove the $q$-analogue \eqref{q_count_q} using generating functions.

Our general setup is as follows. Given a set $S$ of finite words on the alphabet $\mathbb{Z}_{>0}$, define the \boldit{generating function}
\[
\Phi_S = \Phi_S(x_1, x_2, \dots) := \sum_{s\in S}\prod_{i\ge 1}x_i^{\# \text{$i$'s in $s$}}
\]
in the ring of formal power series. Given a sequence $\alpha = (\alpha_1, \dots, \alpha_\ell)$ of nonnegative integers, we set $\mathbf{x}^\alpha := x_1^{\alpha_1}x_2^{\alpha_2}\cdots x_\ell^{\alpha_\ell}$, and we let $[\mathbf{x}^\alpha]$ denote the \boldit{coefficient operator} which extracts the coefficient of $\mathbf{x}^\alpha$ from a formal power series.

Let $\ogfbiggest{\ell}$ denote the generating function of Richardson words on the alphabet $[\ell]$, which is characterized as follows:
\[
[\mathbf{x}^\alpha]\ogfbiggest{\ell} = \begin{cases}
|\Rtabs{\alpha}|, & \text{ if $\alpha$ is a partition of length at most $\ell$}; \\
0, & \text{ otherwise}
\end{cases}
\]
for all $\alpha$. We will also study $\ogfprime{\ell}$, the generating function for prime Richardson words with largest letter $\ell$.

\begin{thm}\label{richardson_ogf}
The generating functions $\big(\ogfprime{\ell}\big)_{\hspace*{-1pt}\ell\ge 1}$ and $\big(\ogfbiggest{\ell}\big)_{\hspace*{-1pt}\ell\ge 0}$ are uniquely determined by the recurrences
\begin{align}\label{recurrence_prime}
\ogfprime{\ell} = \ogfprime{\ell-1}\ogfbiggest{\ell-2}x_\ell \quad \text{ for all } \ell\ge 2
\end{align}
and
\begin{align}\label{recurrence_biggest}
\ogfbiggest{\ell} = \frac{\ogfbiggest{\ell-1}}{1 - \ogfprime{\ell}\ogfbiggest{\ell-1}} \quad \text{ for all } \ell \ge 1,
\end{align}
with initial conditions
\[
\ogfprime{1} = x_1 \quad \text{ and } \quad \ogfbiggest{0} = 1.
\]
In particular, the generating functions $\big(\ogfprime{\ell}\big)_{\hspace*{-1pt}\ell\ge 1}$ and $\big(\ogfbiggest{\ell}\big)_{\hspace*{-1pt}\ell\ge 0}$ are all rational.
\end{thm}

\begin{proof}
The recurrence \eqref{recurrence_prime} follows from \cref{prime_recursion}. Now for any set $S$ of finite words such that every element of $S^*$ is obtained uniquely as a concatenation of words in $S$, we have $\Phi_{S^*} = \frac{1}{1 - \Phi_S}$. Hence by \cref{prime_star} we obtain
\[
\ogfbiggest{i} = \frac{1}{1 - (\ogfprime{1} + \cdots + \ogfprime{i})} \quad \text{ for all } i \ge 1.
\]
Substituting this into \eqref{recurrence_biggest}, we see that \eqref{recurrence_biggest} is equivalent to
\[
\frac{1}{1 - (\ogfprime{1} + \cdots + \ogfprime{\ell})} = \frac{\frac{1}{1 - (\ogfprime{1} + \cdots + \ogfprime{\ell-1})}}{1 - \frac{\ogfprime{\ell}}{1 - (\ogfprime{1} + \cdots + \ogfprime{\ell-1})}},
\]
which we can check is true by clearing denominators. Finally, we can verify the initial conditions directly.
\end{proof}

\begin{eg}\label{eg_richardson_ogf}
Using \cref{richardson_ogf}, we calculate that
\begin{align*}
\ogfprime{1} &= x_1, & \ogfbiggest{0} &= 1, \\[4pt]
\ogfprime{2} &= x_1x_2, & \ogfbiggest{1} &= \frac{1}{1-x_1}, \\[4pt]
\ogfprime{3} &= \frac{x_1x_2x_3}{1-x_1}, & \ogfbiggest{2} &= \frac{1}{1-x_1 - x_1x_2}, \\[4pt]
\ogfprime{4} &= \frac{x_1x_2x_3x_4}{(1-x_1)(1-x_1-x_1x_2)}, & \ogfbiggest{3} &= \frac{1}{1-x_1 - x_1x_2 - \frac{x_1x_2x_3}{1-x_1}},
\end{align*}
etc.
\end{eg}

\begin{rmk}
We can rephrase \cref{delete_two} as the statement that
\[
\ogfprime{\ell} = (\ogfbiggest{\ell-2} - \ogfbiggest{\ell-3})x_{\ell-1}x_\ell \quad \text{ for all } \ell\ge 2,
\]
where $\ogfbiggest{-1} := 0$.
\end{rmk}

As a demonstration of the utility of these generating functions, we present an alternative proof of the formula \eqref{q_count_1} for the number of Richardson tableaux of shape $\lambda$. Our main tool will be the negative binomial theorem:
\begin{align}\label{negative_binomial}
(1-x)^{-(b+1)} = \sum_{i\ge 0}\binom{b+i}{b}x^i \quad \text{ for all } b\in\mathbb{N}.
\end{align}

\begin{lem}\label{count_lemma1}
For all $a,b\ge 0$ and $i\ge 2$, we have
\[
[x_i^a]\ogfbiggest{i}^{b+1} = \binom{a+b}{b}\ogfprime{i-1}^a\ogfbiggest{i-2}^a\ogfbiggest{i-1}^{a+b+1},
\]
where above, the operator $[x_i^a]$ extracts the coefficient of $x_i^a$ \textup{(}treating the variables $x_j$ for $j\neq i$ as constants\textup{)}.
\end{lem}

\begin{proof}
By \eqref{recurrence_biggest} and \eqref{recurrence_prime}, we have
\[
\ogfbiggest{i}^{b+1} = \bigg(\frac{\ogfbiggest{i-1}}{1 - \ogfprime{i}\ogfbiggest{i-1}}\bigg)^{\hspace*{-2pt}b+1} = \ogfbiggest{i-1}^{b+1} \cdot (1 - \ogfprime{i-1}\ogfbiggest{i-2}\ogfbiggest{i-1}x_i)^{-(b+1)}.
\]
Note that in the right-hand side above, $x_i$ does not appear in any of the generating functions. Therefore we can extract the coefficient of $x_i^a$ via \eqref{negative_binomial}.
\end{proof}

\begin{lem}\label{count_lemma2}
For all $a,b,c,d\ge 0$ and $i\ge 2$, we have
\[
[x_i^{a+b}]\big(\ogfprime{i}^b\ogfbiggest{i-1}^c\ogfbiggest{i}^{d+1}\big) = \binom{a+d}{d}\ogfprime{i-1}^{a+b}\ogfbiggest{i-2}^{a+b}\ogfbiggest{i-1}^{a+c+d+1},
\]
where above, the operator $[x_i^{a+b}]$ extracts the coefficient of $x_i^{a+b}$ \textup{(}treating the variables $x_j$ for $j\neq i$ as constants\textup{)}.
\end{lem}

\begin{proof}
By \eqref{recurrence_prime}, we have
\begin{align*}
[x_i^{a+b}]\big(\ogfprime{i}^b\ogfbiggest{i-1}^c\ogfbiggest{i}^{d+1}\big) &= [x_i^{a+b}]\big((\ogfprime{i-1}\ogfbiggest{i-2}x_i)^b\ogfbiggest{i-1}^c\ogfbiggest{i}^{d+1}\big) \\
&= \ogfprime{i-1}^b\ogfbiggest{i-2}^b\ogfbiggest{i-1}^c[x_i^a]\ogfbiggest{i}^{d+1}.
\end{align*}
The result then follows by applying \cref{count_lemma1}.
\end{proof}

\begin{proof}[Proof of \eqref{q_count_1}]
The number of Richardson tableaux of shape $\lambda$ is $[\mathbf{x}^\lambda]\ogfbiggest{\ell}$. By repeatedly applying \cref{count_lemma2}, we get
\begin{align*}
[\mathbf{x}^\lambda]\ogfbiggest{\ell}&= [x_1^{\lambda_1}\cdots x_\ell^{\lambda_\ell}]\ogfbiggest{\ell} \\[2pt]
&= [x_1^{\lambda_1}\cdots x_{\ell-1}^{\lambda_{\ell-1}}] \left(\ogfprime{\ell-1}^{\lambda_\ell}\ogfbiggest{\ell-2}^{\lambda_\ell}\ogfbiggest{\ell-1}^{\lambda_\ell+1} \right) \\[2pt]
&= \binom{\lambda_{\ell-1}}{\lambda_\ell}\cdot [x_1^{\lambda_1}\cdots x_{\ell-2}^{\lambda_{\ell-2}}] \left( \ogfprime{\ell-2}^{\lambda_{\ell-1}}\ogfbiggest{\ell-3}^{\lambda_{\ell-1}}\ogfbiggest{\ell-2}^{\lambda_{\ell-1}+\lambda_\ell+1} \right) \\[2pt]
&= \binom{\lambda_{\ell-1}}{\lambda_\ell}\binom{\lambda_{\ell-2} + \lambda_\ell}{\lambda_{\ell-1} + \lambda_\ell}\cdot [x_1^{\lambda_1}\cdots x_{\ell-3}^{\lambda_{\ell-3}}] \left( \ogfprime{\ell-3}^{\lambda_{\ell-2}}\ogfbiggest{\ell-4}^{\lambda_{\ell-2}}\ogfbiggest{\ell-3}^{\lambda_{\ell-2}+\lambda_{\ell-1}+\lambda_\ell+1} \right) \\[2pt]
&= \cdots \\[2pt]
&= \binom{\lambda_{\ell-1}}{\lambda_\ell}\binom{\lambda_{\ell-2} + \lambda_\ell}{\lambda_{\ell-1} + \lambda_\ell} \cdots \binom{\lambda_2 + \lambda_4 + \cdots + \lambda_\ell}{\lambda_3 + \lambda_4 + \cdots + \lambda_\ell}\cdot [x_1^{\lambda_1}]\left(\ogfprime{1}^{\lambda_2}\ogfbiggest{0}^{\lambda_2}\ogfbiggest{1}^{\lambda_2 + \lambda_3 + \cdots + \lambda_\ell+1} \right).
\end{align*}
Now by \cref{eg_richardson_ogf} and \eqref{negative_binomial}, we have
\begin{align*}
[x_1^{\lambda_1}]\left(\ogfprime{1}^{\lambda_2}\ogfbiggest{0}^{\lambda_2}\ogfbiggest{1}^{\lambda_2 + \lambda_3 + \cdots + \lambda_\ell+1}\right) &= [x_1^{\lambda_1}]\left(x_1^{\lambda_2}(1-x_1)^{-(\lambda_2 + \lambda_3 + \cdots + \lambda_\ell+1)}\right) \\
&= \binom{\lambda_1 + \lambda_3 + \cdots + \lambda_\ell}{\lambda_2 + \lambda_3 + \cdots + \lambda_\ell},
\end{align*}
completing the proof.
\end{proof}


\section{\texorpdfstring{$L$-slide}{L-slide} characterization of Richardson tableaux}\label{sec_slides}

\noindent The goal of this section is to prove \cref{thm-Lslides}, which characterizes Richardson tableaux in terms of evacuation slides. Recall that an evacuation slide deletes the entry in the box in the upper-left corner of a standard tableau $\sigma$, and slides this empty box through $\sigma$ until it reaches an outer corner. We call such a slide an \boldit{$L$-slide} if it is $L$-shaped, i.e., the empty box slides from the upper-left corner down the first column until it reaches some row $k$ with an outer corner, and then slides right across row $k$ to the outer corner. For example, the slide appearing in \cref{eg.slide} is {\itshape not} an $L$-slide, because it leaves the first column in the first row, but terminates in the fourth row.
\begin{eg}\label{eg_Lslide}
Let $\sigma$ be the Richardson tableau from \cref{eg_prime_decomposition}. Then the first evacuation slide of $\sigma$ is an $L$-slide ending in row $k = 3$, as highlighted in yellow below.
\[
\sigma = \;\begin{ytableau}*(yellow)1&4&8&9&11\\*(yellow)2&5&10\\*(yellow)3&*(yellow)6&*(yellow)12\\7\end{ytableau}\;
\]
In fact, every evacuation slide of $\sigma$ is an $L$-slide. By \cref{thm-Lslides} below, this is equivalent to the fact that $\sigma$ is Richardson. Below, each evacuation slide is shown in yellow, with boxes of $\sigma^{\vee}$ highlighted in green.
\begin{align*}
\sigma & \makebox[0pt][l]{$~=~$}\hphantom{{}\rightsquigarrow{}} \;\scalebox{0.72}{$\begin{ytableau}*(yellow)1&4&8&9&11\\*(yellow)2&5&10\\*(yellow)3&*(yellow)6&*(yellow)12\\7\end{ytableau}$}\;
\rightsquigarrow
\;\scalebox{0.72}{$\begin{ytableau}*(yellow)2&4&8&9&11\\*(yellow)3&*(yellow)5&*(yellow)10\\6&12&*(green)12\\7\end{ytableau}$}\;
\rightsquigarrow
\;\scalebox{0.72}{$\begin{ytableau}*(yellow)3&*(yellow)4&*(yellow)8&*(yellow)9&*(yellow)11\\5&10&*(green)11\\6&12&*(green)12\\7\end{ytableau}$}\;
\rightsquigarrow
\;\scalebox{0.72}{$\begin{ytableau}*(yellow)4&8&9&11&*(green)10\\*(yellow)5&10&*(green)11\\*(yellow)6&12&*(green)12\\*(yellow)7\end{ytableau}$}\;
\rightsquigarrow
\;\scalebox{0.72}{$\begin{ytableau}*(yellow)5&8&9&11&*(green)10\\*(yellow)6&10&*(green)11\\*(yellow)7&*(yellow)12&*(green)12\\*(green)9\end{ytableau}$}\;\\[6pt]
&\rightsquigarrow
\;\scalebox{0.72}{$\begin{ytableau}*(yellow)6&8&9&11&*(green)10\\*(yellow)7&*(yellow)10&*(green)11\\12&*(green)8&*(green)12\\*(green)9\end{ytableau}$}\;
\rightsquigarrow
\;\scalebox{0.72}{$\begin{ytableau}*(yellow)7&*(yellow)8&*(yellow)9&*(yellow)11&*(green)10\\10&*(green)7&*(green)11\\12&*(green)8&*(green)12\\*(green)9\end{ytableau}$}\;
\rightsquigarrow
\;\scalebox{0.72}{$\begin{ytableau}*(yellow)8&*(yellow)9&*(yellow)11&*(green)6&*(green)10\\10&*(green)7&*(green)11\\12&*(green)8&*(green)12\\*(green)9\end{ytableau}$}\;
\rightsquigarrow
\;\scalebox{0.72}{$\begin{ytableau}*(yellow)9&11&*(green)5&*(green)6&*(green)10\\*(yellow)10&*(green)7&*(green)11\\*(yellow)12&*(green)8&*(green)12\\*(green)9\end{ytableau}$}\;
\rightsquigarrow
\;\scalebox{0.72}{$\begin{ytableau}*(yellow)10&*(yellow)11&*(green)5&*(green)6&*(green)10\\12&*(green)7&*(green)11\\*(green)4&*(green)8&*(green)12\\*(green)9\end{ytableau}$}\;\\[6pt]
&\rightsquigarrow
\;\scalebox{0.72}{$\begin{ytableau}*(yellow)11&*(green)3&*(green)5&*(green)6&*(green)10\\*(yellow)12&*(green)7&*(green)11\\*(green)4&*(green)8&*(green)12\\*(green)9\end{ytableau}$}\;
\rightsquigarrow
\;\scalebox{0.72}{$\begin{ytableau}*(yellow)12&*(green)3&*(green)5&*(green)6&*(green)10\\*(green)2&*(green)7&*(green)11\\*(green)4&*(green)8&*(green)12\\*(green)9\end{ytableau}$}\;
\rightsquigarrow
\;\scalebox{0.72}{$\begin{ytableau}*(green)1&*(green)3&*(green)5&*(green)6&*(green)10\\*(green)2&*(green)7&*(green)11\\*(green)4&*(green)8&*(green)12\\*(green)9\end{ytableau}$}\; = \sigma^\vee
\end{align*}
Notice that this calculation of $\sigma^\vee$ agrees with \cref{eg_evacuation_concatenation}.
\end{eg}

\begin{lem}\label{lem.Lshape}
Let $\sigma$ be a Richardson tableau. Then the first evacuation slide of $\sigma$ is an $L$-slide.
\end{lem}

\begin{proof}
Let $\sigma_{i,j}$ denote the entry of $\sigma$ in row $i$ and column $j$, where we set $\sigma_{i,j} := \infty$ if $j > \lambda_i$. Take $k\ge 1$ minimal such that $\sigma_{k+1,1} \ge \sigma_{k,2}$, so that
\[
\sigma_{i+1,1} < \sigma_{i,2} \quad \text{ for all $1 \le i \le k-1$}.
\]
That is, the first evacuation slide of $\sigma$ begins by moving the empty box down the first column until it reaches row $k$. We claim that the slide continues by moving right across row $k$ until it reaches an outer corner, where it terminates, whence it is an $L$-slide. To prove this, it suffices to show that
\begin{align}\label{eqn.claim}
\sigma_{k+1, j} > \sigma_{k, j+1} \quad \text{ for all $1\leq j\leq \lambda_{k+1}$}.
\end{align}

We prove \eqref{eqn.claim} by induction on $j$. The base case $j=1$ follows by the definition of $k$. Now suppose that $2 \le j \le \lambda_{k+1}$ and \eqref{eqn.claim} holds for $j-1$. By \cref{defn_richardson} for the entry $\sigma_{k+1,j}$, we have $\sigma_{k+1,j-1}<\sigma_{k,c}<\sigma_{k+1,j}$, where $\sigma_{k,c}$ denotes the largest entry of $\sigma[\sigma_{k+1,j}]$ in row $k$. By induction we have $\sigma_{k+1, j-1}>\sigma_{k,j}$, so $\sigma_{k,j} < \sigma_{k,c}$. This implies that $c \ge j+1$, so $\sigma_{k,j+1} \le \sigma_{k,c} < \sigma_{k+1,j}$, as desired.  This completes the proof of \eqref{eqn.claim}.
\end{proof}

\begin{thm}\label{thm-Lslides}
The standard tableau $\sigma$ is Richardson if and only if each evacuation slide of $\sigma$ \textup{(}used to calculate $\sigma^\vee$\textup{)} is an $L$-slide.
\end{thm}

\begin{proof}
$(\Rightarrow)$ Suppose that $\sigma$ is Richardson.  By \cref{lem.Lshape}, the first evacuation slide of $\sigma$ is an $L$-slide.  The desired statement now follows by induction provided the tableau resulting from this slide (and decreasing every entry by $1$) is Richardson. By \cref{lem.induction}, this tableau is $(\overline{\sigma^{\vee}})^\vee$. It is Richardson since the operations $\tau \mapsto \tau^\vee$ and $\tau \mapsto \bar{\tau}$ on standard tableaux preserve the Richardson property (by \cref{evacuation_closed} and by definition, respectively).

$(\Leftarrow)$ Suppose that each evacuation slide of $\sigma$ is an $L$-slide.  This implies the same is true for $\bar{\sigma}$, since the $j$th evacuation slide of $\bar{\sigma}$ is just the $j$th evacuation slide of $\sigma$, ignoring the largest entry $n = |\sigma|$. Thus, by induction, we may assume that $\bar{\sigma}$ is Richardson.  It therefore suffices to prove that the Richardson condition in \cref{defn_richardson} holds for the entry $n$ of $\sigma$. 

If $n$ appears in the first row of $\sigma$ the Richardson condition holds vacuously, so we may assume that $n$ appears in row $k>1$. Let $b$ denote the last entry of $\sigma$ in row $k-1$. We need to show that $a<b$ for every entry $a<n$ of $\sigma$ appearing in row $k$ or below. 

Since each evacuation slide of $\sigma$ is an $L$-slide, the path traced by any fixed entry through the evacuation process is $L$-shaped.  In particular, $b$ will continue to label the box at the end of row $k$ until it slides into the first column. Furthermore, $b$ cannot slide into the first column before $n$ does, as such a slide would result in a non-partition shape where row $k-1$ has length $1$ but row $k$ has length at least $2$.

Because $n$ is the largest entry in $\sigma$, once $n$ slides into the first column (in row $k$) it must be the last entry in the first column.  This implies that $a$ already slid into the first column and up in a previous step. That is, after one of the evacuation slides, $a$ was the first entry in row $k-1$ and $b$ was the last entry. Since every row is increasing, we have $a < b$.
\end{proof}


\section{Reading words of standard tableaux}\label{sec_reading}

\noindent In this section we study two permutations (namely $v_\sigma$ and $w_\sigma$) associated to every standard tableau $\sigma$, and use them to characterize when $\sigma$ is Richardson. Recall from \eqref{lambda_dimension} that $n(\lambda)=\sum_{i=1}^{\ell(\lambda)} (i-1)\lambda_i$. 

\begin{thm}\label{thm-lengths}
Let $\sigma\in\SYT(\lambda)$. Then
\begin{align}\label{dimension_bound}
\ell(w_\sigma) - \ell(v_\sigma) \ge n(\lambda),
\end{align}
where equality holds if and only if $\sigma$ is Richardson.
\end{thm}

We will prove \cref{thm-lengths} at the end of this section. We first recall the definition of $v_\sigma$ and $w_\sigma$ from \cref{sec_intro}. Let $n = |\sigma|$, and recall the notation for the symmetric group $S_n$ from \cref{sec_background_Sn}. The \boldit{reading word} of $\sigma$ is the permutation in $S_n$ formed (in one-line notation) by reading its entries row by row from bottom to top, and within each row from left to right. The \boldit{top-down reading word} of $\sigma$ is defined similarly, except that we read the rows from top to bottom instead of bottom to top. We define $v_\sigma, w_\sigma\in S_n$ to be the unique permutations such that
\begin{itemize}[itemsep=2pt]
\item $v_\sigma^{-1}$ is the top-down reading word of $\sigma^{\vee}$; and
\item $w_0w_{\sigma}^{-1}w_0$ is the reading word of $\sigma$.
\end{itemize}
For further background on reading words, we refer to \cite[Appendix A]{stanley24} and \cite[Section 2]{fulton97}. It will follow from \cref{w_PR} below that $w_\sigma$ is the permutation considered by Pagnon and Ressayre in \cite{pagnon_ressayre06} (also denoted $w_\sigma$).

\begin{rmk}\label{reading_word_bruhat}
Since the entries of $\sigma$ increase along rows, we see that $v_\sigma,w_\sigma\in\mincoset{S_n}{\lambda}$, i.e., $v_\sigma$ and $w_\sigma$ are minimal-length coset representatives of $S_\lambda\backslash S_n$. Moreover, it will follow from \cref{main} in the next section that $v_\sigma \le w_\sigma$ in Bruhat order. We point out that our proof of \cref{main} is geometric. It would be interesting to give a combinatorial proof that $v_\sigma \le w_\sigma$.
\end{rmk}

\begin{eg}\label{eg_permutations}
Let $\lambda = (4,2,2)$ and $n = |\lambda| = 8$. We adopt the setup of \cref{eg_permutations_intro}:
\[
\sigma = \;\begin{ytableau}
1 & 3 & 4 & 6 \\
2 & 7 \\
5 & 8
\end{ytableau}\;
,\quad
\sigma^{\vee} = 
\;\begin{ytableau}
1 & 4 & 6 & 7 \\
2 & 5 \\
3 & 8
\end{ytableau}\;
,\quad
v_{\sigma} = 15726348,
\quad
w_{\sigma} = 75182364.
\]
Notice that $v_\sigma,w_\sigma\in\mincoset{S_n}{\lambda}$ (meaning that each of the substrings $1234$, $56$, and $78$ appear in increasing order) and $v_\sigma \leq w_{\sigma}$. We also have
\[
\ell(w_{\sigma}) - \ell(v_\sigma) = 15-9 = 6 = n(\lambda),
\]
which by \cref{thm-lengths} is equivalent to the fact that $\sigma$ is Richardson.
\end{eg}

We have the following explicit relation between $v_{\sigma^\vee}$ and $w_\sigma$:
\begin{prop}\label{v_w_relation}
Let $\sigma\in\SYT(\lambda)$, where $\lambda\vdash n$. Then
\begin{align}\label{equation_v_w_relation}
v_{\sigma^\vee} = w_{0,\lambda}w_\sigma w_0 \quad \text{ in } S_n.
\end{align}
\end{prop}

\begin{proof}
Multiplying on the right by $w_{0,\lambda}w_0$ takes the top-down reading word of $\sigma$, which is $v_{\sigma^{\vee}}^{-1}$, to the reading word of $\sigma$, which is $w_0w_\sigma^{-1}w_0$. That is, $v_{\sigma^\vee}^{-1} w_{0,\lambda}w_0 = w_0w_\sigma^{-1}w_0$, which is equivalent to \eqref{equation_v_w_relation}.
\end{proof}

\begin{rmk}\label{RS}
Let $\sigma\in\SYT(\lambda)$. Then we can recover $\sigma$ from $w_\sigma$ using the Robinson--Schensted bijection (as defined in, e.g., \cite[Section 7.11]{stanley24}):
\begin{align}\label{RS_equation}
\RS(w_0w_\sigma^{-1}w_0) = (\sigma, \pi^\vee) \quad \text{ and } \quad \RS(w_\sigma) = (\pi, \sigma^{\vee}),
\end{align}
where $\pi$ is the unique standard tableau of shape $\lambda$ whose top-down reading word is the identity permutation. Indeed, the first equality of \eqref{RS_equation} is a well-known property of reading words \cite[Lemma A1.2.4]{stanley24}. The second equality of \eqref{RS_equation} then follows from standard properties of $\RS(\cdot)$ \cite[Theorem 7.13.1 and Corollary A1.4.4]{stanley24}; an alternative proof was given by Pagnon and Ressayre \cite[Theorem 2]{pagnon_ressayre06}.

On the other hand, we cannot in general recover $\sigma$ from $v_\sigma$ (unless we also know the shape of $\sigma$). For example, we have $v_\sigma = 12 \in S_2$ both when $\sigma = \ytableausmall{1 & 2}$ and when $\sigma = \ytableausmall{1 \\ 2}$.
\end{rmk}

Let $w\in S_n$. The \boldit{Lehmer code} of $w$ (see \cite[Section 1.3]{stanley12}) is the sequence $\code{w} = (\codej{w}{1}, \dots, \codej{w}{n})$, where
\[
\codej{w}{i} := |\{j \in [n] \mid i < j \text{ and } w(i)>w(j)\}| \quad \text{ for all }i\in [n],
\]
i.e., $\codej{w}{i}$ is the number of inversions of $w$ of the form $(i,\cdot)$. In particular, we have
\begin{align}\label{lehmer_length}
\ell(w) = \sum_{i=1}^n\codej{w}{i}.
\end{align}
Recall that $w$ is uniquely determined by its Lehmer code. We will show that $\code{w_\sigma}$ can be computed easily using $\sigma$.

\begin{eg}\label{eg.Lehmer}
Let $\sigma$ be the tableau from \cref{eg_permutations_intro,eg_permutations}. Recall that $w_\sigma = 75182364$, which has Lehmer code $\code{w_\sigma} = (6,4,0,4,0,0,1,0)$. We can read off $\code{w_\sigma}$ from $\sigma$ as follows. For every entry $j$ of $\sigma$, we count the entries of $\sigma$ which are less than $j$ and appear in any row above $j$. For example, if $j=5$ there are $4$ such entries, shown in bold below:
\[
\sigma = \;\begin{ytableau}\textbf{1}&\textbf{3}&\textbf{4}&6 \\ \textbf{2}&7\\5&8\end{ytableau}\;.
\]
For $j=1, \ldots, 8$ the list of these counts is $(0,1,0,0,4,0,4,6)$, which is exactly $\code{w_\sigma}$ in reverse.
\end{eg}

The next result shows that the calculation of $\code{w_\sigma}$ from \cref{eg.Lehmer} works in general. Recall that $\bar{\sigma}$ denotes the tableau obtained from $\sigma$ by deleting its largest entry.
\begin{lem}\label{lem.w.induction}
Let $\sigma$ be a standard tableau of size $n$. Then the Lehmer code of $w_\sigma$ is given by
\begin{align}\label{eqn.wsig.Lehmer}
\codej{w_\sigma}{n+1-j} = |\{i < j \mid \rowjword{\sigma}{i}< \rowjword{\sigma}{j}\}| \quad \text{ for all } j \in [n].
\end{align}
In particular, let $\lambda$ denote the shape of $\sigma$, and suppose that $n\ge 1$ appears in row $k$ of $\sigma$. Then
\begin{align}\label{eqn.w.induction}
\ell(w_\sigma) - \ell(w_{\bar{\sigma}}) = \codej{w_\sigma}{1} = \lambda_1+\cdots + \lambda_{k-1}.
\end{align}
\end{lem}

\begin{proof}
We proceed by induction on $n$, where the base case $n=0$ holds vacuously. Now suppose that $n\ge 1$ and the result holds for standard tableaux of size $n-1$. Let $y_0$ denote the longest permutation in $S_{n-1}$. By definition, $w_0w_{\sigma}^{-1}w_0$ is the reading word of $\sigma$, and $y_0w_{\bar{\sigma}}^{-1}y_0$ is the reading word of $\bar{\sigma}$. Thus, we obtain the one-line notation of $y_0w_{\bar{\sigma}}^{-1}y_0$ from that of $w_0w_{\sigma}^{-1}w_0$ by deleting $n$. Equivalently, we obtain the one-line notation of $w_{\bar{\sigma}}$ from that of $w_\sigma$ by deleting the first entry (which equals $w_\sigma(1)$) and subtracting $1$ from all later values greater than $w_\sigma(1)$. Also, since $n$ appears in the reading word of $\sigma$ in position $n - (\lambda_1 + \cdots + \lambda_{k-1})$, we have
\[
(w_0w_{\sigma}^{-1}w_0)^{-1}(n) = n - (\lambda_1 + \cdots + \lambda_{k-1}), \quad \text{i.e.,} \quad w_\sigma(1) = \lambda_1 + \cdots + \lambda_{k-1} + 1.
\]
Therefore we get
\begin{align}\label{lehmer_induction}
\codej{w_\sigma}{1} = \lambda_1 + \cdots + \lambda_{k-1} \quad \text{ and } \quad \codej{w_\sigma}{i} = \codej{w_{\bar{\sigma}}}{i-1} \text{ for all } 2 \le i \le n.
\end{align}

Note that \eqref{eqn.wsig.Lehmer} when $j=n$ asserts that $\codej{w_\sigma}{1} = \lambda_1 + \cdots + \lambda_{k-1}$, since every entry of $\sigma$ in rows $1, \ldots, k-1$ is less than $n$. Therefore \eqref{lehmer_induction} implies \eqref{eqn.wsig.Lehmer} by induction, and it also implies \eqref{eqn.w.induction} using \eqref{lehmer_length}.
\end{proof}

\begin{rmk}\label{w_PR}
The formula \eqref{eqn.wsig.Lehmer} for the Lehmer code of $w_\sigma$ agrees with \cite[Theorem 8.9]{pagnon_ressayre06}. This implies that our definition of $w_\sigma$ is equivalent to the definition used in \cite{pagnon03,pagnon_ressayre06}.
\end{rmk}

We now derive a recursive formula for the length of $v_{\sigma}$ analogous to \eqref{eqn.w.induction}:
\begin{lem}\label{lem.v.induction}
Let $\sigma\in \SYT(\lambda)$, where $\lambda\vdash n$ with $n\ge 1$, and suppose that $n$ appears in row $k$ of $\sigma$. For $1 \le i \le k-1$, let $c_i$ denote the column number of the rightmost entry of the slide path in row $i$ of the first evacuation slide of $\sigma^\vee$. Then
\[
\ell(v_\sigma) - \ell(v_{\bar{\sigma}}) = \lambda_1+\cdots + \lambda_{k-1}- 2(c_1+\cdots +c_{k-1}) + k-1.
\]
\end{lem}

Before proving \cref{lem.v.induction}, we first motivate the proof through an example.
\begin{eg}\label{eg.v.induction}
Let $\tau\in\SYT(\lambda)$ denote the standard tableau from \cref{eg.slide} (where it was called $\sigma$), with $\lambda = (4,4,3,3,2)$. Set $\sigma = \tau^\vee$. By \cref{lem.induction}, $(\bar{\sigma})^{\vee}$ is obtained from $\sigma^\vee=\tau$ by performing an evacuation slide, as illustrated below.
\[
\sigma^\vee = \tau = \;\begin{ytableau} *(yellow)1&*(yellow)2&5&13\\3&*(yellow)4&9&15\\6&*(yellow)8&*(yellow)11\\7&12&*(yellow)14\\10&16 \end{ytableau}\;
\quad \quad \quad
(\bar{\sigma})^{\vee} = \;\begin{ytableau} *(yellow)1&*(yellow)3&4&12\\2&*(yellow)7&8&14\\5&*(yellow)10& *(yellow)13\\6&11\\9&15\end{ytableau}
\]
In the notation of \cref{lem.v.induction}, we have $k=4$ and $c_1 = 2$, $c_2 = 2$, $c_3 = 3$, $c_4=3$. We also have
\begin{align*}
v_\sigma^{-1} &= [\mathbf{1}, \mathbf{2}, 5, 13, 3, \mathbf{4}, 9, 15, 6, \mathbf{8}, \mathbf{11}, 7, 12, \mathbf{14}, 10, 16], \\
v_{\bar\sigma}^{-1} &= [\mathbf{1}, \mathbf{3},4,12,2,\mathbf{7},8,14,5,\mathbf{10}, \mathbf{13},6,11,9,15],
\end{align*}
where we have bolded the entries in the slide path. They show how $v_{\bar{\sigma}}^{-1}$ is obtained from $v_{\sigma}^{-1}$: we shift each of the bolded values to replace the bolded value to its left (bumping $1$ out of the permutation), and then subtract $1$ from every value.

The inversions of $v_{\sigma}$ and those of $v_{\bar{\sigma}}$ are closely related. First, observe that $(i,j)$ is an inversion of a permutation $w$ if and only if $i < j$ and $j$ appears before $i$ in the one-line notation of $w^{-1}$. Hence any inversion of $v_{\sigma}$ between a pair of unbolded values (i.e.\ values that are not on the slide path) gives an inversion of $v_{\bar{\sigma}}$ (after subtracting $1$ from both values), and vice-versa. Since the bolded values are increasing, there are no inversions between any pair them.

It remains to consider inversions between a bolded value and an unbolded value. For example, $(4,5)$ and $(4,13)$ are inversions of $v_{\sigma}$, but shifting $4$ to the left into the position of $2$ in the one-line notation for $v_{\sigma}^{-1}$ means that the corresponding pairs $(3,4)$ and $(3,12)$ are not inversions of $v_{\bar{\sigma}}$. Notice that $5$ and $13$ are the entries in $\sigma^\vee$ in the row immediately above $4$ and to its right, so we lose  $2 = \lambda_1 - c_1$ inversions when we go from $v_\sigma$ to $v_{\bar{\sigma}}$.
Similarly, $(3,4)$ is not an inversion of $v_{\sigma}$, but after $4$ is shifted, the corresponding pair $(2,3)$ becomes an inversion of $v_{\bar{\sigma}}$. Notice that $3$ is the entry in $\sigma^\vee$ to the left of $4$ in the same row, so we add $1 = c_1 - 1$ inversion when we go from $v_\sigma$ to $v_{\bar\sigma}$.
 In summary, shifting $4$ erases $2 = \lambda_1 - c_1$ new inversions and adds $1 = c_1 - 1$ inversion. We will show that this holds in general, and thus
\begin{align*}
\ell(v_{\bar{\sigma}}) &= \ell(v_{\sigma}) - (\lambda_1-c_1) + (c_1 -1) - (\lambda_2-c_2) +(c_2-1) - (\lambda_3-c_3) + (c_3-1)  \\[2pt]
&= \ell(v_\sigma) - (4-2) + (2-1) - (4-2) + (2-1) - (3-3) + (3-1)   \\[2pt]
&= \ell(v_\sigma).
\end{align*}
Indeed, the reader can confirm that $\ell(v_{\sigma}) = 26 = \ell(v_{\bar{\sigma}})$ in this case.
\end{eg}

\begin{proof}[Proof of \cref{lem.v.induction}]
Let $d_1 = 1 < d_2 < \cdots < d_m$ denote the entries of $\sigma^\vee$ appearing in the slide path of the first evacuation slide. By \cref{lem.induction} and by definition of the permutations $v_{\sigma}$ and $v_{\bar{\sigma}}$, we obtain the one-line notation of $v_{\bar{\sigma}}^{-1}$ from the one-line notation of $v_{\sigma}^{-1}$ as follows:
\begin{enumerate}[label=(\arabic*), leftmargin=36pt, itemsep=2pt]
\item\label{slide_shift} replace $d_{j-1}$ by $d_j$ for $j = m, m-1, \dots, 2$, bumping $1$ out of the permutation; and 
\item\label{slide_subtract} subtracting $1$ from every entry.
\end{enumerate}
We wish to compare inversions of $v_\sigma$ and $v_{\bar{\sigma}}$, as in \cref{eg.v.induction}. It suffices to consider step \ref{slide_shift} only and ignore step \ref{slide_subtract}. Step \ref{slide_shift} only changes inversions involving entries of the slide path that move more than one position to the left. Let $d_j$ be such an entry, which must be the first entry of the slide path appearing in some row $i$ of $\sigma^{\vee}$ with $2\leq i \leq k$. Then $d_{j-1}$ is the entry immediately above $d_j$ in $\sigma^\vee$, which by definition appears in column $c_{i-1}$.

By definition of evacuation, $d_j$ must be less than every entry $b$ in row $i-1$ to the right of $d_{j-1}$.  The $\lambda_{i-1} - c_{i-1}$ pairs of the form $(d_j,b)$ are inversions of $v_{\sigma}$.  Moving $d_j$ into the position of $d_{j-1}$ in step \ref{slide_shift} eliminates these inversions.

On the other hand, since $\sigma^\vee$ is a standard tableau, $d_j$ is greater than every entry $a$ to its left in row $i$.  The $c_{i-1}-1$ pairs $(a,d_j)$ are not inversions of $v_{\sigma}$, but become inversions after moving $d_j$ into the position of $d_{j-1}$ in step \ref{slide_shift}. Since the numbers between $d_{j-1}$ and $d_j$ in the one-line notation of $v_{\sigma}^{-1}$ are precisely those of the form $b$ or $a$ as above, all other inversions involving $d_j$ remain unchanged in step \ref{slide_shift}.

The previous two paragraphs imply that
\begin{align*}
\ell(v_\sigma) - \ell(v_{\bar{\sigma}}) &= \sum_{i=2}^k\big((\lambda_{i-1}-c_{i-1}) - (c_{i-1}-1)\big) \\[2pt]
&= \lambda_1+\cdots + \lambda_{k-1} - 2(c_1+\cdots + c_{k-1}) + k-1,
\end{align*}
as desired.
\end{proof}

\begin{eg} It is not true, in general, that $\ell(v_\sigma)\geq \ell (v_{\bar{\sigma}})$. Indeed, consider the tableaux below.
\[
\sigma = \sigma^{\vee} = \;\begin{ytableau}
1 & 2 \\
3 & 4
\end{ytableau}\;
\hspace*{48pt}
(\bar{\sigma})^{\vee} = \;\begin{ytableau}
1 & 3 \\
2\end{ytableau}
\]
In this case $v_\sigma = 1234$ has length $0$, while $v_{\bar{\sigma}} = 132$ has length $1$.
\end{eg}

\begin{lem}\label{length_induction}
Let $\sigma$ be a standard tableau of size $n$, where $n\ge 1$, and suppose that $n$ appears in row $k$ of $\sigma$. Then
\[
\ell(w_\sigma) - \ell(v_\sigma) \ge \ell(w_{\bar{\sigma}}) - \ell(v_{\bar{\sigma}}) + k-1,
\]
where equality holds if and only if the first evacuation slide of $\sigma^\vee$ is an $L$-slide.
\end{lem}

\begin{proof}
By \cref{lem.w.induction,lem.v.induction} we have
\begin{align}\label{length_induction_precise}
\ell(w_\sigma) - \ell(v_\sigma) = \ell(w_{\bar{\sigma}}) - \ell(v_{\bar{\sigma}}) + 2(c_1 + \cdots + c_{k-1}) - (k-1),
\end{align}
where $c_i$ is the column number of the rightmost entry of the slide path in row $i$ of the first evacuation slide of $\sigma^\vee$. In particular $c_1 + \cdots + c_{k-1} \ge k-1$, where equality holds if and only if all the $c_i$'s are equal to $1$, i.e., the first evacuation slide of $\sigma^\vee$ is an $L$-slide. Plugging this into \eqref{length_induction_precise} completes the proof.
\end{proof}

We are now in the position to prove \cref{thm-lengths}:
\begin{proof}[Proof of \cref{thm-lengths}]
We proceed by induction on $n = |\sigma|$, where we can verify the base case $n=0$ directly. Let $n\ge 1$ and assume the result holds for all standard tableaux of size $n-1$. Suppose that $n$ appears in $\sigma$ in row $k$, and let $\mu$ denote the partition obtained from $\lambda$ by decreasing the $k$th part by $1$. Then $\mu$ is the shape of $\bar{\sigma}$, and $n(\lambda) = n(\mu) + k-1$. By \cref{length_induction}, we get
\[
\ell(w_\sigma) - \ell(v_\sigma) \ge \ell(w_{\bar{\sigma}}) - \ell(v_{\bar{\sigma}}) + n(\lambda) - n(\mu),
\]
where equality holds if and only if the first evacuation slide of $\sigma^\vee$ is an $L$-slide. By induction, we have
\[
\ell(w_{\bar{\sigma}}) - \ell(v_{\bar{\sigma}}) \ge n(\mu),
\]
where equality holds if and only if $\bar{\sigma}$ is Richardson. Together, these inequalities imply
\[
\ell(w_\sigma) - \ell(v_\sigma) \ge n(\lambda),
\]
where equality holds if and only if the first evacuation slide of $\sigma^\vee$ is an $L$-slide and $\bar{\sigma}$ is Richardson. It remains to show that this equality condition is equivalent to $\sigma$ being Richardson.

First suppose that $\sigma$ is Richardson. Then $\bar{\sigma}$ is also Richardson. By \cref{evacuation_closed} we get that $\sigma^\vee$ is Richardson, so the first evacuation slide of $\sigma^\vee$ is an $L$-slide by \cref{lem.Lshape}. This verifies the equality condition.

Conversely, suppose that equality holds. We must show that $\sigma$ is Richardson. By \cref{evacuation_closed} and \cref{thm-Lslides}, it suffices to show that each evacuation slide of $\sigma^\vee$ is an $L$-slide. We already know that the first evacuation slide of $\sigma^\vee$ is an $L$-slide. By \cref{lem.induction}, the subsequent evacuation slides for $\sigma^\vee$ are precisely the evacuation slides for $(\bar{\sigma})^{\vee}$. Again applying \cref{evacuation_closed} and \cref{thm-Lslides}, these slides are all $L$-slides since $\bar{\sigma}$ is Richardson.
\end{proof}


\section{Springer fibers and Richardson varieties}\label{sec_geometric}

\noindent In this section we prove our main geometric result, \cref{main} below, which says that Richardson tableaux index the irreducible components of Springer fibers which are equal to Richardson varieties.


\subsection{Two involutions on the flag variety}
We begin by defining two involutive isomorphisms of $\Fl_n(\C)$. Let $(x,y) := \transpose{y}x = x_1y_1 + \cdots + x_ny_n$ denote the standard bilinear pairing on $\C^n$, and for $V\subseteq\C^n$, let $V^\perp := \{y\in\C^n\mid (x,y) = 0 \text{ for all } x \in V\}$ denote the orthogonal complement of $V$. The first involution $(\cdot)^\perp : \Fl_n(\C) \to \Fl_n(\C)$ is defined by
\[
F_\bullet = (0 \subset F_1 \subset \cdots \subset F_{n-1} \subset \C^n) \mapsto F_\bullet^\perp := (0 \subset (F_{n-1})^\perp \subset \cdots \subset (F_1)^\perp \subset \C^n).
\]
We point out that in the literature on Springer fibers one often defines $F_\bullet^\perp$ more abstractly, for example by using dual vector spaces or quotients of $\C^n$. The reason we fix a pairing and take orthogonal complements is because we need explicit formulas for how Schubert and opposite Schubert varieties transform under $(\cdot)^\perp$, as in \cref{dual_transpose}\ref{dual_transpose_schubert} below.

We will use the following formula for $(\cdot)^\perp$ in terms of cosets (i.e.\ writing $\Fl_n(\C) = \GL_n(\C)/B$).
\begin{lem}\label{lem_perp_coset}
We have
\begin{align}\label{perp_coset}
(gB)^\perp = (\transpose{g})^{-1}\dot{w}_0B \quad \text{ for all }\, gB \in \Fl_n(\C) = \GL_n(\C)/B.
\end{align}
\end{lem}

\begin{proof}
Given $g\in\GL_n(\C)$, let $F_\bullet = gB$ denote the corresponding flag, so that each $F_j$ is spanned by the first $j$ columns of $g$. Set $h = (\transpose{g})^{-1}$. Then the equation $\transpose{g}h = I$ implies that column $i$ of $g$ is orthogonal to column $j$ of $h$ for all $i \neq j$. In particular, each $F_j^\perp$ is spanned by the last $n-j$ columns of $h$, so $F_\bullet^\perp = h\dot{w}_0B$ as desired.
\end{proof}

Given $\lambda\vdash n$, the second involution of $\Fl_n(\C)$ we consider is left translation by $w_{0,\lambda}$ (the longest element of the Young subgroup $S_\lambda$):
\[
\varphi_\lambda: \Fl_n(\C) \to \Fl_n(\C), \quad F_\bullet \mapsto \dot{w}_{0,\lambda}F_\bullet. 
\]
Recall that $\SF{\lambda} = \SF{N_\lambda}$ is the Springer fiber defined by the nilpotent matrix $N_\lambda$ in Jordan canonical form with weakly decreasing block sizes. A key property of $N_\lambda$ for our purposes is
\begin{align}\label{conjugate_transpose}
\dot{w}_{0,\lambda}N_\lambda\dot{w}_{0,\lambda} = \transpose{N_\lambda}.
\end{align}

We collect some basic properties of $(\cdot)^\perp$ and $\varphi_\lambda$:
\begin{lem}\label{dual_transpose}
Let $\lambda\vdash n$.
\begin{enumerate}[label=(\roman*), leftmargin=*, itemsep=2pt]
\item\label{dual_transpose_flag} For all nilpotent $N\in\gl_n(\C)$, we have $\SF{N}^\perp = \SF{\transpose{N}}$.
\item\label{dual_transpose_commute} The involutions $(\cdot)^\perp$ and $\varphi_\lambda$ commute.
\item\label{dual_transpose_involution} We have $(\varphi_\lambda(\SF{\lambda}))^\perp = \varphi_\lambda(\SF{\lambda}^\perp) = \SF{\lambda}$.
\item\label{dual_transpose_permutation} For all $w\in S_n$, have $(\dot w B)^\perp = \dot w \dot w_0B$ and $\varphi_\lambda(\dot w B) = \dot{w}_{0,\lambda} \dot w B$ in $\Fl_n(\C)$.
\item\label{dual_transpose_schubert} For all $w\in S_n$, we have $\Scell{w}^\perp = \Scellopp{ww_0}$ and $\Svar{w}^\perp = \Svaropp{ww_0}$.
\item\label{dual_transpose_translate} For all $w\in\mincoset{S_n}{\lambda}$, we have $\varphi_\lambda(\Svar{w}) \subseteq \Svar{w_{0,\lambda}w}$ and $\varphi_\lambda(\Svar{w}^\perp) \subseteq \Svaropp{w_{0,\lambda}ww_0}$.
\end{enumerate}
\end{lem}

\begin{proof}
Proof of \ref{dual_transpose_flag}: By \eqref{def_springer}, we have
\begin{eqnarray*}
gB \in \SF{N} &\Leftrightarrow& g^{-1}Ng \in \fu \qquad \text{(by \eqref{def_springer})}\\ 
&\Leftrightarrow& \dot w_0 g^{-1}Ng \dot w_0 \in \dot w_0 \fu \dot w_0 \qquad \text{(conjugating by $\dot w_0$)} \\
&\Leftrightarrow& \dot w_0 \transpose{g} \transpose{N} (\transpose{g})^{-1} \dot w_0 \in \fu \qquad \text{(taking the transpose)}\\
&\Leftrightarrow & (\transpose{g})^{-1}\dot w_0B \in \SF{\transpose{N}} \qquad \text{(by \eqref{def_springer})}\\
&\Leftrightarrow & (gB)^\perp \in \SF{\transpose{N}}  \qquad \text{(by \eqref{perp_coset})},
\end{eqnarray*}
where the third equivalence uses the fact that $\transpose{(\dot{w}_0\fu\dot{w}_0)} = \fu$ and $\transpose{(g^{-1})} = (\transpose{g})^{-1}$. The result follows.

Proof of \ref{dual_transpose_commute}: This follows from \eqref{perp_coset}, since the permutation $w_{0,\lambda}$ is its own inverse and thus $\dot{w}_{0,\lambda} = \dot{w}_{0,\lambda}^{-1} = \transpose{\dot{w}_{0,\lambda}}$.

Proof of \ref{dual_transpose_involution}: The first equality follows from part \ref{dual_transpose_commute}. For the second equality, by part \ref{dual_transpose_flag} and \eqref{conjugate_transpose} we have
\[
\varphi_\lambda(\SF{\lambda}^\perp) = \varphi_\lambda(\SF{\transpose{N_\lambda}}) = \dot{w}_{0,\lambda}\SF{\transpose{N_\lambda}} = \SF{\dot{w}_{0,\lambda}\transpose{N_\lambda}\dot{w}_{0,\lambda}^{-1}} = \SF{N_\lambda} = \SF{\lambda}.
\]

Proof of \ref{dual_transpose_permutation}: The first equality follows from \eqref{perp_coset} since $(\transpose{\dot{w}})^{-1} = \dot{w}$. The second equality follows from the definition of $\varphi_\lambda$.

Proof of \ref{dual_transpose_schubert}: By \eqref{perp_coset}, we have
\[
\Scell{w}^\perp = (B\dot{w}B)^\perp = \left(\transpose{(B\dot{w})}\right)^{-1}\dot{w}_0B = (\transpose{\dot w} B_-)^{-1} \dot{w}_0 B = B_- \dot{w}\dot{w}_0B = \Scellopp{ww_0},
\]
which proves the first equality. The second equality then follows by taking the closure.

Proof of \ref{dual_transpose_translate}: To prove the first containment, it suffices to show that $\varphi_\lambda(\Scell{w}) \subseteq \Svar{w_{0,\lambda}w}$ and then take the closure. Since $w\in \mincoset{S_n}{\lambda}$ the product $w_{0,\lambda}w$ is length-additive, i.e., $\ell(w_{0,\lambda}w) = \ell(w_{0,\lambda}) + \ell(w)$. By definition, $\varphi_\lambda(\Scell{w}) = \dot{w}_{0,\lambda} \Scell{w}$. The containment $\varphi_\lambda(\Scell{w}) \subseteq \Svar{w_{0,\lambda}w}$ then follows from the general fact that $\dot{u}\Scell{v} \subseteq \Svar{uv}$ for all $u,v\in S_n$ such that $uv$ is length-additive \cite[Corollary (of proof) in Section 28.3]{humphreys75}. The second containment then follows by applying $(\cdot)^\perp$ to the first containment, using parts \ref{dual_transpose_commute} and \ref{dual_transpose_schubert}.
\end{proof}

\begin{eg}\label{eg_perp}
Let $F_\bullet = gB \in \Fl_4(\C)$ be the flag from \cref{eg_springer}, so that
\[
g = \begin{bmatrix}
2 & 1 & 0 & 0 \\
0 & 0 & 0 & 1 \\
1 & 0 & 1 & 0 \\
1 & 0 & 0 & 0
\end{bmatrix} \in \GL_4(\C)
\quad\text{ and }\quad
(\transpose{g})^{-1}\dot w_0 = \begin{bmatrix}
0 & 0 & 1 & 0 \\
1 & 0 & 0 & 0 \\
0 & 1 & 0 & 0 \\
0 & -1 & -2 & 1
\end{bmatrix}.
\]
The reader can check that $F_\bullet^\perp = (\transpose{g})^{-1}\dot w_0B$. Now recall from \cref{eg_springer} that $F_\bullet \in \SF{\lambda}$, where $\lambda = (2,1,1)$. Note that $F_\bullet^{\perp}\notin\SF{\lambda}$, but we do have $F_\bullet^\perp \in \SF{\transpose{N_\lambda}}$, in agreement with \cref{dual_transpose}\ref{dual_transpose_flag}. Finally, the flag $\varphi_\lambda(F_\bullet^{\perp})$ is represented by the matrix
\[
\dot{w}_{0,\lambda}(\transpose{g})^{-1}\dot w_0 =
\begin{bmatrix}
0 & 1 & 0 & 0 \\
1 & 0 & 0 & 0 \\
0 & 0 & 1 & 0 \\
0 & 0 & 0 & 1
\end{bmatrix}
\begin{bmatrix}
0 & 0 & 1 & 0 \\
1 & 0 & 0 & 0 \\
0 & 1 & 0 & 0 \\
0 & -1 & -2 & 1
\end{bmatrix} =
\begin{bmatrix}
1 & 0 & 0 & 0 \\
0 & 0 & 1 & 0 \\
0 & 1 & 0 & 0 \\
0 & -1 & -2 & 1
\end{bmatrix},
\]
and so $\varphi_\lambda(F_\bullet^{\perp}) \in \SF{\lambda}$, in agreement with \cref{dual_transpose}\ref{dual_transpose_involution}.
\end{eg}

By \cref{dual_transpose}\ref{dual_transpose_involution}, the involution $F_\bullet\mapsto\varphi_\lambda(F_\bullet^\perp)$ restricts to an automorphism of $\SF{\lambda}$, and hence permutes the irreducible components of $\SF{\lambda}$. Work of van Leeuwen \cite{van_leeuwen00} shows that the induced action on standard tableaux recovers evacuation:
\begin{thm}[{van Leeuwen \cite[Corollary 3.4]{van_leeuwen00}}]\label{dual_evacuation}
We have $\varphi_\lambda(\SF{\sigma}^\perp) = \SF{\sigma^\vee}$ for all $\sigma\in\SYT(\lambda)$.
\end{thm}

\begin{eg}
Consider the flags $F_\bullet$ and $\varphi_{\lambda}(F_{\bullet}^{\perp})$ from \cref{eg_perp}, where $\lambda = (2,1,1)$. Recall from \cref{eg_springer} that $F_\bullet \in \SFcell{\sigma}$, where $\sigma = \ytableausmall{1 & 3\\ 2\\ 4}$. Since $\sigma^{\vee}=\sigma$, we conclude that $\varphi_{(2,1,1)}(F_{\bullet}^{\perp}) \in \SF{\sigma}$ by \cref{dual_evacuation}. Note that this is not immediately obvious, since $\varphi_{(2,1,1)}(F_{\bullet}^{\perp})$ is not in $\SFcell{\sigma}$, but rather $\varphi_{(2,1,1)}(F_{\bullet}^{\perp})\in\SFcell{\tau^\vee}$, where $\tau = \ytableausmall{1 & 2\\ 3\\ 4}$ and $\tau^\vee = \ytableausmall{1 & 4 \\ 2 \\ 3}$. This is consistent with the fact that $F_\bullet \in \SF{\tau}$, as observed in \cref{eg_springer}.
\end{eg}


\subsection{Richardson envelope of an irreducible component}
We now determine the unique minimal Richardson variety containing a given irreducible component $\SF{\sigma}$, stated in \cref{main_geometric} below.  We begin with the following observation.

\begin{lem}\label{w_contained}
Let $\sigma\in\SYT(\lambda)$. Then $\dot{w}_\sigma B\in\SF{\lambda}$.
\end{lem}

\begin{proof}
This follows from \cref{permutations_in_SF}, since $w_\sigma\in\mincoset{S_n}{\lambda}$ by \cref{reading_word_bruhat}.
\end{proof}

We also need the following result which combines work of Pagnon \cite{pagnon03} and of Pagnon and Ressayre \cite{pagnon_ressayre06}, using \cref{w_PR}:

\begin{thm}[{Pagnon \cite[Theorem 6.1.2]{pagnon03}, Pagnon and Ressayre \cite[Theorem 5.1]{pagnon_ressayre06}}]\label{PR} Let $\sigma\in \SYT(\lambda)$. Then $\Scell{w_\sigma}$ is the unique Schubert cell such that $\SF{\lambda} \cap \Scell{w_\sigma}$ is a dense subset of $\SF{\sigma}$. In particular (using \cref{w_contained}), $\dot w_\sigma B\in \SF{\sigma}$ and $\SF{\sigma} \subseteq \Svar{w_\sigma}$.
\end{thm}

\begin{cor}\label{dual_PR}
Let $\sigma\in\SYT(\lambda)$. Then $\dot v_\sigma B\in\SF{\sigma}$ and $\SF{\sigma}\subseteq\Svaropp{v_\sigma}$.
\end{cor}

\begin{proof}
By \cref{PR}, we have $\dot w_{\sigma^\vee}B\in \SF{\sigma^\vee}$ and $\SF{\sigma^\vee} \subseteq \Svar{w_{\sigma^\vee}}$. Applying the map $F_\bullet \mapsto \varphi_\lambda(F_\bullet^\perp)$ gives
\begin{align}\label{dual_PR_initial}
\varphi_\lambda((\dot w_{\sigma^\vee}B)^\perp) \in \varphi_\lambda(\SF{\sigma^\vee}^\perp) \quad \text{ and } \quad \varphi_\lambda(\SF{\sigma^\vee}^\perp) \subseteq \varphi_\lambda(\Svar{w_{\sigma^\vee}}^\perp).
\end{align}
Now note that:
\begin{itemize}[itemsep=2pt]
\item $\varphi_\lambda((\dot w_{\sigma^\vee}B)^\perp) = \dot{w}_{0,\lambda} \dot w_{\sigma^\vee} \dot w_0 B$ by \cref{dual_transpose}\ref{dual_transpose_permutation};
\item $w_{0,\lambda}w_{\sigma^\vee}w_0 = v_\sigma$ by \cref{v_w_relation};
\item $\varphi_\lambda(\SF{\sigma^\vee}^\perp) = \SF{\sigma}$ by \cref{dual_evacuation}; and
\item $\varphi_\lambda(\Svar{w_{\sigma^\vee}}^\perp) \subseteq \Svaropp{w_{0,\lambda}w_{\sigma^\vee}w_0}$ by \cref{dual_transpose}\ref{dual_transpose_translate}, since $w_{\sigma^{\vee}}\in\mincoset{S_n}{\lambda}$ by \cref{reading_word_bruhat}.
\end{itemize}
Thus \eqref{dual_PR_initial} implies that
\[
\dot v_\sigma B \in \SF{\sigma} \quad \text{ and } \quad \SF{\sigma} \subseteq \Svaropp{v_\sigma},
\]
as desired.
\end{proof}

\begin{thm}\label{main_geometric}
Let $\sigma\in\SYT(\lambda)$. Then the minimal Richardson variety containing $\SF{\sigma}$ is $\Rvar{v_\sigma}{w_\sigma}=\Svaropp{v_\sigma} \cap \Svar{w_\sigma} $. That is, we have:
\begin{itemize}[itemsep=2pt]
\item $\SF{\sigma}\subseteq\Rvar{v_\sigma}{w_\sigma}$; and
\item if $\SF{\sigma}\subseteq\Rvar{v}{w}$ for some $v,w\in S_n$, then $\Rvar{v_\sigma}{w_\sigma} \subseteq \Rvar{v}{w}$.
\end{itemize}
In particular, we have $v_\sigma \le w_\sigma$.
\end{thm}

\begin{proof}
By \cref{PR,dual_PR}, we have $\SF{\sigma} \subseteq \Svaropp{v_\sigma} \cap \Svar{w_\sigma} = \Rvar{v_\sigma}{w_\sigma}$. In particular $\Rvar{v_\sigma}{w_\sigma}$ is nonempty, and therefore $v_\sigma \le w_\sigma$ by \eqref{richardson_nonempty}. Conversely, suppose that $\SF{\sigma}\subseteq\Rvar{v}{w}$ for some $v,w\in S_n$. By \cref{PR,dual_PR} we have $\dot v_\sigma B, \dot w_\sigma B\in \SF{\sigma}$, and thus $\dot v_\sigma B, \dot w_\sigma B \in \Rvar{v}{w}$.  Since $\{\dot v_\sigma B\} = \Rvar{v_\sigma}{v_\sigma}$ and $\{\dot w_\sigma B\} = \Rvar{w_\sigma}{w_\sigma}$, applying \eqref{richardson_closure} we get $v \le v_\sigma \le w$ and $v \le w_\sigma \le w$. Applying \eqref{richardson_closure} again gives $\Rvar{v_\sigma}{w_\sigma} \subseteq \Rvar{v}{w}$, as desired.
\end{proof}

\cref{main_geometric} describes the unique minimal Richardson variety containing $\SF{\sigma}$. It would be interesting to solve the reverse problem, of finding the maximal Richardson varieties contained in $\SF{\sigma}$. As we will see in \cref{sec_positivity}, solving this problem is equivalent to finding the maximal cells of the totally nonnegative part of $\SF{\sigma}$.
\begin{prob}\label{prob_richardson_containment}
Let $\sigma$ be a standard tableau. Find all Richardson varieties contained in $\SF{\sigma}$ which are maximal by containment.
\end{prob}

\begin{eg}\label{eg_prob_richardson_containment}
Let $\sigma = \ytableausmall{1 & 2 \\ 3 & 4}$, which is not a Richardson tableau. By \cref{main_geometric}, the minimal Richardson variety containing $\SF{\sigma}$ is $\Rvar{v_\sigma}{w_\sigma} = \Rvar{1234}{3412}$. One can show (cf.\ \cref{figure_tnn} below) that the maximal Richardson varieties contained in $\SF{\sigma}$ are $\Rvar{1234}{1324}$ and $\Rvar{3142}{3412}$.
\end{eg}


\subsection{Main geometric result}
We finally have all the ingredients needed to prove the main geometric result of this paper:
\begin{thm}\label{main}
Let $\sigma\in\SYT(\lambda)$. The irreducible component $\SF{\sigma}$ of $\SF{\lambda}$ is a Richardson variety if and only if $\sigma$ is a Richardson tableau, in which case $\SF{\sigma} = \Rvar{v_\sigma}{w_\sigma}$.
\end{thm}

\begin{proof}
($\Rightarrow$) Suppose that $\SF{\sigma}$ is a Richardson variety. Then by \cref{main_geometric} we have $\SF{\sigma} = \Rvar{v_\sigma}{w_\sigma}$, and taking dimensions gives $n(\lambda) = \ell(w_\sigma) - \ell(v_\sigma)$. Hence $\sigma$ is a Richardson tableau by \cref{thm-lengths}.

($\Leftarrow$) Suppose that $\sigma$ is a Richardson tableau. Then $\dim(\Rvar{v_\sigma}{w_\sigma}) = n(\lambda)$ by \cref{thm-lengths}. Hence $\SF{\sigma}$ and $\Rvar{v_\sigma}{w_\sigma}$ are irreducible closed subvarieties of $\Fl_n(\C)$ of the same dimension, and by \cref{main_geometric} we have $\SF{\sigma} \subseteq \Rvar{v_\sigma}{w_\sigma}$. Thus $\SF{\sigma} = \Rvar{v_\sigma}{w_\sigma}$.
\end{proof}


\section{Totally nonnegative Springer fibers}\label{sec_positivity}

\noindent In this section we use Lusztig's work \cite{lusztig21} on totally nonnegative Springer fibers to obtain yet another characterization of Richardson tableaux in \cref{richardson_Z} below, this time in terms of the Bruhat order on $S_n$.  Conversely, we use our results on Richardson tableaux to deduce new enumerative statements about totally nonnegative Springer fibers in \cref{tp_main} below.

We begin by recalling some background on total positivity for flag varieties and Springer fibers in type $A$ from \cite{lusztig94,lusztig21}. We call a matrix $g\in\GL_n(\C)$ \boldit{totally positive} if the determinant of every square submatrix of $g$ is real and positive, and let $\GL_n^{>0}$ denote the set of all totally positive $g$. The \boldit{totally positive flag variety} $\Fl_n^{>0}$ is defined to be the image of $\GL_n^{>0}$ in $\Fl_n(\C) = \GL_n(\C)/B$, and the \boldit{totally nonnegative flag variety} $\Fl_n^{\ge 0}$ is defined to be the closure of $\Fl_n^{>0}$ in the Euclidean topology. Equivalently, $V_\bullet\in\Fl_n(\C)$ is totally nonnegative if and only if all of its \emph{Pl\"{u}cker coordinates} are nonnegative; see \cite[Section 1]{bloch_karp23} for further discussion.

For $v\le w$ in $S_n$, the \boldit{totally positive Richardson cell} is defined to be $\Rtp{v}{w} := \Rcell{v}{w}\cap\Fl_n^{\ge 0}$. Rietsch \cite{rietsch99} showed that $\Rtp{v}{w}$ is homeomorphic to $\mathbb{R}^{\ell(w) - \ell(v)}$, and hence we have the cell decomposition
\begin{align}\label{decomposition_flag}
\Fl_n^{\ge 0} = \bigsqcup_{v \le w}\Rtp{v}{w}.
\end{align}
In fact, the cell decomposition \eqref{decomposition_flag} is a regular CW complex homeomorphic to a closed ball \cite{galashin_karp_lam22}. The unique top-dimensional cell of $\Fl_n^{\ge 0}$ is $\Fl_n^{>0} = \Rtp{e}{w_0}$.

Now let $N$ be an $n\times n$ nilpotent matrix such that $I+N\in\GL_n^{\ge 0}$, i.e., the determinant of every square submatrix of $I+N$ is nonnegative. The \boldit{totally nonnegative Springer fiber} is defined to be $\SFtnn{N} := \SF{N}\cap\Fl_n^{\ge 0}$. Lusztig \cite[Corollary 1.16]{lusztig21} showed that, remarkably, each totally positive Richardson cell $\Rtp{v}{w}$ is either contained in $\SFtnn{N}$, or has empty intersection with $\SFtnn{N}$. In particular, $\SFtnn{N}$ has a cell decomposition into totally positive Richardson cells $\Rtp{v}{w}$ for certain pairs $(v,w)$, the set of which was described combinatorially by Lusztig using the Bruhat order. 

The combinatorics of Lusztig's description of the cells of $\SFtnn{N}$ depends on a block-diagonal decomposition of $N$ into upper-triangular and lower-triangular blocks, where two consecutive blocks are allowed to overlap in their corners if one is upper-triangular and the other is lower-triangular. We recall this description in the case that all such blocks are upper-triangular and weakly decrease in size along the diagonal. In particular, we consider the case $\SF{N} = \SF{\lambda}$, where $\lambda\vdash n$ (we indeed have $I+N_\lambda\in\GL_n^{\ge 0}$). By~\cite[Corollary 1.16]{lusztig21} there is a cell decomposition
\begin{align}\label{decomposition_springer}
\SFtnn{\lambda} = \bigsqcup_{(v,w)\in Z_\lambda}\Rtp{v}{w},
\end{align}
where
\[
Z_\lambda := \{(v,w) \in \mincoset{S_n}{\lambda} \times \mincoset{S_n}{\lambda} \mid v \le w \text{ and } s_jv\nleq w \text{ for all } j \in [n-1]\setminus\{\lambda_1, \lambda_1 + \lambda_2, \dots\}\}.
\]

Recall that $\dim_\C(\SF{\lambda}) = n(\lambda)$. Hence if $\Rtp{v}{w} \subseteq \SFtnn{\lambda}$, then taking real dimensions gives
\begin{align}\label{dimension_Z}
\ell(w) - \ell(v) \le n(\lambda) \quad \text{ for all } (v,w) \in Z_\lambda.
\end{align}
It will follow later from \cref{tp_main} that $\SFtnn{\lambda}$ in fact has real dimension $n(\lambda)$.

\begin{eg}\label{eg.TNNSpringer} Consider $\SFtnn{(2,2)}$. The set $Z_{(2,2)}$ is displayed on the left in \cref{figure_tnn} below, and each element labels the corresponding cell from the decomposition~\eqref{decomposition_springer} of $\SFtnn{(2,2)}$. The three maximal cells are labeled by the pairs $(1234,1324)$, $(1324,3142)$, and $(3142,3412)$, which index cells of dimensions $1$, $2$, and $1$, respectively. In particular, the totally nonnegative Springer fiber $\SFtnn{(2,2)}$ is not pure-dimensional. The set $Z_{(3,1)}$ for $\SFtnn{(3,1)}$ is similarly displayed on the right in \cref{figure_tnn}. From this picture it is clear that, unlike $\SFtnn{(2,2)}$, the totally nonnegative Springer fiber $\SFtnn{(3,1)}$ is pure-dimensional. 
\end{eg}

\begin{figure}[ht]
\begin{center}
\[
\begin{tikzpicture}[baseline=(current bounding box.center),scale=1.6]
\tikzstyle{vertex}=[inner sep=0,minimum size=1.2mm,circle,draw=teal!80!black,fill=teal!80!black,semithick]
\pgfmathsetmacro{\s}{0.84};
\pgfmathsetmacro{\sbig}{1.2};
\fill[color=teal!30](0,1)--(1,2)--(0,3)--(-1,2)--(0,1);
\node[vertex](1234)at(0,0)[label={[below=2pt]\scalebox{\s}{$(1234,1234)$}}]{};
\node[vertex](1324)at(0,1)[label={[right=1pt]\raisebox{-12pt}{\scalebox{\s}{$(1324,1324)$}}}]{};
\node[vertex](1342)at(-1,2)[label={[left=0pt]\raisebox{-12pt}{\scalebox{\s}{$(1342,1342)$}}}]{};
\node[vertex](3124)at(1,2)[label={[right=0pt]\raisebox{-12pt}{\scalebox{\s}{$(3124,3124)$}}}]{};
\node[vertex](3142)at(0,3)[label={[left=1pt]\raisebox{-12pt}{\scalebox{\s}{$(3142,3142)$}}}]{};
\node[vertex](3412)at(0,4)[label={[above=-3pt]\scalebox{\s}{$(3412,3412)$}}]{};
\path[very thick,color=teal!80!black](1234.center)edge node[left=-2pt]{\color{black}\scalebox{\s}{$(1234,1324)$}}(1324.center) (1324.center)edge node[left=1pt]{\color{black}\scalebox{\s}{$(1324,1342)$}}(1342.center) edge node[right=1pt]{\color{black}\scalebox{\s}{$(1324,3124)$}}(3124.center) (1342.center)edge node[left=1pt]{\color{black}\scalebox{\s}{$(1342,3142)$}}(3142.center) (3124.center)edge node[right=1pt]{\color{black}\scalebox{\s}{$(3124,3142)$}}(3142.center) (3142.center)edge node[right=-2pt]{\color{black}\scalebox{\s}{$(3142,3412)$}}(3412.center);
\node[inner sep=0]at(0,2.3){\scalebox{\sbig}{$\ytableausmall{1 & 3 \\ 2 & 4}$}};
\node[inner sep=0]at(0,1.78){\scalebox{\s}{$(1324,3142)$}};
\end{tikzpicture}
\hspace*{72pt}
\begin{tikzpicture}[baseline=(current bounding box.center),scale=2]
\tikzstyle{vertex}=[inner sep=0,minimum size=1.2mm,circle,draw=teal!80!black,fill=teal!80!black,semithick]
\pgfmathsetmacro{\s}{0.84};
\pgfmathsetmacro{\sbig}{1.2};
\node[vertex](1234)at(0,0)[label={[below=2pt]\scalebox{\s}{$(1234,1234)$}}]{};
\node[vertex](1243)at(0,1)[label={[left=0pt]\raisebox{-12pt}{\scalebox{\s}{$(1243,1243)$}}}]{};
\node[vertex](1423)at(0,2)[label={[left=0pt]\raisebox{-12pt}{\scalebox{\s}{$(1423,1423)$}}}]{};
\node[vertex](4123)at(0,3)[label={[above=-3pt]\raisebox{-12pt}{\scalebox{\s}{$(4123,4123)$}}}]{};
\path[very thick,color=teal!80!black](1234.center)edge node[left=-2pt]{\color{black}\scalebox{\s}{$(1234,1243)$}} node[right=-2pt]{\color{black}\scalebox{\sbig}{$\ytableausmall{1 & 3 & 4 \\ 2}$}}(1243.center) (1243.center)edge node[left=-2pt]{\color{black}\scalebox{\s}{$(1243,1423)$}} node[right=-2pt]{\color{black}\scalebox{\sbig}{$\ytableausmall{1 & 2 & 4 \\ 3}$}}(1423.center) (1423.center)edge node[left=-2pt]{\color{black}\scalebox{\s}{$(1423,4123)$}} node[right=-2pt]{\color{black}\scalebox{\sbig}{$\ytableausmall{1 & 2 & 3 \\ 4}$}}(4123.center);
\node[inner sep=0]at(0,-0.4){};
\end{tikzpicture}\vspace*{-12pt}
\]
\caption{Totally nonnegative Springer fibers $\SFtnn{\lambda} \subseteq \Fl_4(\C)$, with cells labeled by elements of $Z_\lambda$ and top-dimensional cells additionally labeled by Richardson tableaux of shape $\lambda$. Left: $\lambda = (2,2)$. Right: $\lambda = (3,1)$.}
\label{figure_tnn}
\end{center}
\end{figure}
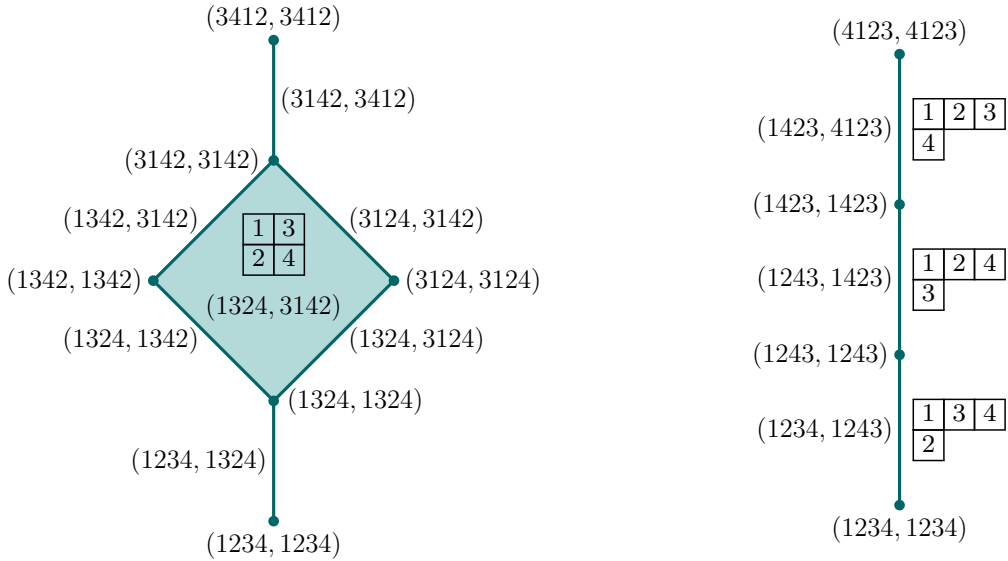

\begin{rmk}
In Lusztig's definition of $Z_\lambda$ from \cite[\S 1.15]{lusztig21}, the condition $w\in\mincoset{S_n}{\lambda}$ is omitted; it is only assumed that $w\in S_n$. However, the other conditions defining $Z_\lambda$ imply that $w\in\mincoset{S_n}{\lambda}$. To see this, proceed by contradiction and suppose that $w\notin\mincoset{S_n}{\lambda}$, i.e., $s_jw \le w$ for some $j\in [n-1]\setminus\{\lambda_1, \lambda_1 + \lambda_2, \dots\}$. Since $v\in\mincoset{S_n}{\lambda}$, we have $v \le s_jv$. Thus $s_jv \le w$ by the lifting property for Coxeter groups \cite[Proposition 2.2.7]{bjorner_brenti05}, contradicting $s_jv\nleq w$.
\end{rmk}

\begin{rmk}
Lusztig's work implies that $\SFtnn{N}$ is a subcomplex of \eqref{decomposition_flag}, and hence it is a regular CW complex by \cite{galashin_karp_lam22}. However, we emphasize that $\SFtnn{N}$ may have multiple top-dimensional cells and is not necessarily pure-dimensional; see \cref{eg.TNNSpringer} and \cref{figure_tnn}. In particular, $\SFtnn{N}$ is not in general homeomorphic to a closed ball. In the positive direction, Bao and He \cite[Theorem 4.2]{bao_heb} show that $\SFtnn{N}$ is contractible, verifying a conjecture of Lusztig \cite[\S 8.16]{lusztig94}.
\end{rmk}

The following observations connect $\SFtnn{\lambda}$ with the work of \cref{sec_geometric}:
\begin{prop}\label{geometry_positivity}
Let $v\le w$ in $S_n$.
\begin{enumerate}[label=(\roman*), leftmargin=*, itemsep=2pt]
\item\label{geometry_positivity_containment} The Richardson variety $\Rvar{v}{w}$ is contained in $\SF{\lambda}$ if and only if $(v,w)\in Z_\lambda$.
\item\label{geometry_positivity_component} The Richardson variety $\Rvar{v}{w}$ is an irreducible component of $\SF{\lambda}$ if and only if $(v,w)\in Z_\lambda$ and $\ell(w) - \ell(v) = n(\lambda)$ \textup{(}i.e.\ we have equality in \eqref{dimension_Z}\textup{)}.
\end{enumerate}
\end{prop}

\begin{proof}
Proof of \ref{geometry_positivity_containment}: ($\Rightarrow$) Suppose that $\Rvar{v}{w} \subseteq \SF{\lambda}$, so that in particular we have $\Rcell{v}{w} \subseteq \SF{\lambda}$. Intersecting both sides with $\Fl_n^{\ge 0}$ gives $\Rtp{v}{w} \subseteq \SFtnn{\lambda}$. Then \eqref{decomposition_springer} implies $(v,w) \in Z_\lambda$.

($\Leftarrow$) Suppose that $(v,w)\in Z_\lambda$. Then $\Rtp{v}{w} \subseteq \SF{\lambda}$. Since the Zariski closure of $\Rtp{v}{w}$ equals $\Rvar{v}{w}$ (see, e.g., \cite[Section 4.4]{marsh_rietsch04}) and $\SF{\lambda}$ is Zariski-closed, we obtain $\Rvar{v}{w} \subseteq \SF{\lambda}$.

Proof of \ref{geometry_positivity_component}: ($\Rightarrow$) Suppose that $\Rvar{v}{w}$ is an irreducible component of $\SF{\lambda}$. Then $(v,w)\in Z_\lambda$ by part \ref{geometry_positivity_containment}, and $\dim(\Rvar{v}{w}) = \dim(\SF{\lambda})$ implies $\ell(w) - \ell(v) = n(\lambda)$.

($\Leftarrow$) Suppose that $(v,w)\in Z_\lambda$ and $\ell(w) - \ell(v) = n(\lambda)$. Then by part \ref{geometry_positivity_containment} we have $\Rvar{v}{w} \subseteq \SF{\lambda}$, and $\dim(\Rvar{v}{w}) = \dim(\SF{\lambda})$. Since $\Rvar{v}{w}$ is irreducible, it is an irreducible component of $\SF{\lambda}$.
\end{proof}

\begin{lem}\label{tp_cells}
Let $\lambda\vdash n$ and $v\le w$ in $S_n$. Then $\Rtp{v}{w}$ is an $n(\lambda)$-dimensional cell of $\SFtnn{\lambda}$ if and only if $(v,w) = (v_\sigma, w_\sigma)$ for some Richardson tableau $\sigma\in\SYT(\lambda)$.
\end{lem}

\begin{proof}
Note that $\Rtp{v}{w}$ is an $n(\lambda)$-dimensional cell of $\SF{\lambda}$ if and only if $(v,w)\in Z_\lambda$ and $\ell(w) - \ell(v) = n(\lambda)$. In turn, this is equivalent to $\Rvar{v}{w}$ being an irreducible component of $\SF{\lambda}$, by \cref{geometry_positivity}\ref{geometry_positivity_component}. Finally, by \cref{main}, this is in turn equivalent to $(v,w)$ being equal to $(v_\sigma, w_\sigma)$ for some $\sigma\in\Rtabs{\lambda}$.
\end{proof}

\begin{thm}\label{tp_main}
Let $\lambda\vdash n$. Then the totally nonnegative Springer fiber $\SFtnn{\lambda}$ is a regular CW complex of dimension $n(\lambda)$, whose top-dimensional cells are precisely $\Rtp{v_\sigma}{w_\sigma}$ for all $\sigma\in\Rtabs{\lambda}$. In particular, the number of top-dimensional cells of $\SFtnn{\lambda}$ is
\[
|\Rtabs{\lambda}| = \binom{\lambda_{\ell-1}}{\lambda_\ell}\binom{\lambda_{\ell-2} + \lambda_\ell}{\lambda_{\ell-1} + \lambda_\ell}\binom{\lambda_{\ell-3} + \lambda_{\ell-1} + \lambda_\ell}{\lambda_{\ell-2} + \lambda_{\ell-1} + \lambda_\ell}\cdots\binom{\lambda_1 + \lambda_3 + \lambda_4 + \cdots + \lambda_\ell}{\lambda_2 + \lambda_3 + \lambda_4 + \cdots + \lambda_\ell},
\]
where $\ell$ denotes $\ell(\lambda)$.
\end{thm}

\begin{proof}
We already know that $\SFtnn{\lambda}$ is a regular CW complex of dimension at most $n(\lambda)$. By \cref{tp_cells}, its $n(\lambda)$-dimensional cells are precisely $\Rtp{v_\sigma}{w_\sigma}$ for all $\sigma\in\Rtabs{\lambda}$. Since $\Rtabs{\lambda} \neq \emptyset$, we get that $\SFtnn{\lambda}$ has full dimension $n(\lambda)$. Finally, the formula for $|\Rtabs{\lambda}|$ is just \eqref{q_count_1} of \cref{q_count}.
\end{proof}

\begin{eg}
We have 
\[
\Rtabs{(2,2)} = \left\{\;\scalebox{0.8}{$\begin{ytableau}1 & 2\\ 3 & 4\end{ytableau}$}\;\right\} \; \text{ and } \; \Rtabs{(3,1)} = \left\{\;\scalebox{0.8}{$\begin{ytableau}1 & 2 & 3\\ 4\end{ytableau}$}\;,\; \scalebox{0.8}{$\begin{ytableau}1 &2 & 4\\ 3 \end{ytableau}$}\;,\; \scalebox{0.8}{$\begin{ytableau}1 & 3 & 4\\ 2\end{ytableau}$}\; \right\}.
\]
Thus by \cref{tp_main} there is one top-dimensional cell in $\SFtnn{(2,2)}$ and three top-dimensional cells in $\SFtnn{(3,1)}$. The reader can confirm that each Richardson tableau labels the corresponding top-dimensional cell in \cref{figure_tnn}.
\end{eg}

We emphasize that \cref{tp_main} only applies to a subset of the totally nonnegative Springer fibers studied by Lusztig, even in type $A$. It does not address $\SFtnn{N}$ when the diagonal blocks of $N$ are not weakly sorted by size or when $N$ has both upper- and lower-triangular blocks. The following example demonstrates that the cell decomposition of $\SFtnn{N}$ depends on $N$. We leave it as an open problem to study $\SFtnn{N}$ in greater generality.
\begin{eg}\label{eg_tnn}
Consider the following nilpotent matrices of Jordan type $\lambda = (2,1,1)$:
\[
N_\lambda = \begin{bmatrix}
0 & 1 & 0 & 0 \\
0 & 0 & 0 & 0 \\
0 & 0 & 0 & 0 \\
0 & 0 & 0 & 0
\end{bmatrix} \quad \text{ and } \quad
N = \begin{bmatrix}
0 & 0 & 0 & 0 \\
0 & 0 & 1 & 0 \\
0 & 0 & 0 & 0 \\
0 & 0 & 0 & 0
\end{bmatrix}, \quad \text{ where } I + N_\lambda,\, I + N \in \GL_4^{\ge 0}.
\]
Then $\SF{N}$ and $\SF{\lambda}$ are isomorphic as complex varieties, but their totally nonnegative parts are not isomorphic as CW complexes. Indeed, $\SFtnn{\lambda}$ is pure of dimension $3$ with maximal cells $\Rtp{1234}{1432}$, $\Rtp{1324}{4132}$, and $\Rtp{1342}{4312}$. On the other hand, $\SFtnn{N}$ is not pure, with maximal cells $\Rtp{1243}{4213}$ and $\Rtp{2134}{2431}$ of dimension $3$ as well as $\Rtp{1234}{2143}$ and $\Rtp{2413}{4231}$ of dimension $2$. (This calculation can be verified from \cite{lusztig21}; we omit the details.)
\end{eg}

As a consequence of the preceding results, we obtain a Bruhat-theoretic characterization of Richardson tableaux: 
\begin{cor}\label{richardson_Z}
Let $\sigma\in\SYT(\lambda)$, where $\lambda\vdash n$. Then the following are equivalent:
\begin{enumerate}[label=(\roman*), leftmargin=*, itemsep=2pt]
\item\label{richardson_Z_richardson} $\sigma$ is a Richardson tableau;
\item\label{richardson_Z_strong} $(v_\sigma, w_\sigma) \in Z_\lambda$; and
\item\label{richardson_Z_weak} $s_jv_\sigma \nleq w_\sigma$ for all $j \in [n-1]\setminus\{\lambda_1, \lambda_1 + \lambda_2, \dots\}$.
\end{enumerate}
\end{cor}

\begin{proof}
\ref{richardson_Z_richardson} $\Rightarrow$ \ref{richardson_Z_strong}: This follows from \cref{tp_main}.

\ref{richardson_Z_strong} $\Rightarrow$ \ref{richardson_Z_richardson}: Suppose that $(v_\sigma, w_\sigma) \in Z_\lambda$. Then combining \eqref{dimension_Z} and \eqref{dimension_bound}, we get $\ell(w_\sigma) - \ell(v_\sigma) = n(\lambda)$. Hence $\sigma$ is Richardson by \cref{thm-lengths}.

\ref{richardson_Z_strong} $\Leftrightarrow$ \ref{richardson_Z_weak}: This follows from \cref{reading_word_bruhat} and the definition of $Z_\lambda$.
\end{proof}

\begin{eg}\label{eg_richardson_Z}
We illustrate using \cref{richardson_Z}\ref{richardson_Z_weak} to verify the Richardson condition for tableaux. Let $\lambda = (4,2,2)$ and $\sigma, \tau \in \SYT(\lambda)$ be as in \cref{eg_richardson_tableau_intro}. Recall from \cref{eg_permutations_intro} that
\[
v_{\sigma} = 15726348 \quad \text{ and } \quad w_{\sigma} = 75182364.
\]
We can check that $\sigma$ is Richardson by verifying that $s_jv_\sigma \nleq w_\sigma$ for $j = 1, 2, 3, 5, 7$:
\begin{gather*}
25716348 \nleq 75182364, \quad 15736248 \nleq 75182364, \quad 15726438 \nleq 75182364, \\[2pt]
16725348 \nleq 75182364, \quad 15826347 \nleq 75182364.
\end{gather*}

On the other hand, we calculate that
\[
v_\tau = 15237684 \quad \text{ and } \quad w_\tau = 71582634.
\]
Notice that $s_2v_\tau \le w_\tau$ (i.e.\ $15327684 \le 71582634$), implying that $\tau$ is not Richardson.
\end{eg}

\begin{rmk}\label{multiply_reflection}
Let $\sigma\in\SYT(\lambda)$, where $\lambda\vdash n$. We point out that the permutations $s_jv_\sigma$ for $j \in [n-1]\setminus\{\lambda_1, \lambda_1 + \lambda_2, \dots\}$ are precisely the permutations obtained from the one-line notation of $v_\sigma$ by swapping two numbers which appear as adjacent entries in the same row of $\sigma^\vee$. We will use this fact in \cref{sec_smoothness}.
\end{rmk}


\section{Smoothness}\label{sec_smoothness}

\noindent Smoothness for irreducible components $\SF{\sigma}$ of type $A$ Springer fibers has been extensively studied. In this section, we show that $\SF{\sigma}$ is smooth for all Richardson tableaux $\sigma$:

\begin{thm}\label{thm.smooth}
All Richardson tableaux are smooth. In particular \textup{(}by \cref{main}\textup{)}, if the irreducible component $\SF{\sigma}$ is equal to a Richardson variety, then $\SF{\sigma}$ is smooth.
\end{thm}

We begin with a brief review of previous work on this topic. For convenience, we say that the standard tableau $\sigma$ is \boldit{smooth} if $\SF{\sigma}$ is smooth, and \boldit{singular} otherwise. By \cref{dual_evacuation}, $\sigma$ is smooth if and only if $\sigma^\vee$ is smooth. Vargas \cite[Theorem 4.5]{vargas79} (cf.\ \cite{spaltenstein82}) showed that all tableaux of hook shape are smooth, and Fung \cite{fung03} showed that all tableaux with at most two rows are smooth. More generally, Fresse and Melnikov \cite[Section 1.2]{fresse_melnikov10} showed that for a given partition $\lambda$, all standard tableaux of shape $\lambda$ are smooth if and only if $\lambda$ is a hook, $\ell(\lambda) \le 3$ and $\lambda_3 \le 1$, or $\lambda = (2,2,2)$. To do so, they used the fact that if $\sigma$ is a smooth tableau then $\bar{\sigma}$ is a smooth tableau, and that the converse holds if $\bar{\sigma}$ is obtained from $\sigma$ by deleting a box in the last column. Fresse and Melnikov \cite[Theorem 1.2]{fresse_melnikov11} also characterized which standard tableaux with at most two columns are smooth.

Additional examples of smooth components include the Richardson components studied by Pagnon and Ressayre \cite{pagnon_ressayre06}, the generalized Richardson components of Fresse \cite{fresse11}, and the components coming from $K$-orbits studied by Graham and Zierau \cite{graham_zierau11}. We will discuss these further in \cref{sec_comparison}. In particular, we will see that Richardson components and the components coming from $K$-orbits are equal to Richardson varieties, and therefore the fact that these components are smooth is a special case of \cref{thm.smooth}. Additionally, since all tableaux of hook shape are Richardson, we also recover the result of Vargas mentioned above that all tableaux of hook shape are smooth.

Examples of singular tableaux include
\begin{align}\label{eg_singular_tableaux}
\;\begin{ytableau}
1 & 3 \\
2 & 5 \\
4 \\
6
\end{ytableau}\;
\quad \text{ and } \quad 
\;\begin{ytableau}
1 & 2 & 5 \\
3 & 4 \\
6 & 7
\end{ytableau}\;
\end{align}
due to Vargas \cite[Section 5]{vargas79} and, respectively, Fresse and Melnikov \cite[Section 2.3]{fresse_melnikov10}. Note that neither of these tableaux are Richardson, in agreement with \cref{thm.smooth}. 

The remainder of the section is devoted to the proof of \cref{thm.smooth}. We begin with two key results about singularities of Schubert and Richardson varieties. Let $\sing(X)$ denote the \boldit{singular locus} of a variety $X$, that is, the closed subvariety of singular points of $X$. Recall that $t_{i,j}\in S_n$ is the reflection which swaps $i$ and $j$ and fixes all other elements of $[n]$.

\begin{thm}[{Deodhar \cite{deodhar85}}]\label{bruhat_graph}
Let $v\le w$ in $S_n$.
\begin{enumerate}[label=(\roman*), leftmargin=*, itemsep=2pt]
\item\label{bruhat_graph_bound} We have
\begin{align}\label{deodhar_inequality}
|\{t_{i,j}\in S_n\mid v < vt_{i,j} \le w\}| \ge \ell(w) - \ell(v).
\end{align}
\item\label{bruhat_graph_equal} The Schubert variety $\Svar{w}$ is smooth at the point $\dot{v}B$ if and only if equality holds in~\eqref{deodhar_inequality}.
\end{enumerate}
\end{thm}

The left-hand side of \eqref{deodhar_inequality} is the number of vertices adjacent to $v$ in the \emph{Bruhat graph} of $[v,w]$; see \cite{billey_lakshmibai00} for further background.
\begin{thm}[{Billey and Coskun \cite[Corollary 2.10]{billey_coskun12}; cf.\ \cite[Corollary 1.2]{knutson_woo_yong13}}]\label{sing_R}
Let $v\le w$ in $S_n$. Then
\[
\sing(\Rvar{v}{w}) = (\sing(\Svaropp{v})\cap\Svar{w}) \cup (\Svaropp{v}\cap\sing(\Svar{w})).
\]
\end{thm}

We mention that an explicit description of $\sing(\Svar{w})$ (which we will not need here) appears in \cite{gasharov01,billey_warrington03, cortez03, kassel_lascoux_reutenauer03, manivel01a}. We will use the results above to obtain a criterion for smoothness of Richardson varieties (\cref{cor.sing_R}).
\begin{lem}\label{smooth_point_check}
Let $v\le w$ in $S_n$.
\begin{enumerate}[label=(\roman*), leftmargin=*, itemsep=2pt]
\item\label{smooth_point_check1} We have $\Svaropp{v}\cap\sing(\Svar{w}) = \emptyset$ if and only if $\Svar{w}$ is smooth at the point $\dot{v}B$.
\item\label{smooth_point_check2} We have $\sing(\Svaropp{v})\cap\Svar{w} = \emptyset$ if and only if $\Svar{w_0v}$ is smooth at the point $\dot{w}_0\dot{w}B$.
\end{enumerate}
\end{lem}

\begin{proof}
Proof of \ref{smooth_point_check1}: ($\Rightarrow$) This follows from the fact that $\dot{v}B \in \Svaropp{v}$.

($\Leftarrow$ by contrapositive) Suppose that $\Svaropp{v}\cap\sing(\Svar{w}) \neq \emptyset$. Since the Schubert varieties are the closures of the $B$-orbits of $\Fl_n(\C)$, we see that $\sing(\Svar{w})$ is a union of Schubert varieties. Hence there exists $u\in S_n$ such that $\Svar{u}\subseteq\sing(\Svar{w})$ and $\Svaropp{v}\cap \Svar{u} \neq \emptyset$. Then \eqref{richardson_nonempty} implies $v \le u$, so $\dot{v}B\in \Svar{u}\subseteq\sing(\Svar{w})$. Thus $\dot{v}B$ is a singular point of $\Svar{w}$.

Proof of \ref{smooth_point_check2}: Since $\dot{w}_0B_-\dot{w}_0 = B$, we have $\dot{w}_0\Svaropp{u} = \Svar{w_0u}$ for all $u\in S_n$. Left translation by $\dot w_0$ is an automorphism of $\Fl_n(\C)$, so
\[
\dot w_0\left( \sing(\Svaropp{v})\cap \Svar{w} \right) = \sing(\Svar{w_0v})\cap \Svaropp{w_0w}.
\]
Hence
\[
\sing(\Svaropp{v})\cap\Svar{w} = \emptyset
\quad\Leftrightarrow\quad
\sing(\Svar{w_0v})\cap\Svaropp{w_0w} = \emptyset
\quad\Leftrightarrow\quad
\dot{w}_0\dot{w}B \notin \sing(\Svar{w_0v}),
\]
where the second equivalence is just part \ref{smooth_point_check1} with $(v,w)$ replaced by $(w_0w, w_0v$).
\end{proof}

\begin{cor}\label{cor.sing_R}
Let $v\le w$ in $S_n$. The Richardson variety $\Rvar{v}{w}$ is smooth if and only if $\Svar{w}$ is smooth at the point $\dot{v}B$ and $\Svar{w_0v}$ is smooth at the point $\dot{w}_0\dot{w}B$.
\end{cor}

\begin{proof}
This follows from \cref{sing_R,smooth_point_check}.
\end{proof}

We now show that $\Rvar{v_\sigma}{w_\sigma}$ is smooth for all Richardson tableaux $\sigma$ by using \cref{bruhat_graph} to verify the two conditions in \cref{cor.sing_R}. We encourage the reader to consult \cref{eg_reflections} to help follow the proofs of \cref{reflections1,reflections2}, which contain the key technical arguments. We also recall that, by definition of the Bruhat order, for all $v\in S_n$ and $(i,j)\in [n]^2$ with $i<j$, we have
\begin{eqnarray}\label{Br_inc}
v< vt_{i,j} \quad\Leftrightarrow\quad v(i)<v(j).
\end{eqnarray}

\begin{lem}\label{reflections1}
Let $\sigma\in\SYT(\lambda)$ be a Richardson tableau, where $\lambda\vdash n$.
\begin{enumerate}[label=(\roman*), leftmargin=*, itemsep=2pt]
\item\label{reflections1_pairs} For all $(i,j)\in [n]^2$ with $i < j$, we have
\begin{align}\label{eqn.R_bruhat1}
v_\sigma < v_{\sigma}t_{i,j} \leq w_\sigma
\end{align}
if and only if in the tableau $\sigma^\vee$,
\begin{align}\label{reflections1_tableau}
\text{$i$ appears in a row above $j$ and is the largest entry in $[j]$ in its row}.
\end{align}
\item\label{reflections1_count} There are exactly $n(\lambda)$ pairs $(i,j)$ satisfying \eqref{reflections1_tableau}.
\item\label{reflections1_smooth} The Schubert variety $\Svar{w_\sigma}$ is smooth at the point $\dot{v}_\sigma B$.
\end{enumerate}
\end{lem}

\begin{proof}
Since $v_\sigma^{-1}$ is the top-down reading word of $\sigma^{\vee}$, \eqref{Br_inc} becomes (for $i < j$)
\[
v_\sigma < v_\sigma t_{i,j}
\quad\Leftrightarrow\quad
v_\sigma(i)<v_\sigma(j)
\quad\Leftrightarrow\quad
\rowjword{\sigma^\vee}{i}\leq \rowjword{\sigma^\vee}{j}. 
\]
In other words, $v_\sigma < v_\sigma t_{i,j}$ if and only if $i$ appears in a row weakly above the row containing $j$ in $\sigma^\vee$. We use this characterization of the inversions of $v_\sigma$ several times below. 

Proof of the forward direction of \ref{reflections1_pairs}: Given $(i,j)\in [n]^2$ with $i<j$ such that \eqref{eqn.R_bruhat1} holds, we show that \eqref{reflections1_tableau} holds. Since $v_\sigma < v_\sigma t_{i,j}$, we have that $i$ appears in a row weakly above $j$ in $\sigma^\vee$.

Now proceed by contradiction and suppose that $i$ is not the largest entry in $[j]$ in its row. Let $k$ denote the entry immediately to the right of $i$ in $\sigma^\vee$, so that $i < k \le j$. By \cref{richardson_Z}\ref{richardson_Z_weak} and \cref{multiply_reflection}, we get $v_\sigma t_{i,k} \nleq w_\sigma$. Since $v_\sigma t_{i,j} \le w_\sigma$, we have $k \neq j$.  Since $k>i$ occurs in the same row of $\sigma^{\vee}$ as $i$ and since this row is weakly above the row containing $j$, we get $v_\sigma(i)<v_\sigma(k)<v_{\sigma}(j)$. Then we calculate that
\begin{eqnarray*}
(v_\sigma t_{i,k}) (i) = v_\sigma (k) &<&v_\sigma(j) = (v_\sigma t_{i,k})(j), \\
(v_\sigma t_{i,k}t_{i,j}) (k) = v_\sigma (i)&<&v_\sigma(k) = (v_\sigma t_{i,k}t_{i,j})(j),
\end{eqnarray*}
and applying~\eqref{Br_inc} for each inequality we conclude 
\[
v_\sigma t_{i,k} < v_\sigma t_{i,k}t_{i,j} < v_\sigma t_{i,k}t_{i,j}t_{k,j}. 
\]
Now $t_{i,k}t_{i,j}t_{k,j} = t_{i,j}$, so $v_\sigma t_{i,k} < v_\sigma t_{i,j}$. By~\eqref{eqn.R_bruhat1} we have $v_\sigma t_{i,j} \le w_\sigma$, so $v_\sigma t_{i,k} \le w_\sigma$, a contradiction.

Proof of \ref{reflections1_count}: Note that for a fixed $j\in [n]$, the number of $i < j$ such that $(i,j)$ satisfies \eqref{reflections1_tableau} is exactly $\rowjword{\sigma^\vee}{j} - 1$ (since we get one such $i$ for each row strictly above $j$ in $\sigma^\vee$). Hence the total number of such pairs $(i,j)$ is
\[
\sum_{j=1}^n(\rowjword{\sigma^\vee}{j} - 1) = \sum_{i=1}^{\ell(\lambda)}(i-1)\lambda_i = n(\lambda).
\]

Proof of \ref{reflections1_pairs}: Note that the reflections $t_{i,j}$ satisfying \eqref{eqn.R_bruhat1} are precisely those enumerated on the left-hand side of Deodhar's inequality \eqref{deodhar_inequality} when $(v,w) = (v_\sigma,w_\sigma)$. Also, the right-hand side of \eqref{deodhar_inequality} is $\ell(w_\sigma)-\ell(v_\sigma) = n(\lambda)$ by \cref{thm-lengths}. Hence \eqref{deodhar_inequality} implies that there are at least $n(\lambda)$ reflections $t_{i,j}$ satisfying \eqref{eqn.R_bruhat1}. On the other hand, we know from part \ref{reflections1_count} that exactly $n(\lambda)$ pairs $(i,j)$ satisfy \eqref{reflections1_tableau}, and by the forward direction of part \ref{reflections1_pairs}, no other pairs satisfy \eqref{eqn.R_bruhat1}. Thus \eqref{eqn.R_bruhat1} and \eqref{reflections1_tableau} are equivalent, and moreover, equality holds in \eqref{deodhar_inequality}.

Proof of \ref{reflections1_smooth}: This follows from \cref{bruhat_graph}\ref{bruhat_graph_equal}, since equality holds in \eqref{deodhar_inequality} (as we just proved).
\end{proof}

\begin{lem}\label{reflections2}
Let $\sigma\in\SYT(\lambda)$ be a Richardson tableau, where $\lambda\vdash n$.
\begin{enumerate}[label=(\roman*), leftmargin=*, itemsep=2pt]
\item\label{reflections2_pairs} For all $(i,j)\in [n]^2$ with $i < j$, we have
\begin{align}\label{eqn.R_bruhat2}
w_\sigma w_0 < w_\sigma w_0t_{i,j} \leq v_\sigma w_0
\end{align}
if and only if in the tableau $\sigma$,
\begin{align}\label{reflections2_tableau}
\text{$i$ appears in a row above $j$ and is the largest entry in $[j]$ in its row}.
\end{align}
\item\label{reflections2_count} There are exactly $n(\lambda)$ pairs $(i,j)$ satisfying \eqref{reflections2_tableau}.
\item\label{reflections2_smooth} The Schubert variety $\Svar{w_0v_\sigma}$ is smooth at the point $\dot{w}_0\dot{w}_\sigma B$.
\end{enumerate}
\end{lem}

\begin{proof}
We follow a similar approach to the proof of \cref{reflections1}. Recall from \cref{v_w_relation} that $w_\sigma w_0 = w_{0,\lambda} v_{\sigma^{\vee}}$. Because $v_{\sigma^{\vee}}^{-1}$ is the top-down reading word of $\sigma$, the permutation $(w_{0,\lambda} v_{\sigma^{\vee}})^{-1} = v_{\sigma^\vee}^{-1}w_{0,\lambda}$ is obtained by reading the entries of $\sigma$ row by row from top to bottom, and within each row from right to left (rather than left to right). As a result, from \eqref{Br_inc} we obtain (for $i < j$)
\begin{align}\label{Br_inc_rewritten2}
w_{0,\lambda}v_{\sigma^\vee} < w_{0,\lambda}v_{\sigma^\vee}t_{i,j}
\quad\Leftrightarrow\quad
w_{0,\lambda}v_{\sigma^\vee}(i)< w_{0,\lambda}v_{\sigma^\vee}(j)
\quad\Leftrightarrow\quad
\rowjword{\sigma}{i}< \rowjword{\sigma}{j}.
\end{align}
In other words, $w_{0,\lambda}v_{\sigma^\vee} < w_{0,\lambda}v_{\sigma^\vee}t_{i,j}$ if and only if $i$ appears in a row strictly above $j$ in $\sigma$.

Proof of the forward direction of \ref{reflections2_pairs}: Given $(i,j)\in [n]^2$ with $i<j$ such that \eqref{eqn.R_bruhat2} holds, we show that \eqref{reflections2_tableau} holds.  Using \cref{v_w_relation}, we rewrite \eqref{eqn.R_bruhat2} as
\begin{align}\label{eqn.R_bruhat2_rewritten}
w_{0,\lambda}v_{\sigma^\vee} < w_{0,\lambda}v_{\sigma^\vee}t_{i,j} \le w_{0,\lambda}w_{\sigma^\vee}.
\end{align}
By \eqref{Br_inc_rewritten2}, we have that $i$ appears in a row strictly above $j$ in $\sigma$.

We consider the projection of the second inequality of \eqref{eqn.R_bruhat2_rewritten} to $\mincoset{S_n}{\lambda}$, which preserves the inequality by \cref{bruhat_projection}.  Since $w_{0,\lambda}\in S_\lambda$ and $w_{\sigma^\vee}\in\mincoset{S_n}{\lambda}$ (see \cref{reading_word_bruhat}), we obtain
\begin{align}\label{eqn.R_bruhat2_rewritten2}
\mincoset{(v_{\sigma^\vee}t_{i,j})}{\lambda} \le w_{\sigma^\vee}.
\end{align}
The one-line notation of $\mincoset{(v_{\sigma^\vee}t_{i,j})}{\lambda}$ is obtained from that of $v_{\sigma^{\vee}}t_{i,j}$ by swapping the entries of $v_{\sigma^{\vee}}t_{i,j}$ in positions $i$ and $j$ and then reordering the elements within each set~\eqref{eqn.Young} so that they increase from left to right (cf.\ \cref{eg_mincoset}). 

In order to give a concrete formula for $\mincoset{(v_{\sigma^\vee}t_{i,j})}{\lambda}$, we introduce some notation. Let $i < i_2 < \cdots < i_a$ denote the entries of $\sigma$ in the same row as $i$ in the interval $[i,j]$. Our assumption that $i$ is the not the largest entry in $[j]$ in its row implies $a\ge 2$. Similarly, let $j > j_2 > \cdots > j_b$ denote the entries of $\sigma$ in the same row as $j$ in the interval $[i,j]$, where $b\ge 1$. We also set $i^* = v_{\sigma^{\vee}}(i)$ and $j^* = v_{\sigma^{\vee}}(j)$. As $i<i_2<\cdots<i_a$ are consecutive entries in the same row of $\sigma$, we have
\begin{align}\label{eqn.x.set}
\{v_{{\sigma}^{\vee}}(i), v_{{\sigma}^{\vee}}(i_2), \ldots, v_{\sigma^\vee}(i_a) \} = \{i^*, i^*+1, \ldots, i^*+a-1\}
\end{align}
and $s_{i^*}, s_{i^* + 1}, \ldots, s_{i^*+a-2} \in S_\lambda$. Similarly, 
\begin{align}\label{eqn.y.set}
\{v_{\sigma^\vee}(j_b), \ldots, v_{\sigma^\vee}(j_2), v_{\sigma^\vee}(j) \} = \{ j^*-b+1, \ldots, j^*-1, j^* \}
\end{align}
and $s_{j^*-b+1}, \dots, s_{j^*-2}, s_{j^*-1} \in S_\lambda$. 
The permutation 
\[
x = s_{i^*+a-2}\cdots s_{i^*+1}s_{i^*} \in S_\lambda
\]
is the cycle $(i^*+a-1, i^*+a-2, \ldots, i^*+1, i^* )$, and 
\[
y: = s_{j^*-b+1}\cdots s_{j^*-2}s_{j^*-1} \in S_\lambda
\]
is the cycle $(j^*-b+1, j^*-b+2, \ldots, j^*-1, j^*)$, where $y = e$ if $b=1$. Because $i$ appears in a row strictly above $j$ in $\sigma$, the sets \eqref{eqn.x.set} and \eqref{eqn.y.set} are disjoint and $i^*+a-1 < j^*-b+1$. In particular, $x$ and $y$ commute and $\ell(xy) = \ell(x) + \ell(y)$. 
With this notation in place, we get
\begin{align}\label{eqn.cosetrep}
\mincoset{(v_{\sigma^\vee}t_{i,j})}{\lambda} = x y v_{\sigma^{\vee}} t_{i,j}. 
\end{align}

Since $a\geq 2$ we have $s_{i^*}\leq x$, and as $v_{\sigma^{\vee}}\in\mincoset{S_n}{\lambda}$, we conclude
\begin{align}\label{eqn.Br_leq}
s_{i^*} v_{\sigma^{\vee}} \leq x y v_{\sigma^{\vee}}.
\end{align}
Also, because $x$ and $y$ commute, $x(j^*)=j^*$, and $y(i^*)=i^*$, we get
\begin{align}\label{last_step}
xy v_{\sigma^{\vee}} (i) = x (i^*) = i^*+a -1 < j^*-b+1 = y(j^*) = xy v_{\sigma^{\vee}} (j),
\end{align}
which implies $ xy v_{\sigma^{\vee}} \leq xy v_{\sigma^{\vee}} t_{i,j}$ by \eqref{Br_inc}. Combining \eqref{eqn.Br_leq}, \eqref{last_step}, \eqref{eqn.cosetrep}, and \eqref{eqn.R_bruhat2_rewritten2} yields
\[
s_{i^*} v_{\sigma^{\vee}} \leq xy v_{\sigma^{\vee}} \leq xy v_{\sigma^{\vee}} t_{i,j} = \mincoset{(v_{\sigma^\vee}t_{i,j})}{\lambda} \leq w_{\sigma^\vee}.
\]
The fact that $s_{i^*} v_{\sigma^{\vee}}\leq w_{\sigma^\vee}$ contradicts \cref{richardson_Z}\ref{richardson_Z_weak}, since $s_{i^*}\in S_\lambda$. This completes the proof of the forward direction of part \ref{reflections2_pairs}.

Now just as in the proof of \cref{reflections1} (since $\sigma$ and $\sigma^\vee$ have the same shape), we obtain parts \ref{reflections2_pairs} and \ref{reflections2_count}, along with the fact that equality holds in \eqref{deodhar_inequality} when $(v,w) = (w_\sigma w_0, v_\sigma w_0)$. Since conjugating by $w_0$ is an automorphism of the Bruhat order on $S_n$, equality also holds in \eqref{deodhar_inequality} when $(v,w) = (w_0 w_\sigma, w_0 v_\sigma)$. Then part \ref{reflections2_smooth} follows from \cref{bruhat_graph}\ref{bruhat_graph_equal}.
\end{proof}

\begin{eg}\label{eg_reflections}
We illustrate \cref{reflections1,reflections2} along with their proofs. We adopt the setup of \cref{eg_permutations_intro,eg_permutations}. Recall that $\lambda = (4,2,2)$ and
\[
\sigma = \;\begin{ytableau}
1 & 3 & 4 & 6 \\
2 & 7 \\
5 & 8
\end{ytableau}\;
,\quad
\sigma^{\vee} = 
\;\begin{ytableau}
1 & 4 & 6 & 7 \\
2 & 5 \\
3 & 8
\end{ytableau}\;
,\quad
v_{\sigma} = 15726348,
\quad
w_{\sigma} = 75182364.
\]
We calculate that the reflections $t_{i,j}$ satisfying $v_\sigma < v_{\sigma}t_{i,j} \leq w_\sigma$ are indexed by
\[
(i,j) = (1,2),\, (1,3),\, (2,3),\, (4,5),\, (5,8),\, (7,8),
\]
in agreement with \cref{reflections1}\ref{reflections1_pairs}. Note that there are precisely $n(\lambda) = 6$ such reflections, in agreement with \cref{reflections1}\ref{reflections1_count}. 

Let us illustrate why $(i,j) = (4,8)$ is not such a reflection, following the proof of \cref{reflections1}. In this case, $4$ is not the largest entry of $[8]$ in its row in $\sigma^\vee$, and $k=6$ is immediately to its right. Since $i=4$ and $k=6$ are in the same row, we know from \cref{richardson_Z}\ref{richardson_Z_weak} and \cref{multiply_reflection} that $15736248=v_\sigma t_{i,k} \nleq w_\sigma$. On the other hand, $v_\sigma t_{i,j} = 15786342$ is obtained from $v_\sigma t_{i,k}=15736248$ by multiplying on the right by $t_{i,j}$ and then $t_{k,j}$, both of which move up in the Bruhat order. Thus, we have $v_\sigma t_{i,k}=15736248 < 15786342= v_\sigma t_{i,j} $, and this implies $v_\sigma t_{i,j} = 15786342 \nleq w_\sigma$ (otherwise we would obtain a contradiction to $v_\sigma t_{i,k} \nleq w_\sigma$).

Now we consider the setup of \cref{reflections2}. We have
\[
w_\sigma w_0 = w_{0,\lambda}v_{\sigma^\vee} = 46328157 \quad \text{ and } \quad v_\sigma w_0 = w_{0,\lambda}w_{\sigma^\vee} = 84362751.
\]
There are precisely $n(\lambda) = 6$ reflections $t_{i,j}$ satisfying $w_\sigma w_0 < w_\sigma w_0t_{i,j} \leq v_\sigma w_0$, indexed by
\[
(i,j) = (1,2),\, (2,5),\, (4,5),\, (6,7),\, (6,8),\, (7,8),
\]
in agreement with \cref{reflections2}. To see that $(i,j) = (3,8)$ is not such a reflection as in the proof of \cref{reflections2}, suppose for a contradiction that $w_\sigma w_0t_{3,8} \leq v_\sigma w_0$. As in the proof, projecting to $\mincoset{S_8}{\lambda}$ implies that $\mincoset{(v_{\sigma^\vee}t_{3,8})}{\lambda} \leq w_{\sigma^\vee}$. The key idea now is to express $\mincoset{(v_{\sigma^\vee}t_{3,8})}{\lambda}$ in terms of $v_{\sigma^\vee}t_{3,8}$.

We have $v_{\sigma^{\vee}} = 15237468$ and $v_{\sigma^{\vee}} t_{3,8} = 15837462$, so (as in \cref{eg_mincoset}) we calculate that
\[
\mincoset{(v_{\sigma^\vee}t_{3,8})}{\lambda} = \mincoset{(15837462)}{\lambda} = 15728364 = xy v_{\sigma^{\vee}} t_{3,8},
\]
where $x = s_3s_2$ and $y = s_7$. By assumption, $3$ is not the largest element in $[8]$ in its row in $\sigma$; the entry immediately to its right is $4$. Hence in the permutation $v_{\sigma^{\vee}} t_{3,8} = 15837462$, the numbers that were originally in positions $3$ and $4$ in $v_{\sigma^\vee}$ (i.e.\ $2$ and $3$), are out of order and need to be sorted to obtain $\mincoset{(v_{\sigma^\vee}t_{3,8})}{\lambda}$. This means that $s_2 \leq x = s_3s_2$, and we obtain
\[
s_2 v_{\sigma^{\vee}} \leq xy v_{\sigma^\vee} \leq xy v_{\sigma^{\vee}} t_{3,8} = \mincoset{(v_{\sigma^\vee}t_{3,8})}{\lambda} \leq w_{\sigma^\vee}.
\]
This contradicts \cref{richardson_Z}\ref{richardson_Z_weak}, since $s_2\in S_\lambda$.
\end{eg}

The proof of~\cref{thm.smooth} follows handily from the previous two lemmas.

\begin{proof}[Proof of~\cref{thm.smooth}]
Let $\sigma$ be a Richardson tableau. By \cref{cor.sing_R}, we must check that $\Svar{w_\sigma}$ is smooth at the point $\dot{v}_\sigma B$ and $\Svar{w_0v_\sigma}$ is smooth at the point $\dot{w}_0\dot{w}_\sigma B$. These statements are verified in \cref{reflections1}\ref{reflections1_smooth} and \cref{reflections2}\ref{reflections2_smooth}.
\end{proof}


\section{Cohomology classes}\label{sec_cohomologyclass}

\noindent In this section we consider the cohomology classes of the irreducible components $\SF{\sigma}$ for Richardson tableaux $\sigma$. We show that each of these classes is a product of Schubert classes in \cref{cor.Schubert}. We then make this formula more combinatorially explicit in the case of hook-shaped tableaux $\sigma$ in \cref{thm.Guemes}, following work of G\"{u}emes \cite{guemes89}.

We begin by recalling some background. Each subvariety $X$ of $\Fl_n(\C)$ determines a cohomology class $[X]\in H^{\mathrm{codim}(X)}(\Fl_n(\C);\mathbb{Z})$. The \boldit{Borel presentation} identifies the cohomology ring of $\Fl_n(\C)$ as
\begin{eqnarray}\label{eqn.Borel}
H^*(\Fl_n(\C);\mathbb{Z}) \cong \mathbb{Z}[x_1, \ldots, x_n]/ I^+,
\end{eqnarray}
where $I^+$ denotes the ideal in $\mathbb{Z}[x_1, \ldots, x_n]$ generated by constant-free symmetric polynomials.  The collection of classes of the Schubert varieties $\Svar{w}$ for $w\in S_n$ form an integral basis of $H^*(\Fl_n(\C);\mathbb{Z})$. Furthermore, the assignment $[\Svar{w}]\mapsto \mathfrak{S}_{w_0w}(x_1, \ldots, x_n)$ of each Schubert class to the \boldit{Schubert polynomial} $\mathfrak{S}_{w_0w}(x_1, \ldots, x_n)$ yields the isomorphism~\eqref{eqn.Borel}. The literature on Schubert polynomials and their remarkable combinatorial properties is vast (see, for example, \cite{manivel01, abe_billey16}); we will not define them here. The study of Schubert polynomials, and in particular how to expand a product $\mathfrak{S}_v\cdot\mathfrak{S}_w$ in the basis of Schubert polynomials, forms the combinatorial foundation for the \boldit{Schubert calculus} of the flag variety $\Fl_n(\C)$.

The following is a restatement of a problem posed by Springer~\cite[Problem 4]{springer83}, which was originally phrased in terms of homology classes:

\begin{prob}\label{prob_cohomology}
For each standard tableaux $\sigma$, find a combinatorial rule for expanding the cohomology class $[\SF{\sigma}]$ in the Schubert basis.
\end{prob}

As far as we are aware, only two partial solutions to Springer's problem appear in the literature. First, G\"uemes \cite{guemes89} gave a combinatorial rule for the Schubert expansion of $[\SF{\sigma}]$ for all tableaux $\sigma$ of hook shape, which we explore in detail in \cref{thm.Guemes} below.
Second, Graham and Zierau \cite[Theorems 4.6 and 4.7]{graham_zierau11} apply localization techniques in equivariant cohomology and $K$-theory to obtain the cohomology classes for Springer fiber components coming from $K$-orbits. We discuss these components further, and the approach of Graham--Zierau for computing $[\SF{\sigma}]$, in \cref{GZ_components}.

The rest of this section focuses on \cref{prob_cohomology} in the case that $\sigma$ is a Richardson tableau. Since $\SF{\sigma} = \Rvar{v_\sigma}{w_\sigma}$ by \cref{main}, \cref{prob_cohomology} is equivalent to computing $[\Rvar{v_\sigma}{w_\sigma}]$. The following result explains the importance of Richardson varieties to Schubert calculus:
\begin{prop}[{cf.\ \cite[Section 6.5]{speyer}}]\label{prop.Richardson.class}
Let $v\le w$ in $S_n$. The isomorphism \eqref{eqn.Borel} sends the cohomology class $[\Rvar{v}{w}]$ to the product of Schubert polynomials $\mathfrak{S}_{v}\cdot \mathfrak{S}_{w_0w}$.
\end{prop}

Thus, combining \cref{main} and \cref{prop.Richardson.class}, we obtain the following result and subsequent reformulation of \cref{prob_cohomology} in terms of Schubert calculus:
\begin{cor}\label{cor.Schubert}
Let $\sigma$ be a Richardson tableau of size $n$. The isomorphism \eqref{eqn.Borel} sends the cohomology class $[\SF{\sigma}]$ to the product of Schubert polynomials $\mathfrak{S}_{v_\sigma}\cdot \mathfrak{S}_{w_0w_\sigma}$.
\end{cor} 

\begin{prob}[Reformulation of \cref{prob_cohomology} for Richardson tableaux]\label{prob_Rcohomology}
Let $\sigma$ be a Richardson tableau. Find a combinatorial rule for expanding the product $\mathfrak{S}_{v_\sigma}\cdot \mathfrak{S}_{w_0w_\sigma}$ in the basis of Schubert polynomials. 
\end{prob}

\begin{rmk}
After hearing a talk about a preliminary version of this paper, Hunter Spink and Vasu Tewari informed us they found an answer to \cref{prob_Rcohomology}, achieved by counting certain chains in the $k$-Bruhat order of $S_n$~\cite{spink_tewari25}.
\end{rmk}

As noted in \cref{explicit}\ref{explicit_hook}, every tableau of hook shape is Richardson. Furthermore, we will show in \cref{GZ_components} that the components from \cite{graham_zierau11} are a subset of those corresponding to Richardson tableaux. Thus, the work of both \cite{guemes89,graham_zierau11} answering \cref{prob_cohomology} in special cases also yields solutions to the Schubert calculus problem of \cref{prob_Rcohomology} in these cases.

We conclude this section by reformulating the main result of G\"uemes~\cite{guemes89} in terms of Schubert polynomials (see \cref{thm.Guemes}). First we set up some notation. For $m\in\mathbb{N}$, let $\delta^{(m)} \vdash \binom{m}{2}$ denote the partition of staircase shape $(m-1, m-2, \dots, 1)$. A \boldit{G\"{u}emes tableau} $\tau$ is a filling of the boxes of the diagram of $\delta^{(m)}$ with (possibly repeated) entries in $[n-1]$, such that entries strictly increase across rows from left to right and down columns from top to bottom, and also weakly increase along diagonals from southwest to northeast.
Letting $\tau_{i,j}$ denote the entry of $\tau$ in row $i$ and column $j$, we define
\[
x_{\tau}= c_1c_2\cdots c_m \in S_n, \quad \text{ where } c_j = s_{\tau_{m-j,j}} s_{\tau_{m-j-1,j}} \cdots s_{\tau_{1,j}}.
\]
We call $\tau$ \boldit{reduced} if $\ell(x_\tau) = \binom{m}{2}$, i.e., $c_1c_2\cdots c_m$ is a reduced word for $x_\tau$.

\begin{eg}\label{eg_strict}
Consider the two G\"{u}emes tableaux of shape $\delta^{(4)}$ below.
\[
\tau = \;\begin{ytableau}
1 & 3 & 4 \\
2 & 4 \\
3
\end{ytableau}\;
\quad\quad\quad
\tau' = \;\begin{ytableau}
1 & 3 & 4 \\
2 & 4 \\
4
\end{ytableau}\;
\]
Then $\tau$ is reduced, since $x_\tau = s_3s_2s_1s_4s_3s_4$ has length $6$, whereas $\tau'$ is not reduced, since $x_{\tau'} = s_4s_2s_1s_4s_3s_4 = s_2s_1s_3s_4$ has length $4$.
\end{eg}

\begin{thm}[{G\"uemes \cite[Theorem 3.1.1]{guemes89}}]\label{thm.Guemes}
Let $\sigma$ be a standard tableau of hook shape $(k, 1^{n-k}) \vdash n$. Then
\[
\mathfrak{S}_{v_\sigma} \cdot \mathfrak{S}_{w_0w_\sigma} = \sum_{\tau} \mathfrak{S}_{w_0x_\tau},
\]
where the sum is over all reduced G\"{u}emes tableaux of shape $\delta^{(n-k+1)}$ with entries in $[n-1]$, whose first row consists of the entries $w_\sigma^{-1}(n) < w_{\sigma}^{-1}(n-1) < \cdots < w_\sigma^{-1}(n-k+1)$.
\end{thm}

\begin{proof}
The statement is a reformulation of \cite[Theorem 3.1.1]{guemes89}. Indeed, the conventions of \cite{guemes89} are that the tableaux indexing components of $\SF{\lambda}$ are standard in the sense that their entries strictly decrease across rows and down columns.  Replacing the `decreasing' standard tableau $\tilde{\sigma}$ with $w_0\cdot \tilde{\sigma}$ returns us to the conventions in this paper. Furthermore, the entries of the first row of $\tau$ are exactly as stated in \cite[Theorem 3.1.1]{guemes89}, using the fact that $w_\sigma^{-1}w_0$ is the reading word of $\tilde{\sigma}$. Finally, \cite[Theorem 3.1.1]{guemes89} computes the homology class of each component of $\SF{N}$. We obtain the reformulation above by combining the fact that the cohomology class of $\Svar{x_\tau}$ is represented by $\mathfrak{S}_{w_0x_\tau}$ with \cref{prop.Richardson.class}.
\end{proof}

\begin{rmk}\label{evacuation_hook}
We point out that the evacuation of a tableau $\sigma$ of hook shape $(k, 1^{\ell-1}) \vdash n$ is particularly simple to compute. Namely, for $2 \le j \le n$ let $j'$ denote $n+2-j$. If the entries of the first row of $\sigma$ are $1 < a_2 < a_3 < \cdots < a_k$, then the first row of $\sigma^\vee$ is $1 < a_k' < \cdots < a_3' < a_2'$. Similarly, if the entries of the first column of $\sigma$ are $1 < b_2 < b_3 < \cdots < b_\ell$, then the first column of $\sigma^\vee$ is $1 < b_\ell' < \cdots < b_3' < b_2'$.
\end{rmk}

\begin{eg}\label{eg.hook.Schubert}
Consider the standard tableaux of hook shape $(4,1^3)\vdash 7$ below:
\[
\sigma = \;\begin{ytableau}
1 & 4 & 5 & 6 \\
2 \\
3 \\
7\end{ytableau}\;
\quad\text{ and }\quad
\sigma^{\vee} = 
\;\begin{ytableau}
1 & 3 & 4 & 5 \\
2 \\
6 \\
7
\end{ytableau}\;.
\]
We have $v_\sigma = 1523467$ and $w_\sigma = 7123654$. We use \cref{thm.Guemes} to expand $[\SF{\sigma}]$ in the Schubert basis, by summing over all reduced G\"{u}emes tableaux $\tau$ of shape $\delta^{(4)}$ with first row equal to $1 < 5 < 6$. There are four such $\tau$'s, shown below along with the associated permutation $x_\tau$.\vspace*{4pt}
\begin{center}\setlength{\tabcolsep}{12pt}
\begin{tabular}{cccc}
\;\begin{ytableau}
1 & 5 & 6 \\
2 & 6 \\
3
\end{ytableau}\;
&
\;\begin{ytableau}
1 & 5 & 6 \\
2 & 6 \\
4
\end{ytableau}\;
&
\;\begin{ytableau}
1 & 5 & 6 \\
3 & 6 \\
4
\end{ytableau}\;
&
\;\begin{ytableau}
1 & 5 & 6 \\
4 & 6 \\
5
\end{ytableau}\;\\
\rule[24pt]{0pt}{0pt}$\begin{aligned}
x_\tau &= s_3s_2s_1s_6s_5s_6\\
&= 4123765
\end{aligned}$
&
$\begin{aligned}
x_\tau &= s_4s_2s_1s_6s_5s_6\\
&= 3125764
\end{aligned}$
&
$\begin{aligned}
x_\tau &= s_4s_3s_1s_6s_5s_6\\
&= 2153764
\end{aligned}$
&
$\begin{aligned}
x_\tau &= s_5s_4s_1s_6s_5s_6\\
&= 2136754
\end{aligned}$
\end{tabular}
\end{center}
Multiplying each $x_\tau$ on the left by $w_0$, we obtain
\[
\mathfrak{S}_{v_\sigma} \cdot \mathfrak{S}_{w_0w_\sigma}
=
\mathfrak{S}_{1523467} \cdot \mathfrak{S}_{1765234}
=
\mathfrak{S}_{4765123} + \mathfrak{S}_{5763124} + \mathfrak{S}_{6735124} + \mathfrak{S}_{6752134}.
\]
We can also verify that $\sigma$ and $\sigma^\vee$ are related as in \cref{evacuation_hook}.
\end{eg}


\section{Comparison with other special components}\label{sec_comparison}

\noindent In this section, we compare the irreducible components of Springer fibers which are equal to Richardson varieties with three other special families of irreducible components: the Richardson components of Pagnon and Ressayre \cite{pagnon_ressayre06}, the generalized Richardson components of Fresse \cite{fresse11}, and the components associated to $K$-orbits studied by Graham and Zierau \cite{graham_zierau11}.


\subsection{Richardson components}
Pagnon and Ressayre \cite{pagnon_ressayre06} studied \boldit{Richardson components} of $\SF{\lambda}$. By definition, these are irreducible components $\SF{\sigma}$ of $\SF{\lambda}$ which are homogeneous under the action of a parabolic subgroup of $\GL_n(\C)$.  
Equivalently, they are components isomorphic to products of complete flag varieties $\Fl_m(\C)$. (See the discussion in \cite[Section 2.1.3]{fresse11} for more detail.)
We emphasize that this is a \emph{strictly weaker} condition than $\SF{\sigma}$ being a Richardson variety. It follows from \cite[Theorem 7.2]{pagnon_ressayre06} that the component $\SF{\sigma}$ is a Richardson component if and only if the standard tableaux $\sigma$ is a concatenation of single columns, such as
\begin{equation}\label{eqn.Richardson_component}
\sigma = \;\begin{ytableau}
1 \\
2 \\
3
\end{ytableau}\;
\circ
\;\begin{ytableau}
1 \\
2
\end{ytableau}\;
\circ
\;\begin{ytableau}
1 \\
2 \\
3 \\
4
\end{ytableau}\;
\circ
\;\begin{ytableau}
1 \\
2 \\
3
\end{ytableau}\; = \;\begin{ytableau}
1 & 4 & 6 & 10 \\
2 & 5 & 7 & 11 \\
3 & 8 & 12\\
9
\end{ytableau}\;.
\end{equation}
In particular, by \cref{prime_concatenation}\ref{prime_concatenation_forward}, if $\SF{\sigma}$ is a Richardson component then $\sigma$ is a Richardson tableau. Conversely, if $\sigma$ denotes the prime Richardson tableau
\[
\sigma = \;\begin{ytableau}
1 & 3 \\
2 \\
4
\end{ytableau}\;
\]
from \eqref{nonexample}, then $\SF{\sigma}$ is not a Richardson component. Although $\SF{\sigma}$ is equal to a Richardson variety, it is not homogeneous under the action of a parabolic subgroup of $\GL_4(\C)$. 

By definition, every Richardson component is smooth. More generally, Pagnon and Ressayre \cite[Theorem 4]{pagnon_ressayre06} show that if $\SF{\sigma}$ is a Richardson component and $2 \le j \le n$ is such that $\rowjword{\sigma}{j} < \rowjword{\sigma}{j-1}$, then $\SF{\tau}$ is also smooth, where $\tau$ is obtained from $\sigma$ by swapping the entries $j-1$ and $j$. Note that such a $\tau$ need not be a Richardson tableau (for example, let $\sigma$ be as in \eqref{eqn.Richardson_component} and take $j=6$). Thus, \cite[Theorem 4]{pagnon_ressayre06} is not a special case of our \cref{thm.smooth}. It would be interesting to consider components corresponding to tableaux obtained similarly from any Richardson tableau:
\begin{prob}\label{prob_swap}
Let $\sigma$ be a Richardson tableau of size $n$, and let $\tau$ be obtained from $\sigma$ by swapping the entries $j-1$ and $j$, where $2 \le j \le n$. Under what conditions \textup{(}on $j$ and $\sigma$\textup{)} is $\SF{\tau}$ smooth?
\end{prob}


\subsection{Generalized Richardson components}
Fresse \cite[Section 2.3]{fresse11} studied a family of irreducible components generalizing the Richardson components. We briefly describe these components here. In \cite[Section 2.3.1]{fresse11}, Fresse introduces particular orbits of the centralizer subgroup $Z(N_\lambda) = \{g\in \GL_n(\C) \mid gN_\lambda g^{-1} = N_\lambda\}$ in $\GL_n(\C)$ on the Springer fiber $\SF{\lambda}$, called \emph{Jordan orbits}. We call $\SF{\sigma}$ a \boldit{generalized Bala--Carter component} if it contains a dense Jordan orbit.  Furthermore, we call $\SF{\sigma}$ a \boldit{generalized Richardson component} if $\SF{\transpose{\sigma}}$ is a generalized Bala--Carter component, where $\transpose{\sigma}$ denotes the conjugate (or transpose) of the tableau $\sigma$. Every Richardson component is a generalized Richardson component. Also, generalized Richardson components are all smooth, since they are iterated bundles over projective spaces \cite[Theorem 2.7]{fresse11}.

The family of generalized Richardson components is different from the family of irreducible components equal to Richardson varieties, and neither is contained in the other. Indeed, Fresse showed that $\SF{\sigma}$ is a generalized Richardson component for all standard tableaux $\sigma$ with at most two rows, but not every tableau with at most two rows is a Richardson tableau (see \cref{explicit}\ref{explicit_two_row}). Conversely, one can verify from \cite[Proposition 2.6]{fresse11} that if $\sigma$ denotes the Richardson tableau from \eqref{nonexample}, then $\SF{\sigma}$ is not a generalized Richardson component.


\subsection{Components associated to \texorpdfstring{$K$}{K}-orbits}\label{GZ_components}

Graham and Zierau \cite{graham_zierau11} studied components of Springer fibers $\SF{\lambda}$ associated to closed orbits of the action of $K=\GL_p(\C) \times \GL_q(\C)$ on $\Fl_n(\C) \cong G/B$ (where $p+q= n$ and $G = \GL_n(\C)$). These components were introduced by Barchini and Zierau \cite{barchini_zierau08}, who studied their interaction with the representation theory of $\mathrm{SU}(p,q)$ \cite{barchini_zierau08}. We briefly recall the definition of these components for $G=\GL_n(\C)$. Let $\mu: T^*(G/B) \to \mathcal{N}$ denote the Springer resolution, where $\mathcal{N}$ is the nilpotent cone of $\mathfrak{gl}_n(\C)$; see \cite[Section 6]{jantzen04} for an expository introduction to this topic. Then $\mu^{-1}(N_\lambda) \cong \SF{\lambda}$. The group $K$ acts with finitely many orbits on $G/B$, and each closed $K$-orbit $\mathcal{Q}$ is isomorphic to the flag variety of $K$. The image of $T_{\mathcal{Q}}^*(G/B)$ under the restriction of the Springer resolution $\mu_\mathcal{Q}$ to $T_{\mathcal{Q}}^*(G/B)$ is known to be a single $K$-orbit $K\cdot N$ in $\mathcal{N}$ for some nilpotent matrix $N$, and moreover $\mu_{\mathcal{Q}}^{-1}(N)$ is a single component of the Springer fiber $\mu^{-1}(N)$.  We call such components \boldit{$K$-components}.

The closed $K$-orbits in $G/B$ are in bijection with the $p$-element subsets $I\subseteq [n]$. We recall from \cite[Appendix A]{graham_zierau11} how, starting from such a subset $I$, we a construct a standard tableau $\sigma(I)$ indexing a $K$-component $\SF{\sigma(I)} \subseteq \SF{\lambda}$ (where $\lambda$ depends on $I$). We $2$-color the elements of $[n]$, where the elements of $I$ are black and the elements of $[n]\setminus I$ are white. The first row of $\sigma(I)^\vee$ consists of $1$ and all $i\in [n]$ such that $i-1$ is colored differently than $i$. The subsequent rows of $\sigma(I)^\vee$ are obtained by deleting the entries in its first row from $[n]$ and repeating the procedure above. That is, the second row of $\sigma(I)^\vee$ consists of the smallest element of $[n]$ not in the first row along with all $i\in [n]$ not in the first row such that the next smallest element of $[n]$ not in the first row is colored differently than $i$, etc. (The reason this procedure defines $\sigma(I)^\vee$ rather than $\sigma(I)$ is because the conventions of \cite{graham_zierau11} differ from ours, in that they index components of $\SF{N}$ by considering the action of $N$ on subspaces $F_j$ rather than quotients $\C^n/F_{n-j}$; cf.\ \cref{remark_steinberg}. By \cref{dual_transpose} and \cref{dual_evacuation}, translating between the two conventions amounts to applying evacuation.)

By \cite[Lemma A.7]{graham_zierau11} we have $\sigma(I) = \sigma(J)$ if and only if $I=J$ or $I=[n]\setminus J$. In particular, for a fixed $n\geq 1$, there are exactly $2^{n-1}$ $K$-components.
\begin{eg}\label{eg_GZ}
Let $n = 7$ and $I = \{3,4,6,7\}$. Then we calculate (see \cref{figure_eg_GZ}) that
\[
\sigma(I)^\vee = \;\begin{ytableau}
1 & 3 & 5 & 6 \\
2 & 4 \\
7
\end{ytableau}\; \quad \text{ and } \quad \sigma(I) = \;\begin{ytableau}
1 & 3 & 4 & 6 \\
2 & 7 \\
5
\end{ytableau}\;,
\]
and the corresponding $K$-component is $\SF{\sigma(I)}$.
\end{eg}
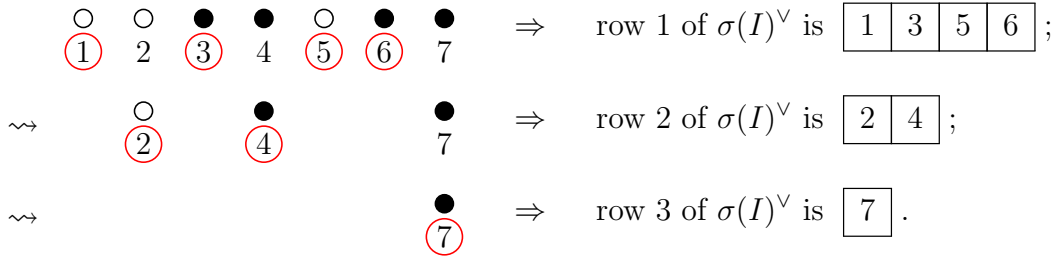
\begin{figure}[ht]
\begin{center}
\begin{align*}
&\begin{tikzpicture}[baseline=(current bounding box.center)]
\tikzstyle{out1}=[inner sep=0,minimum size=2.4mm,circle,draw=black,fill=black,semithick]
\tikzstyle{in1}=[inner sep=0,minimum size=2.4mm,circle,draw=black,fill=white,semithick]
\pgfmathsetmacro{\r}{0.24};
\pgfmathsetmacro{\hs}{0.80};
\pgfmathsetmacro{\vs}{0.44};
\useasboundingbox(0,0)rectangle(7.5*\hs,-\vs);
\foreach \x in {3,4,6,7}{
\node[out1]at(\x*\hs,0){};
\node at(\x*\hs,-\vs){$\x$};}
\foreach \x in {1,2,5}{
\node[in1]at(\x*\hs,0){};
\node at(\x*\hs,-\vs){$\x$};}
\foreach \x in {1,3,5,6}{
\draw[semithick,color=red](\x*\hs,-\vs)circle[radius=\r];}
\end{tikzpicture}
\quad \Rightarrow \quad
\text{ row $1$ of $\sigma(I)^\vee$ is }
\;\begin{ytableau}
1 & 3 & 5 & 6
\end{ytableau}\;;
\\[12pt]
&\begin{tikzpicture}[baseline=(current bounding box.center)]
\tikzstyle{out1}=[inner sep=0,minimum size=2.4mm,circle,draw=black,fill=black,semithick]
\tikzstyle{in1}=[inner sep=0,minimum size=2.4mm,circle,draw=black,fill=white,semithick]
\pgfmathsetmacro{\r}{0.24};
\pgfmathsetmacro{\hs}{0.80};
\pgfmathsetmacro{\vs}{0.44};
\useasboundingbox(0,0)rectangle(7.5*\hs,-\vs);
\foreach \x in {4,7}{
\node[out1]at(\x*\hs,0){};
\node at(\x*\hs,-\vs){$\x$};}
\foreach \x in {2}{
\node[in1]at(\x*\hs,0){};
\node at(\x*\hs,-\vs){$\x$};}
\foreach \x in {2,4}{
\draw[semithick,color=red](\x*\hs,-\vs)circle[radius=\r];}
\node[inner sep=0]at(0,-\vs/2){$\rightsquigarrow$};
\end{tikzpicture}
\quad \Rightarrow \quad
\text{ row $2$ of $\sigma(I)^\vee$ is }
\;\begin{ytableau}
2 & 4
\end{ytableau}\;;
\\[12pt]
&\begin{tikzpicture}[baseline=(current bounding box.center)]
\tikzstyle{out1}=[inner sep=0,minimum size=2.4mm,circle,draw=black,fill=black,semithick]
\tikzstyle{in1}=[inner sep=0,minimum size=2.4mm,circle,draw=black,fill=white,semithick]
\pgfmathsetmacro{\r}{0.24};
\pgfmathsetmacro{\hs}{0.80};
\pgfmathsetmacro{\vs}{0.44};
\useasboundingbox(0,0)rectangle(7.5*\hs,-\vs);
\foreach \x in {7}{
\node[out1]at(\x*\hs,0){};
\node at(\x*\hs,-\vs){$\x$};}
\foreach \x in {}{
\node[in1]at(\x*\hs,0){};
\node at(\x*\hs,-\vs){$\x$};}
\foreach \x in {7}{
\draw[semithick,color=red](\x*\hs,-\vs)circle[radius=\r];}
\node[inner sep=0]at(0,-\vs/2){$\rightsquigarrow$};
\end{tikzpicture}
\quad \Rightarrow \quad
\text{ row $3$ of $\sigma(I)^\vee$ is }
\;\begin{ytableau}
7
\end{ytableau}\;.
\end{align*}
\caption{Calculating $\sigma(I)^\vee$ from $I = \{3,4,6,7\}$, when $n=7$; see \cref{eg_GZ}.}
\label{figure_eg_GZ}
\end{center}
\end{figure}

Every tableau $\sigma(I)^\vee$ is Richardson, which we can verify using \cref{crop_tableau}. Hence by \cref{evacuation_closed,main}, every $K$-component is a Richardson variety. Conversely, there exist irreducible components $\SF{\sigma}$ equal to Richardson varieties which are not $K$-components. For example, there are $2^3 = 8$ $K$-components when $n=4$, but there are $9$ Richardson tableaux of size $4$ (see \cref{eg.4}). The reader can verify that the unique Richardson tableau of size $4$ which does not index a $K$-component is \eqref{nonexample}.

Graham and Zierau prove that each component $\SF{\sigma(I)}$ is invariant under the action of the maximal torus $T\subseteq G$ of diagonal matrices and is an iterated bundle~\cite[Section 2]{graham_zierau11}. They show that $\SF{\sigma(I)}$ is smooth, describe its set of $T$-fixed points, and compute the weights of $T$ acting on the tangent space of $\SF{\sigma(I)}$ at each fixed point. They use these results to compute the equivariant and $K$-theory classes of $\SF{\sigma(I)}$ by applying computational techniques of nil-Hecke algebras (cf.\ \cite[Chapter 11]{kumar02} and \cite{richmond_zainoulline}).  We note that the resulting formulas are not `combinatorially positive' since they may include signs.

\cref{main,thm.smooth} yield generalizations of the key geometric facts needed to apply localization techniques (e.g.~\cite[Proposition 4.5]{graham_zierau11}) to the components $\SF{\sigma}$ for Richardson tableaux $\sigma$. In particular, each $\SF{\sigma}$ is a smooth Richardson variety. Every Richardson variety $\Rvar{v}{w}$ is $T$-invariant, and its set of $T$-fixed points can be identified with the Bruhat interval $[v,w]=\{x\in S_n\mid v\leq x \leq w \}$. The weights of $T$ acting on the tangent space $T_{\dot x B} (\Rvar{v}{w})$ are uniquely determined by the reflections $t_{i,j}$ such that $v\leq xt_{i,j} \leq w$, i.e., the edges of the \emph{Bruhat graph} of $[v,w]$ incident to $x$. (This follows from a combinatorial description of the tangent space of a Schubert variety \cite[Theorem 5.4.2]{billey_lakshmibai00}, using the definition \eqref{richardson_defn} of $\Rvar{v}{w}$.) In order to compute $[\SF{\sigma}]$ via localization, we need a concrete description of such weights:
\begin{prob}\label{prob_equiv_classes}
Let $\sigma$ be a Richardson tableau. Use localization techniques in equivariant cohomology and $K$-theory to compute the equivariant and $K$-theory classes of $\SF{\sigma}$. In particular, given $x\in [v,w]$, is there combinatorial description \textup{(}in terms of $\sigma$\textup{)} of the set of reflections $t_{i,j}$ such that $v_\sigma \leq xt_{i,j} \leq w_\sigma$?
\end{prob}


\bibliographystyle{alpha}
\bibliography{ref}

\end{document}